\documentclass{article}
\usepackage{latexsym}
\usepackage{amsmath}
\usepackage{amsthm}
\usepackage{amscd}
\usepackage{amsfonts}
\usepackage{amssymb}
\usepackage{dsfont}
\usepackage{bm}
\usepackage{bbm}
\usepackage{float}
\usepackage{graphicx}
\usepackage{color}
\usepackage[top = 2.5cm , bottom = 4.0 cm, right = 3.08cm , left = 3.08cm]{geometry}
\numberwithin{equation}{section}

\theoremstyle{plain}   
\theoremstyle{plain}   \newtheorem{Lem}{Lemma}
\theoremstyle{plain} 	\newtheorem{Cor}{Corollary}
\theoremstyle{plain} 	\newtheorem{The}{Theorem}
\theoremstyle{plain} 	\newtheorem{Prop}{Proposition}
\theoremstyle{plain} 	
\theoremstyle{plain}	
\theoremstyle{plain}	
\theoremstyle{plain}	 
\theoremstyle{plain}	
\theoremstyle{plain}   
\theoremstyle{plain} 
\theoremstyle{plain} 

\def\ceil#1{\left\lceil#1\right\rceil} 
\def\floor#1{\left\lfloor#1\right\rfloor} 
\def\implies{\Longrightarrow}
\def\iff{\Longleftrightarrow}
\def\indicator{\mathds{1}}

\def\Var{\mbox{Var}}

\def\Lambdah{\widehat{\Lambda}}
\def\lambdah{\widehat{\lambda}}
\def\lambdat{\widetilde{\lambda}}
\def\omegat{\widetilde{\omega}}

\def\Thetat{\widetilde{\Theta}}

\def\xib{\overline{\xi}}

\def\E{\mathbb{E}}

\def\N{\mathbb{N}}

\def\P{\mathbb{P}}

\def\Z{\mathbb{Z}}

\def\AM{\mathcal{A}}
\def\BM{\mathcal{B}}

\def\ZM{\mathcal{Z}}

\def\Jt{\widetilde{J}}

\def\Pt{\widetilde{P}}

\def\Tt{\widetilde{T}}

\def\Wt{\widetilde{W}}
\def\Xt{\widetilde{X}}
\def\Yt{\widetilde{Y}}
\def\Zt{\widetilde{Z}}

\def\nt{\widetilde{n}}

\def\xt{\widetilde{x}}

\def\Ah{\widehat{A}}

\def\Mh{\widehat{M}}

\def\Yh{\widehat{Y}}

\def\Gammab{\overline{\Gamma}}

\def\pih{\widehat{\pi}}

\def\zetat{\widetilde{\zeta}}

\makeatletter
\renewcommand*{\@fnsymbol}[1]{\ensuremath{\ifcase#1\or * \or ** \else\@ctrerr\fi}}
\makeatother

\begin{document}

\title{Transience, Recurrence and the Speed of a Random Walk in a Site-Based Feedback Environment}

\author{Ross G. Pinsky\thanks{Department of Mathematics, Technion--Israel Institute of Technology.  E-mail - \texttt{pinsky@math.technion.ac.il}. }
~~and~ Nicholas F. Travers\thanks{Department of Mathematics, Technion--Israel Institute of Technology.  E-mail - \texttt{travers@tx.technion.ac.il}. }}

\date{}
\maketitle

\vspace{-5 mm}
\begin{abstract}
We study a random walk on $\Z$ which evolves in a dynamic environment determined by its own trajectory.
Sites flip back and forth between two modes, $p$ and $q$. $R$ consecutive right jumps from a site in the $q$-mode
are required to switch it to the $p$-mode, and $L$ consecutive left jumps from a site in the $p$-mode are required to
switch it to the $q$-mode. From a site in the $p$-mode the walk jumps right with probability $p$ and left with
probability $1-p$, while from a site in the $q$-mode these probabilities are $q$ and $1-q$.

We prove a sharp cutoff for right/left transience of the random walk in terms of an explicit function of the parameters $\alpha = \alpha(p,q,R,L)$. 
For $\alpha > 1/2$ the walk is transient to $+\infty$ for any initial environment, whereas for $\alpha < 1/2$ the walk is transient to $-\infty$ 
for any initial environment. In the critical case, $\alpha = 1/2$, the situation is more complicated and the behavior of the walk depends 
on the initial environment. Nevertheless, we are able to give a characterization of transience/recurrence in many instances, including when 
either $R=1$ or $L=1$ and when $R=L=2$. In the noncritical case, we also show that the walk has positive speed, and in some situations 
are able to give an explicit formula for this speed. 
\end{abstract}

\section{Introduction and Statement of Results}
\label{sec:Intro}
In this paper we introduce a process we call a site-based feedback random walk on $\mathbb{Z}$.
The process $(X_n)_{n \geq 0}$ is a nearest neighbor random walk governed by four parameters:
$p,q\in(0,1)$ and $R,L\in\mathbb{N}$. An informal description is as follows.

Initially each site $x \in \Z$ is set to either the $p$-mode or the $q$-mode. From a site in the $p$-mode the walk
jumps right with probability $p$ and left with probability $1-p$, whereas from a site in the $q$-mode these
probabilities are $q$ and $1-q$, respectively. A site $x$ switches from the $q$-mode to the $p$-mode after the
walk jumps right from $x$ on $R$ consecutive visits to $x$, and a site $x$ switches from the $p$-mode
to the $q$-mode after the walk jumps left from $x$ on $L$ consecutive visits to $x$.

In light of this description, we say the random walk $(X_n)$ has \emph{positive feedback} if
$q < p$ and \emph{negative feedback} if $q > p$. Of course, if $q = p$ the situation is trivial;
we just have a simple random walk of bias $p$.

We now give the formal description and set some notation.

\begin{itemize}
\item $\Lambda = \{(p,0),...,(p,L-1),(q,0),...,(q,R-1)\}$ is the set of \emph {single site configurations}.
A typical configuration is denoted by $\lambda = (r,i)$, where $r \in \{p,q\}$ is the \emph{mode} and $i$
is the number of \emph{charges} in favor of the alternative mode.
\item $\Lambda_p =  \{(p,0),...,(p,L-1)\}$ is the set of $p$-configurations, and \\
$\Lambda_q =  \{(q,0),...,(q,R-1)\}$ is the set of $q$-configurations.
\item $\omega = \{\omega(x)\}_{x \in \Z} \in \Lambda^{\Z}$ is the \emph{initial environment}.
$\omega_n$ is the (random) environment at time $n \geq 0$, $\omega_0 = \omega$.
\item At each step the walk $(X_n)$ jumps right or left according to the following rules:
\begin{align*}
\mbox{ If } \omega_n(X_n) \in \Lambda_p, ~\mbox{ then }
\left\{ \begin{array}{l}
\P(X_{n+1} = X_n + 1) = p, \\
\P(X_{n+1} = X_n - 1) = 1- p.
\end{array} \right.
\end{align*}
\begin{align*}
\mbox{ If } \omega_n(X_n) \in \Lambda_q,  ~\mbox{ then }
\left\{ \begin{array}{l}
\P(X_{n+1} = X_n + 1) = q, \\
\P(X_{n+1} = X_n - 1) = 1- q.
\end{array} \right.
\end{align*}
\item
\noindent
For all $x \not= X_n$, $\omega_{n+1}(x) = \omega_n(x)$. The configuration at the current
position of the walk $X_n$ is updated as follows, depending on the direction of the next jump:
\begin{align*}
& \mbox{ If } \omega_n(X_n) \in \Lambda_p \cup \{(q, R -1)\} \mbox{ and } X_{n+1} = X_n + 1, \mbox{ then } \omega_{n+1}(X_n) = (p,0). \\
& \mbox{ If } \omega_n(X_n) = (q,i), 0 \leq i \leq R-2, \mbox{ and } X_{n+1} = X_n + 1, \mbox{ then } \omega_{n+1}(X_n) = (q,i+1). \\
& \mbox{ If } \omega_n(X_n) \in \Lambda_q \cup \{(p, L -1)\} \mbox{ and } X_{n+1} = X_n - 1, \mbox{ then } \omega_{n+1}(X_n) = (q,0). \\
& \mbox{ If } \omega_n(X_n) = (p,i), 0 \leq i \leq L-2, \mbox{ and } X_{n+1} = X_n - 1, \mbox{ then } \omega_{n+1}(X_n) = (p,i+1).
\end{align*}

\end{itemize}

This site-based feedback random walk is motivated by so-called cookie random walks and shares
certain fundamental properties of two outgrowths of the basic cookie random walk.  A basic cookie random
walk on $\mathbb{Z}$ is defined as follows. Let $M\ge1$ be a positive integer. At each site $x \in \Z$, place
a pile of $M$ ``cookies'' with values $\omega(x,k)\in[0,1]$, $k=1,\ldots, M$. For $k\le M$, the $k$-th time the
process reaches site $x$, it eats the $k$-th cookie at that site, whose value is $\omega(x,k)$, and this
empowers the process to jump to the right with probability $\omega(x,k)$ and to the left with
probability $1-\omega(x,k)$. After the site $x$ has been visited $M$ times, whenever the process visits
that site again, it behaves like an ordinary simple, symmetric random walk, jumping left or
right with equal probability. Cookie random walks were first introduced by Benjamini and Wilson \cite{Benjamini2003}; 
see the survey paper of Kosygina and Zerner \cite{Kosygina2013} for more on cookie random walks and a bibliography.

We now describe two outgrowths of the basic cookie random walk. Kozma, Orenshtein, and Shinkar
\cite{Kozma2013} recently considered a \emph{periodic cookie} random walk. Instead of having a
cookie only the first $M$ times the process is at a given site, consider periodic cookies with period $M$,
and assume that these cookies are identical at each $x \in \Z$. Thus, one defines $\omega(k)$,
$k\in\mathbb{N}$, with $\omega(k+M)=\omega(k)$. For each $x \in \Z$, the $k$th time the process
is at $x$ it jumps right or left with probabilities $\omega(k)$ and $1-\omega(k)$ respectively.
In particular, the process never reverts to a simple, symmetric random walk at any site. Another
outgrowth of the basic cookie random walk is the ``have your cookie and eat it'' random walk
\cite{Pinsky2010}. Now there is only one cookie at each site; call it $\omega(x), x \in \Z$, and
assume $\omega(x)>1/2$. When the process first reaches $x$, it jumps right with probability
$\omega(x)$ and left with probability $1-\omega(x)$. For each site $x$, as long as the process
continues to jump to the right from $x$, it continues to use this right-biased cookie; but after the
first time the process jumps to the left from $x$, the cookie at $x$ is removed. From then on,
whenever the process is at $x$, it behaves like a simple, symmetric random walk,
jumping left or right with equal probability.

The site-based feedback random walk has something in common with each of the two above
processes. In particular, the sequence of configurations encountered on repeated visits to 
a given site in the site-based feedback case is a finite-state Markov chain, and thus ``roughly periodic''
on long time scales, while the jump mechanism at a given site in the site-based feedback case
depends not only on the number of visits to that site but also on the direction of the jumps
on these visits, as in the ``have your cookie and eat it" random walk. However, the site-based feedback random 
walk is also fundamentally different from both of the above processes in that it itself has \emph{both} 
of these properties, and also in that it has persistent interactions with its environment, whereas in the 
``have your cookie and eat it'' case the interactions at a given site $x$ terminate after 
the first leftward jump. 

In this paper we study the transience/recurrence properties of the site-based feedback
random walk, and in the transient case we study the speed of the process. Some new 
features occur that were not present in other cookie random walk models.
In particular, the initial environment can have a dramatic influence on 
the behavior for certain critical values of the parameters $p,q,R,L$. 

Before stating the results, we need to introduce a bit more notation and terminology.
Let $\P_{\omega,k}$ denote the probability measure for the random walk started at $X_0 = k$
in the initial environment $\omega$, and let $\P_{\omega} = \P_{\omega, 0}$. Also, let $\E_{\omega}$ and
$\E_{\omega,k}$ denote, respectively, expectations with respect to the measures $\P_{\omega}$ and
$\P_{\omega,k}$. Finally, for $x \in \Z$, let $N_x$ be the total number of visits to site $x$:
\begin{align}
\label{eq:DefNx}
N_x = |\{n \geq 0 : X_n = x\}|.
\end{align}
We say that the random walk path $(X_n)$ is \footnotemark{}:
\begin{itemize}
\item \emph{recurrent} if $N_0 = \infty$.
\item \emph{right transient}, or \emph{transient to $+\infty$}, if $\lim_{n \to \infty} X_n = +\infty$,
and \emph{left transient}, or \emph{transient to $-\infty$}, if $\lim_{n \to \infty} X_n = -\infty$.
\item \emph{ballistic} if $\liminf_{n \to \infty} |X_n|/n > 0$.
\end{itemize}
\footnotetext{Note that these definitions do not have any a.s. qualifications, and are simply statements about the
(random) path $(X_n) = (X_0, X_1,...)$. Thus, the random walk $(X_n)$ has some probability of being right transient,
some probability of being left transient, and some probability of being recurrent. Typically one says that a random walk 
$(X_n)$ is recurrent/right transient/left transient if, according to our definitions, it is a.s. recurrent/right transient/left transient. 
However, for our model there are some situations (see Theorem \ref{LR2crit}) where there is positive probability 
both for transience to $+\infty$ and for transience to $-\infty$, so for consistency we will speak of all of these properties probabilistically.}

Our first theorem gives the cutoff point for left/right transience.
\begin{The}
\label{thm:RightLeftTransienceCutoff}
Define $\alpha = \alpha(p,q,R,L) \in (0,1)$ by
\begin{align}
\label{eq:DefAlpha}
\alpha =  \frac{p \cdot \big[(1-q)q^R(1 - (1-p)^L)\big] ~+~ q \cdot \big[p(1-p)^L(1-q^R)\big]}
{\big[(1-q)q^R(1 - (1-p)^L)\big] ~+~ \big[p(1-p)^L(1-q^R)\big]}.
\end{align}
\begin{itemize}
\item If $\alpha > 1/2$ then the random walk $(X_n)$ is $\P_{\omega}$ a.s. right transient, for any initial environment $\omega$.
\item If $\alpha < 1/2$ then the random walk $(X_n)$ is $\P_{\omega}$ a.s. left transient, for any initial environment $\omega$.
 \end{itemize}
\end{The}

We will call the vector $(p,q,R,L)$ the \emph{parameter quadruple} for the random walk $(X_n)$.
In light of Theorem \ref{thm:RightLeftTransienceCutoff}, we say that the parameter quadruple
$(p,q,R,L)$ is \emph{critical} if $\alpha(p,q,R,L) = 1/2$, and \emph{noncritical} otherwise.
Our next theorem shows that in the noncritical case, the random walk is not just transient,
but in fact ballistic.

\begin{The}
\label{thm:BallisticityWhenNonCritical}
If $\alpha(p,q,R,L) \neq 1/2$, then there exists a $\beta = \beta(p,q,R,L) > 0$ such that,
for any initial environment $\omega$,
\begin{align}
\label{eq:LiminfXnnGreaterDelta}
\liminf_{n \to \infty} \frac{|X_n|}{n} \geq \beta, ~ \P_{\omega} \mbox{ a.s. }
\end{align}
Moreover, if $\alpha>1/2$ ($\alpha<1/2$) and the initial environment $\omega(x)$ is constant for
$x \geq m$ ($x \leq - m$) then $\E_\omega(N_x)$ is also constant for $x \geq m$ ($x \leq - m$),
and denoting this common value by $\gamma$,
\begin{align}
\label{eq:SpeedENxInverse}
 \lim_{n \to \infty} \frac{|X_n|}{n} =
\frac1\gamma,~ \P_{\omega} \mbox{ a.s.}
\end{align}
Here, $m$ can be any nonnegative integer.
\end{The}

The following proposition  characterizes some properties of the fundamental function $\alpha$.
We choose to analyze $\alpha$ as a function of $p$ for fixed $R,L,q$; of course, a similar analysis also
works to analyze $\alpha$ as a function of $q$ for fixed $R,L,p$.

\begin{Prop}
\label{prop:PropertiesOfAlpha}
Let $R,L,q$ be fixed and consider $\alpha$ as a function of $p$, $\alpha(p) \equiv \alpha(p,q,R,L)$.
\begin{itemize}

\item[(i)] If $q = 1/2$, then
\begin{equation*}
\begin{aligned}
\label{eq:alphaHalf}
\alpha(1/2) = 1/2 ~,~ \alpha(p) < 1/2 \mbox{ for } p < 1/2 ~,~ \alpha(p) > 1/2 \mbox{ for } p > 1/2.
\end{aligned}
\end{equation*}

\item[(ii)] If $q < 1/2$, then there exists a unique critical point $p_0 = p_0(q,R,L) \in (1/2,1)$ such that
\begin{equation}
\begin{aligned}
\label{eq:alphap0}
\alpha(p_0) = 1/2 ~,~ \alpha(p) < 1/2 \mbox{ for } p < p_0 ~,~ \alpha(p) > 1/2 \mbox{ for } p > p_0.
\end{aligned}
\end{equation}

\item[(iii)] For $q < 1/2$ the critical point $p_0$ from (ii) satisfies
\begin{align}
\label{eq:qPlusp0}
&q+p_0(q,R,L)<1, \ \text{if}\ R<L; \nonumber \\
&q+p_0(q,R,L)>1, \ \text{if}\ R>L.
\end{align}
Also, for any fixed $R$ and $L$, $p_0(q,R,L)$ is a decreasing function of $q$, for $q \in (0,1/2)$.

\item[(iv)] If $q < 1/2$ and $L=1$, then
\begin{equation}\label{p_0}
p_0 = \frac{1-2q+q^{R+1}}{1-2q+q^R}.
\end{equation}
If $q>1/2$  and $L=1$, then (\ref{eq:alphap0}) still
holds with $p_0 =  \frac{1-2q+q^{R+1}}{1-2q+q^R}$ as long as $1 - 2q + q^{R+1} > 0$.
However, if $1 - 2q + q^{R+1} \leq 0$, then $\alpha(p) > 1/2$, for all $p \in (0,1)$.

\item[(v)] If $q < 1/2$ and $L=R$, then $p_0 =1 - q.$ If $q > 1/2$
and $L=R$, then $1-q$ is still a critical point (i.e. $\alpha(1-q) = 1/2$),
but it is not always unique.

\item[(vi)] For any $R,L,q$, $\lim_{p \to 1} \alpha(p) = 1$. In particular, $\alpha > 1/2$
for all sufficiently large $p$.

\end{itemize}
\end{Prop}

\noindent \bf Remark 1.\rm\
Part (v) shows that $\alpha(p)$ is not always a monotonic function of $p$, and, in fact,
often it is not. Consequently, increasing $p$ (with fixed $q$, $R$, $L$) may sometimes
change the process from the right transient regime to the left transient regime. However, this phenomena
can only occur when $q > 1/2$, by part (ii), in which case the process has negative feedback at all critical points.
Illustrative plots are given in Figure \ref{fig:AlphaPlot}.
\medskip

\noindent \bf Remark 2.\rm\
As noted before the proposition, we could have considered $\alpha$ as a function of $q$ for fixed
$p,R,L$. We note, in particular, that  in the case that $p>1/2$, there exists  a unique critical point
$q_0=q_0(p,R,L) \in (0,1/2)$, and when in addition, $R=1$, one has
\begin{equation}\label{q_0}
q_0=\frac{p(1-p)^L}{2p-1+(1-p)^L}.
\end{equation}
Moreover, if $R = 1$ and $p \leq 1/2$, then there is still a unique critical point $q_0$ given by \eqref{q_0}
as long as $2p-1 + (1-p)^L > 0$. However, if $2p-1+(1-p)^L \leq 0$, then $\alpha < 1/2$, for all $q \in (0,1)$.
\bigskip

\begin{figure}[h]
\label{fig:AlphaPlot}
\hspace{-12 mm}
\includegraphics[scale=0.45]{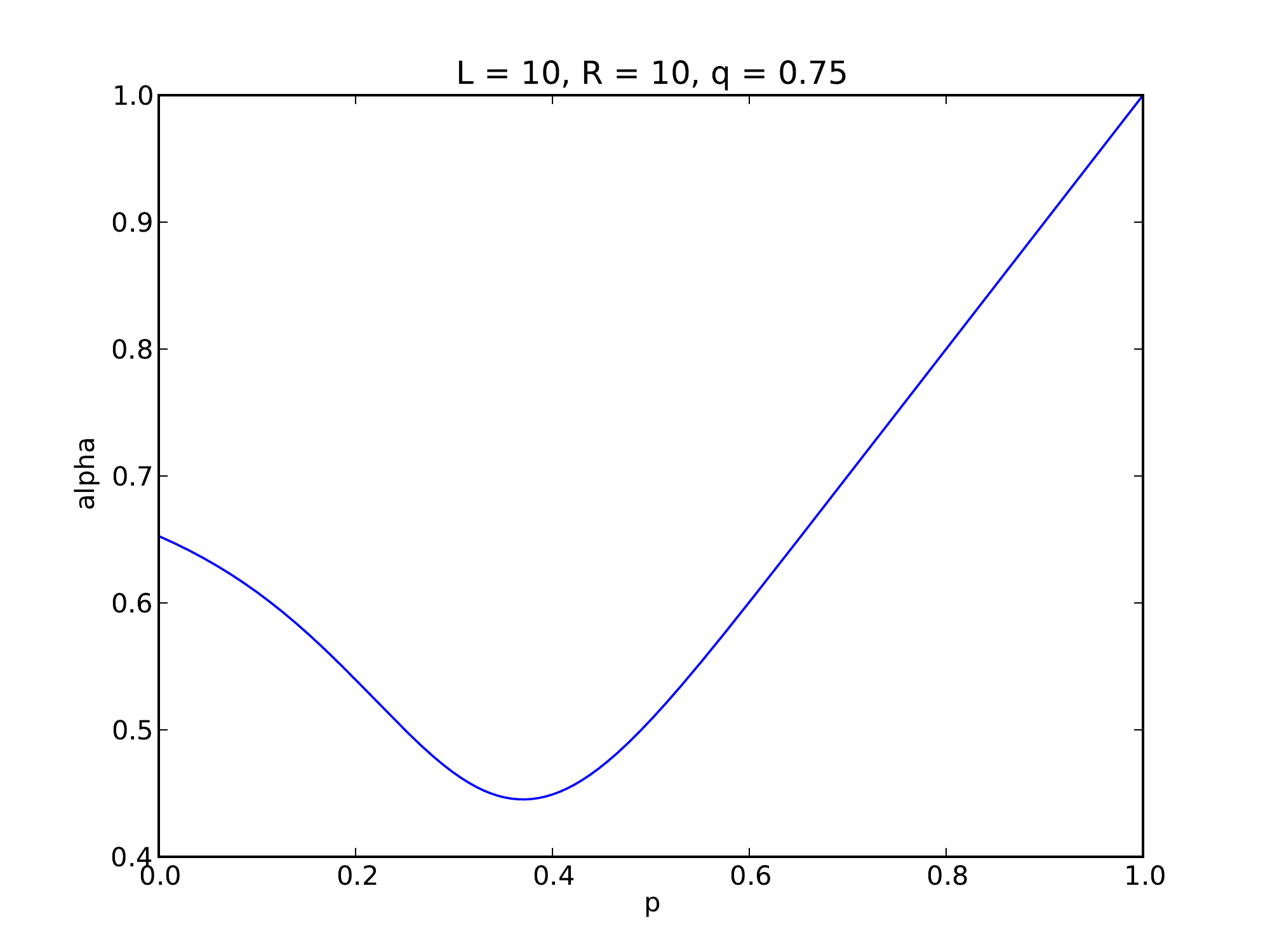}
\hspace{-10 mm} 
\includegraphics[scale=0.45]{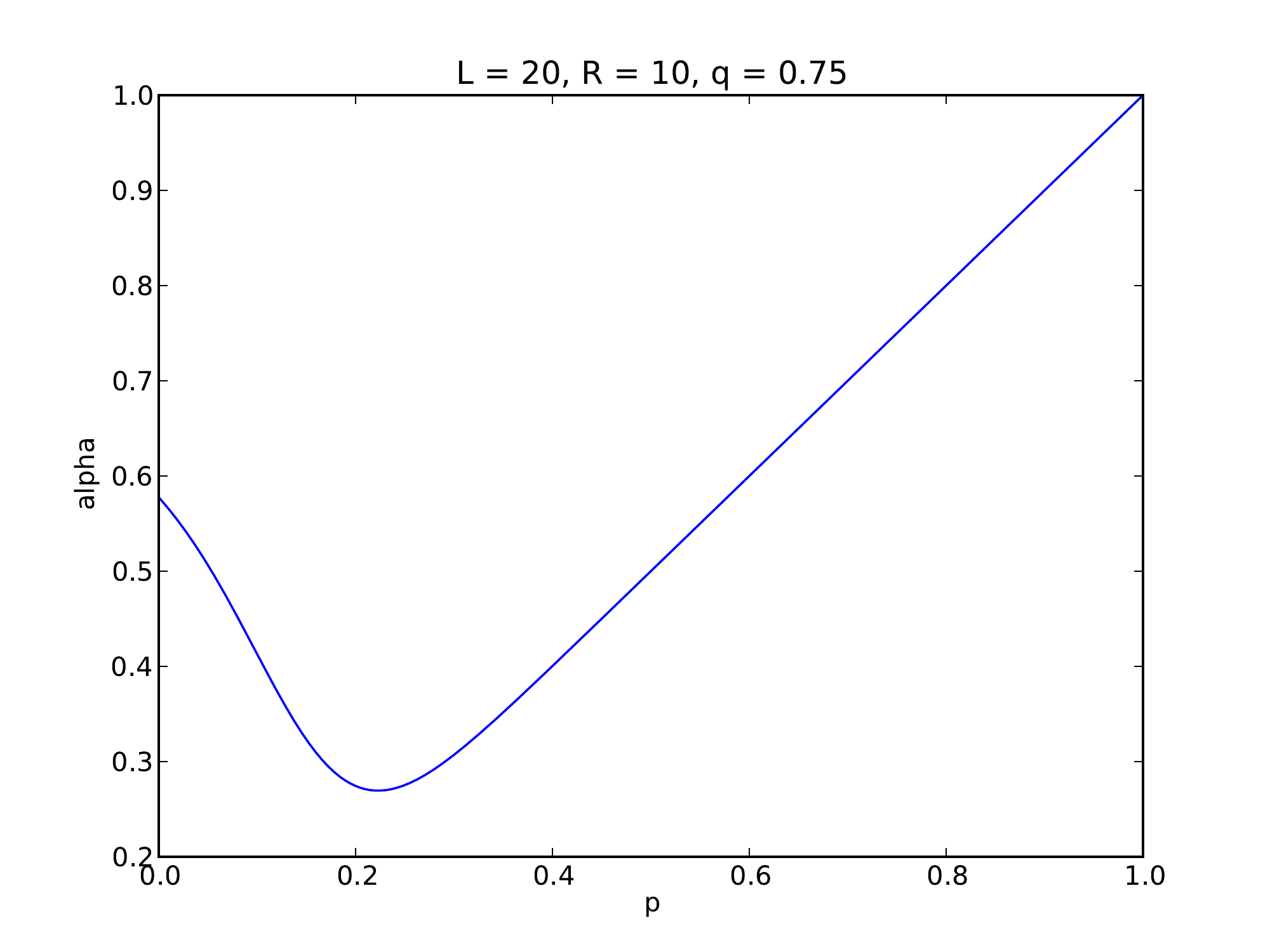}
\caption{Plots of $\alpha(p)$ with $L = 10, R = 10, q = 0.75$ (left) and $L = 20, R = 10, q = 0.75$ (right). 
In both cases, as $p$ increases from $0$ to $1$ the parameter quadruple $(p,q,R,L)$ passes from right transient 
($\alpha > 1/2$), to left transient ($\alpha < 1/2$), and back to right transient.}
\end{figure}

In general for cookie-type random walks, it is very difficult to obtain an explicit
formula for the speed in the ballistic regime. However, the additional level of interaction
between the random walker and the environment in the site-based feedback case
makes a calculation of the speed possible in some situations. Before moving on to the critical case,
we present two results that give an exact characterization of the limiting speed with $R$ or $L$
equal to $1$, in certain initial environments.  We assume that $\alpha > 1/2$, but analogous results
for $\alpha < 1/2$ are easily inferred by symmetry considerations. Specifically, if $\alpha(p,q,R,L) < 1/2$
then $\alpha(1-q, 1-p, L, R) > 1/2$, and the speed to $-\infty$ with parameters $p,q,R,L$ in an initial
environment $\omega$ is the same as the speed to $+\infty$ with parameters $1-q,1-p,L,R$
in an initial environment $\omega'$ defined by $\omega'(x) = \omega(-x)^*, x \in \Z$,
where $(q,i)^* = (1-q,i)$ and $(p,i)^* = (1-p,i)$.

\begin{The}
\label{thm:L1Speed}
Let $L = 1$ and $\alpha > 1/2$. If $\omega(x) = (q,0)$ in a neighborhood of $+\infty$, then
\begin{align}
\label{eq:L1Speed}
\lim_{n \to \infty} \frac{X_n}{n} = \frac{1 - t_*}{1 + t_*},~ \P_{\omega} \mbox{ a.s.},
\end{align}
where $t_*$ is the unique root of the polynomial
\begin{align}
\label{eq:DefPolynomialP}
P(t) = (1 - q) + (pq - p - 1)t + (p+q)t^2 - pqt^3 - (p-q) q^R (t^R - t^{R+1})
\end{align}
in the interval $(1-q,1)$.
\end{The}

\begin{The}
\label{thm:R1Speed}
Let $R = 1$ and $\alpha > 1/2$. Assume that the limiting right density of $(p,i)$ sites
$d_i \equiv \lim_{n \to \infty} \frac{1}{n} |\{0 \leq x \leq n-1: \omega(x) = (p,i)\}|$
exists, for each $0 \leq i \leq L-1$, and let $d_L = 1 - \sum_{i=0}^{L-1} d_i$
denote the limiting right density of $(q,0)$ sites. Then,
\vspace{-2 mm}
\begin{align}
\label{eq:R1Speed}
\lim_{n \to \infty} \frac{X_n}{n} = \frac{1}{\sum_{i = 0}^L a_i d_i}~,~ \P_{\omega} \mbox{ a.s.,}
\end{align}
where
\vspace{-2 mm}
\begin{align}
\label{eq:Def_a0}
a_0 = \frac{1 + (p/q-1)(1-p)^L}{(2p-1) - (p/q - 1)(1-p)^L}~,
\end{align}
and, for $1 \leq i \leq L$,
\begin{align}
\label{eq:Def_ai}
a_i = \frac{1 + (p/q-1)(1-p)^{L-i}}{p} ~+~ \left( \frac{(1-p) + (p/q-1)(1-p)^{L-i}}{p} \right) a_0.
\end{align}
\end{The}
\medskip

\noindent \bf Remark.\rm\ 
In the case $R = L = 1$ with $\alpha > 1/2$ (i.e. with $p+q > 1$), the speed $s = \lim_{n \to \infty} \frac{X_n}{n}$ 
from \eqref{eq:R1Speed} reduces to 
\begin{align*}
s = \frac{p+q-1}{1 + (1-2d)(p-q)}~,
\end{align*} 
where $d = d_0$ is the limiting right density of $(p,0)$ sites in the initial environment $\omega$. 
Interestingly, this formula for $s$ is invariant under the interchange $p \leftrightarrow q$, $d \leftrightarrow 1-d$. 
That is, the speed of the walk is the same with positive or negative reinforcement, as along as the 
pair of bias parameters and the limiting right density of sites initially in the configuration
with the more favorable bias for jumping right remain the same. \medskip

We now turn to the critical case, $\alpha=1/2$. Here, there are two possibilities:
positive feedback with $q < 1/2 < p$ or negative feedback with $p < 1/2 < q$.
In the case of positive feedback the situation is somewhat simpler, but in both
cases the analysis is more delicate than before, and the transience/recurrence of
the random walk often depends heavily on the initial environment.

Our first result shows that there always exist initial environments for which the random walk is recurrent. 

\begin{The}
\label{recurrpossible}
If $\alpha(p,q,R,L)=1/2$, then there exist initial environments $\omega$ for
which the random walk $(X_n)$ is a.s. recurrent. In particular, in the positive feedback case, 
$q < p$, the random walk $(X_n)$ is $\P_{\omega}$ a.s. recurrent for any initial 
environment $\omega$ with $\omega(x)=(q,0)$ for $x$ in a neighborhood of $+\infty$ 
and $\omega(x)=(p,0)$ for $x$ in a neighborhood of $-\infty$. 
\end{The}

The next two theorems, and concomitant corollary and proposition, concern the situation that either $R$ or $L$ is 1. 
In this case, we can give an essentially complete description of when the random walk is recurrent, right transient, 
or left transient. However, for technical reasons, we will need to assume in many instances that the initial 
environment $\omega$ is constant either in a neighborhood of $+\infty$, a neighborhood of $-\infty$, or both.
Our first result indicates that, when $R$ or $L$ is equal to 1, only one of the two directions is possible for transience.

\begin{The}
\label{RorL1}
Assume $\alpha = 1/2$.
\begin{itemize}
\item If $R = 1$ and $q < p$, then the random walk $(X_n)$ is $\P_{\omega}$ a.s. not transient to $+\infty$,
for any initial environment $\omega$.
\item If $R = 1$ and $p < q$, then the random walk $(X_n)$ is $\P_{\omega}$ a.s. not transient to $+\infty$, for any environment
$\omega$ which is constant in a neighborhood of $+\infty$.
\item If $L = 1$ and $q < p$, then the random walk $(X_n)$ is $\P_{\omega}$ a.s. not transient to $-\infty$,
for any initial environment $\omega$.
\item If $L = 1$ and $p < q$, then the random walk $(X_n)$ is $\P_{\omega}$ a.s. not transient to $-\infty$, for any environment
$\omega$ which is constant in a neighborhood of $-\infty$.
\end{itemize}
\end{The}

The following corollary is an immediate consequence of this theorem and part (ii) of
Lemma \ref{lem:SimpleTransConditions} in section \ref{subsec:BasicLemmas}.

\begin{Cor}
\label{RandL1}
Let $R=L=1$ and  $p=1-q$; so $\alpha=1/2$.
\begin{itemize}
\item If $q < p$ then the random walk $(X_n)$ is $\P_{\omega}$ a.s. recurrent, for any initial environment $\omega$.
\item If $p < q$ then the random walk $(X_n)$ is $\P_{\omega}$ a.s. recurrent, for any initial environment $\omega$
which is constant in neighborhoods of $+ \infty$ and $- \infty$.
\end{itemize}
\end{Cor}

Our next theorem gives specific conditions to determine if the random walk is recurrent or transient to $+\infty$ in
the case $L = 1$ and $R > 1$ (which, by Theorem \ref{RorL1}, are the only possibilities). By symmetry considerations,
if $R=1$, instead of $L=1$, then the result obtained for $L=1$ will hold with the roles of  $q,p,R$ and  $\pm\infty$
replaced by $1-p,1-q,L$ and $\mp\infty$ respectively.

\begin{The}
\label{L1andEnvironment}
Assume that $L=1$, $R \geq 2$, and $\alpha = 1/2$. Thus, by \eqref{p_0}, $p=p_0=\frac{1-2q+q^{R+1}}{1-2q+q^R}$. In the
case of negative feedback, $p < q$, assume also that the initial environment $\omega$ is constant in a neighborhood of $-\infty$.
\begin{itemize}
\item[(i)] If $\omega(x) = (q,0)$ in a neighborhood of $+\infty$, then the random walk
$(X_n)$ is $\P_{\omega}$ a.s. recurrent.
\item[(ii)] If $\omega(x) = (q,i)$ in a neighborhood of $+\infty$, $1 \leq i \leq R-1$, then the random walk
$(X_n)$ is
\begin{align*}
\P_{\omega}& \mbox{ a.s. recurrent if } P_{R,i}(q) \geq 0, \mbox{ and }\\
\P_{\omega}& \mbox{ a.s. transient to $+\infty$ if } P_{R,i}(q) < 0,
\end{align*}
where
\begin{align}
\label{poly}
P_{R,i}(q) & = (2R-1)q^{R+2}-(3R+1)q^{R+1}+(R+1)q^R \nonumber \\ &~~~ -2q^{R+2-i}+3q^{R+1-i}-q^{R-i}+(1-2q)^2.
\end{align}
\item[(iii)] If $\omega(x) = (p,0)$ in a neighborhood of $+\infty$, then the random walk $(X_n)$ is
\begin{align*}
\P_{\omega}& \mbox{ a.s. recurrent if } P_{R,R}(q) \geq 0, \mbox{ and }\\
\P_{\omega}& \mbox{ a.s. transient to $+\infty$ if } P_{R,R}(q) < 0,
\end{align*}
where
\begin{align}
\label{polyR}
P_{R,R}(q)=(2R-1)q^{R+2}-(3R+1)q^{R+1}+(R+1)q^{R}+2q^2-q.
\end{align}
Moreover, for each $R \geq 2$ there exists a unique root $q_*(R)\in(0,1/2)$ of the polynomial $P_{R,R}(q)$, $P_{R,R}(q) > 0$ for
$q > q_*(R)$, $P_{R,R}(q) < 0$ for $q < q_*(R)$, and $\lim_{R\to\infty}q_*(R)=1/2$.
\end{itemize}
\end{The}

\noindent \bf Remark 1. \rm\
In the case of positive feedback, $q < p$, one can also determine between right transience and
recurrence for some environments that are not constant in a neighborhood of $+\infty$,
using the comparison lemma given in section \ref{subsec:ComparisonOfEnvironments}.
In particular, $(p,0)$ is the most favorable environment for right transience in the positive feedback
case, so if the random walk is not right transient with the $(p,0)$ environment in a neighborhood
of $+\infty$, then it is not right transient for any initial environment.
\medskip

\noindent \bf Remark 2. \rm\
$P_{R,R}(q)$ is the same polynomial one obtains by substituting $i = R$ into the definition of $P_{R,i}(q)$
in \eqref{poly}.
\medskip

\noindent \bf Remark 3. \rm\
For any $1 \leq i \leq R$, $P_{R,i}(q)$ has a double root at $1$. $P_{R,R}(q)$ also has
a single root at $0$ and factors as $P_{R,R}(q) = q(1-q)^2 \Pt_{R,R}(q)$ where
\begin{align}
\label{eq:FactoredPolyR}
\Pt_{R,R}(q)=-1+\sum_{j=1}^{R-3}jq^{j+1}+(2R-1)q^{R-1}.
\end{align}
Here, the sum is defined to be 0, for $R = 2,3$.
\medskip

\noindent \bf Remark 4. \rm\
Using \eqref{eq:FactoredPolyR} one finds that $q_*(2) = 1/3$ and $q_*(3) = 1/\sqrt{5} \approx 0.447$. Using
a combination of analytical techniques and computer generated plots one also finds the following behavior
for $P_{R,i}$, $1 \leq i \leq R-1$. For $R=2,3,4$, $P_{R,i}(q) \geq 0$ for all $1 \leq i \leq R-1$ and $q \in (0,1)$.
For $R = 5,6$, $P_{R,i}(q) \geq 0$ for all $1 \leq i \leq R-2$ and $q \in (0,1)$. However, $P_{5,4}(q) < 0$
if (and only if) $q \in (a,b)$, where $a\approx.410$ and $b\approx.473$, and $P_{6,5}(q) < 0$ if
(and only if) $q \in (a,b)$, where $a\approx.391$ and $b\approx.490$. For $i = 5,6$ there
are ranges of $q$ for which $P_{7,i}(q) < 0$.
\medskip

The next proposition characterizes asymptotic properties of the function $P_{R,i}(q)$ in the limit of large $R$,
for two different cases of $i = i_R$. In the first case, $R-i_R$ grows to infinity, so the process must jump right
many consecutive times from a given site to switch it to the $p$-mode, starting in the $(q,i_R)$ initial
environment. In the second case, $i_R = R-k$, for a fixed $k$, so the process need jump right only
$k$ consecutive times from any site to switch it to the $p$-mode, starting in the $(q,i_R)$ initial
environment. The proof of both cases is straightforward, and is left to the reader. 

\begin{Prop} ~\
\label{Rlarge}
\begin{itemize}
\item[(i)] If $(R - i_R) \rightarrow \infty$ as $R \rightarrow \infty$ then, for any fixed $q \in (0,1)$,
$P_{R,i_R}(q) > 0$ for all sufficiently large $R$.
\item[(ii)] Let $i_R = R-k$, and define
\begin{align}
\label{Qpoly}
Q_k(q)= (1-2q)^2 -q^k + 3q^{k+1} - 2q^{k+2}.
\end{align}
If $Q_k(q) > 0$ then $P_{R,i_R}(q) > 0$ for sufficiently large $R$, and if
$Q_k(q) < 0$ then $P_{R,i_R}(q) < 0$ for sufficiently large $R$.
\end{itemize}
\end{Prop}

\noindent \bf Remark.\rm\
The polynomial $Q_k(q)$ factors as $(2q-1)(2q-1+q^k-q^{k+1})$. Using this representation it
is not hard to verify the following facts: $Q_k(q)>0$ for $q>1/2$, and there exists an
$a_k\in(\frac12(3-\sqrt5),1/2)\approx(.382,1/2)$ such that $Q_k(q)>0$ for $q\in(0,a_k)$ and
$Q_k(q)<0$ for $q\in(a_k,1/2)$. Furthermore, $a_k$ is increasing in $k$ and $\lim_{k\to\infty}a_k=1/2$.
\medskip

Corollary \ref{RandL1} showed that if $R=L=1$ then in the critical case $p = 1-q$ the random walk is always recurrent
(assuming constant initial environment in neighborhoods of $\pm \infty$ in the negative feedback regime). 
When $R=L=2$ the behavior in the critical case is much more complicated; in particular,
for appropriate initial environments, there is a positive probability for both transience
to $+\infty$ and transience to $-\infty$.

\begin{The}\label{LR2crit}
Let $R=L=2$ and assume that $p=1-q$; so $\alpha=\frac12$. In the negative feedback case, $q>\frac12$, assume
also that the initial environment $\omega$ is constant in a neighborhood of $+\infty$ and in a neighborhood of $-\infty$.
Let
$$
q^*_1=\frac{\sqrt{13}-3}2\approx0.303
$$
 and let
$$
q^*_2\approx .682\ \text{ be the unique root in $(0,1)$ of } q^3+q-1.
$$
\begin{itemize}
\item[(i)] Let $q<q_1^*$.

a. If $\omega(x) =(p,0)$ in a neighborhood of $+\infty$ and
$\omega(x) \neq(q,0)$ for any $x$ in a neighborhood of $-\infty$, then
the random walk $(X_n)$ is $\mathbb{P}_\omega$ a.s. transient to $+\infty$.

b. If $\omega(x) =(q,0)$ in a neighborhood of $-\infty$ and
$\omega(x) \neq(p,0)$ for any $x$ in a neighborhood of $+\infty$, then
the random walk $(X_n)$ is $\mathbb{P}_\omega$ a.s. transient to $-\infty$.

c. If $\omega(x) \neq(p,0)$ for any $x$ in a neighborhood of $+\infty$ and
$\omega(x) \neq(q,0)$ for any $x$ in a neighborhood of $-\infty$, then
the random walk $(X_n)$ is $\mathbb{P}_\omega$ a.s. recurrent.

d. If $\omega(x) =(p,0)$ in a neighborhood of $+\infty$ and
$\omega(x) = (q,0)$ in a neighborhood of $-\infty$, then
$\mathbb{P}_\omega(X_n\to+\infty) = 1 - \mathbb{P}_\omega(X_n\to-\infty) \in (0,1).$

\item[(ii)] Let $q>q^*_2$. 
Then (a)-(d) of (i)  hold with the roles of $(p,0)$ and $(q,0)$ reversed.

\item[(iii)] Let $q\in[q^*_1,q^*_2]$.
Then the random walk $(X_n)$ is  $\mathbb{P}_\omega$ a.s. recurrent.
\end{itemize}
\end{The}

\noindent \bf Remark.\rm\ 
The proof of Theorem \ref{LR2crit} relies on the eigen-decomposition of a $2 \times 2$ matrix. 
In principal, one could also apply similar methods to characterize the transience/recurrence 
properties for general $R, L \geq 2$. Specifically, to determine if there is positive probability of right transience 
one needs to diagonalize an $L \times L$ matrix, and to determine if there is positive probability of left transience 
one needs to diagonalize an $R \times R$ matrix. For $R = L = 3$ this is possible analytically, but in general it is not.
Nevertheless, it could be done numerically for reasonably sized $R$ and $L$. In the case 
$R \not= L$, one would also need to determine numerically the critical value/s of $p$ such that $\alpha(p,q,R,L) = 1/2$, 
as the entries of these matrices depend on both $p$ and $q$. \medskip

We close this introductory section with an \emph{open problem}. In the non-critical case, Theorem
\ref{thm:BallisticityWhenNonCritical} shows that the random walk is ballistic. The open problem is to 
show that in the critical case, if the random walk is transient, then it is not ballistic; that is,
$\lim_{n\to\infty}\frac{X_n}n=0$ a.s. Note that if this were not always true, then in light of 
Theorem \ref{thm:BallisticityWhenNonCritical}, there would be cases (that is, choices of $R$ and $L$) 
for which the  speed would have a discontinuity as a function of $p$ and $q$. \\

The remainder of the paper is organized as follows. In section \ref{sec:Preliminaries} we introduce
some important constructions that will be central to our proofs and establish a number of simple lemmas. In
section \ref{sec:NoncriticalCase} we prove Theorems \ref{thm:RightLeftTransienceCutoff}--\ref{thm:R1Speed}
concerning the behavior of the random walk $(X_n)$ in the noncritical case. In section
\ref{sec:CriticalCase} we prove Theorems \ref{recurrpossible}--\ref{LR2crit} concerning the behavior
of the random walk in the critical case. Finally, in section \ref{sec:AnalysisOfAlpha} we prove Proposition
\ref{prop:PropertiesOfAlpha} which characterizes properties of the important function $\alpha$.

\section{Preliminaries}
\label{sec:Preliminaries}

In this section we introduce a basic framework for proving the theorems stated above and establish a number
of useful lemmas. Section \ref{subsec:AuxilliaryMarkovChains} gives constructions of the single site Markov chains
$(Y_n^x)_{n \in \N}$ and the right jumps Markov chain $(Z_x)_{x \geq 0}$, which will be the primary tools used in the proofs of
Theorems \ref {thm:RightLeftTransienceCutoff}, \ref{thm:BallisticityWhenNonCritical}, \ref{recurrpossible},  \ref{L1andEnvironment} and \ref{LR2crit}.
Section \ref{subsec:BasicLemmas} gives three simple lemmas that will be used in a number of places. The first two concern conditions
for transience, and the other relates hitting times to speed. Finally, section \ref{subsec:ComparisonOfEnvironments} gives an important
lemma comparing the possibility of transience in different environments. 

\subsection{Auxilliary Markov Chains}
\label{subsec:AuxilliaryMarkovChains}

\subsubsection{The Single Site Markov Chains $(Y_n^x)_{n \in \N}$}
\label{subsubsec:SingleSiteMarkovChain}

Let $M$ be the stochastic transition matrix on the set of single site configurations
$\Lambda$, with nonzero entries defined as follows:
\begin{align*}
& M_{(p,i)(p,0)} = p, M_{(p,i)(p,i+1)} = 1-p~,~ \mbox{ for } 0 \leq i \leq L-2. \\
& M_{(p,L-1)(p,0)} = p, M_{(p,L-1)(q,0)} = 1-p. \\
& M_{(q,i)(q,0)} = 1-q, M_{(q,i)(q,i+1)} = q~,~ \mbox{ for } 0 \leq i \leq R-2. \\
& M_{(q,R-1)(q,0)} = 1-q, M_{(q,R-1)(p,0)} = q.
\end{align*}
For $x \in \Z$, let $(Y_n^x)_{n \in \N}$ be the Markov chain with state space $\Lambda$, transition matrix $M$, and initial
state $\omega(x)$. We refer to the chain $(Y_n^x)_{n \in \N}$ as the \emph{single site Markov chain} at $x$. It is
the Markovian sequence of configurations at site $x$ that would occur if $x$ were to be visited infinitely often. That is,
\begin{align*}
\P(C_{n+1}^x = \lambda'|C_{n}^x = \lambda, N_x \geq n + 1)
= M_{\lambda \lambda'},~ \lambda, \lambda' \in \Lambda
\end{align*}
where $N_x$ is the total number of visits to site $x$, as above, and
$C_n^x$ is the configuration at site $x$ immediately after the $n$-th visit.

The \emph{extended single site chain} at $x$, $(\Yh_n^x)_{n \in \N} = (Y_n^x, J_n^x)_{n \in \N}$, is the Markov chain
whose states are pairs $(\lambda,j)$, where $\lambda \in \Lambda$ denotes the current configuration at site $x$ and
$j \in \{1,-1\}$ represents the next jump from $x$ ($1$ for right, $-1$ for left). The state space of this chain is
$\Lambdah = \Lambda \times \{1,-1\}$ and the transition matrix $\Mh$ is defined by:
\begin{align*}
& \Mh_{((p,i),1)((p,0),1)} = p, \Mh_{((p,i),1)((p,0),-1)} = 1-p ~,~\mbox{ for } 0 \leq i \leq L-1. \\
& \Mh_{((p,i),-1)((p,i+1),1)} = p, \Mh_{((p,i),-1)((p,i+1),-1)} = 1-p ~,~\mbox{ for } 0 \leq i \leq L-2. \\
& \Mh_{((p,L-1),-1)((q,0),1)} = q, \Mh_{((p,L-1),-1)((q,0),-1)} = 1-q. \\
& \Mh_{((q,i),-1)((q,0),1)} = q, \Mh_{((q,i),-1)((q,0),-1)} = 1-q ~,~\mbox{ for } 0 \leq i \leq R-1. \\
& \Mh_{((q,i),1)((q,i+1),1)} = q, \Mh_{((q,i),1)((q,i+1),-1)} = 1-q ~,~\mbox{ for } 0 \leq i \leq R-2. \\
& \Mh_{((q,R-1),1)((p,0),1)} = p, \Mh_{((q,R-1),1)((p,0),-1)} = 1-p.
\end{align*}
The initial state $\Yh_1^x$ for the chain has the following distribution:
\begin{itemize}
\item If $\omega(x) = (p,i)$, for some $0 \leq i \leq L-1$, then
\begin{align}
\label{eq:DistFrom_pSite}
\P\left(\Yh_1^x = ((p,i),1)\right) = p,~ \P\left(\Yh_1^x = ((p,i),-1)\right) = 1-p.
\end{align}
\item If $\omega(x) = (q,i)$, for some $0 \leq i \leq R-1$, then
\begin{align}
\label{eq:DistFrom_qSite}
\P\left(\Yh_1^x = ((q,i),1)\right) = q,~ \P\left(\Yh_1^x = ((q,i),-1)\right) = 1-q.
\end{align}
\end{itemize}
By construction, the sequence of site configurations $(Y_n^x)$ obtained by projection from
this extended Markov chain state sequence $(\Yh_n^x)$ with transition matrix $\Mh$ and
initial state distributed according to (\ref{eq:DistFrom_pSite}) and (\ref{eq:DistFrom_qSite})
has the same law as above, when defined directly by the transition matrix $M$
with initial state $\omega(x)$. \\

\noindent
\emph{Coupling to the Random Walk $(X_n)$} \\

For a given initial position $x_0$ and initial environment $\omega = \{\omega(x)\}_{x \in \Z}$ one can
construct the random walk $(X_n)_{n \geq 0}$ according to the following two step procedure, similar
to that given in \cite{Amir2013} for cookie random walks.
\begin{enumerate}
\item Run the extended single site Markov chains $(\Yh_n^x)_{n \in \N}$ at each site $x$ independently.
\item Walk deterministically from the initial point $x_0$ according to the corresponding ``jump pattern''
$\{J_n^x\}_{n \in \N, x \in \Z}$. That is, upon the the $k_{th}$ visit to site $x$, the walk jumps right
if $J_k^x = 1$ and left if $J_k^x = -1$. Formally, we have:
\begin{itemize}
\item $X_0 = x_0$.
\item For $n \geq 0$, $X_{n+1} = X_n + J_{K_n}^{X_n}$ where $K_n = |\{0 \leq m \leq n: X_m = X_n\}|$.
\end{itemize}
\end{enumerate}
By definition of the extended single site chains, the random walk $(X_n)_{n \geq 0}$ constructed by this two step
procedure will have the correct law, and in the sequel we always assume our random walk $(X_n)$ to be
defined in this fashion. We also denote by $\P_{\omega}$ the probability measure for the extended single
site chains, run independently at each site $x$, with initial environment $\omega = \{\omega(x)\}_{x \in \Z}$.
This is a slight abuse of notation since the probability measure $\P_{\omega} \equiv \P_{\omega,0}$ introduced
in section \ref{sec:Intro} also specifies the initial position of the random walk as $X_0 = 0$. However, things
should be clear from the context. \\

\noindent
\emph{Stationary Distribution}\\

Since $\Lambda$ is finite and $M$ is an irreducible transition matrix, there exists a unique stationary probability distribution
$\pi$ on $\Lambda$ satisfying $\pi = \pi M$. Solving the linear system $\{\pi = \pi M, \sum_{\lambda \in \Lambda} \pi_{\lambda} = 1\}$,
one obtains the following explicit form for $\pi$ (see Appendix \ref{subsec:StationaryDistributionSingleSiteMC}):
\begin{align}
\label{eq:pipi_piqi}
\pi_{(p,i)} & = \frac{p(1-q)q^R}{(1-q)q^R(1 - (1-p)^L) + p(1-p)^L(1-q^R)} \cdot (1-p)^i ~,~ 0 \leq i \leq L-1. \nonumber \\
\pi_{(q,i)} & = \frac{p(1-q)(1-p)^L}{(1-q)q^R(1 - (1-p)^L) + p(1-p)^L(1-q^R)} \cdot q^i ~,~ 0 \leq i \leq R-1.
\end{align}
In particular,
\begin{align}
\label{eq:pippiq}
\pi_p & \equiv \sum_{i=0}^{L-1} \pi_{(p,i)} =  \frac{(1-q)q^R(1 - (1-p)^L)}{(1-q)q^R(1 - (1-p)^L) + p(1-p)^L(1-q^R)} ~, \mbox{ and } \nonumber \\
\pi_q & \equiv \sum_{i=0}^{R-1} \pi_{(q,i)} =  \frac{p(1-p)^L(1-q^R)}{(1-q)q^R(1 - (1-p)^L) + p(1-p)^L(1-q^R)}.
\end{align}
So, defining $\phi : \Lambdah \rightarrow \{0,1\}$ by $\phi(\lambda,j) = \indicator\{j= 1\}$, we have
\begin{align}
\label{eq:ExpectedValuePhiEqualAlpha}
\E_{\pih}(\phi) = p \cdot \pi_p + q \cdot \pi_q = \alpha ,
\end{align}
where $\pih$ is the stationary distribution for the transition matrix $\Mh$, and $\alpha \in (0,1)$ is as
in Theorem \ref{thm:RightLeftTransienceCutoff}.  It follows, by the ergodic theorem for finite-state
Markov chains, that the limiting fraction of right jumps in the sequence $(J_n^x)_{n \in \N}$
is equal to $\alpha$ a.s., for each site $x$.

\subsubsection{The Right Jumps Markov Chain $(Z_x)_{x \geq 0}$}
\label{subsubsec:RightJumpsMarkovChain}

The \emph{right jumps Markov chain} $(Z_x)_{x \geq 0}$ is defined as follows:
\begin{itemize}
\item $Z_0 = 1$.
\item For $x \geq 1$,
\begin{align}
\label{eq:DefRightJumpsMC}
Z_x = \Theta_x - Z_{x-1} ~~\mbox{ where }~~ \Theta_x = \inf \Big\{n \geq 0: \sum_{m=1}^{n} \indicator\{J_m^x = -1\} = Z_{x-1} \Big\},
\end{align}
\end{itemize}
with the convention that $\sum_{m=1}^0\indicator\{J_m^x = -1\} =0$.
That is, $\Theta_x$ is the first time that there are $Z_{x-1}$ left jumps in the sequence $(J_n^x)_{n \in \N}$, and
$Z_x = \Theta_x - Z_{x-1}$ is the total number of right jumps in the sequence $(J_n^x)_{n \in \N}$
before there are $Z_{x-1}$ left jumps.

For an initial environment $\omega$, we denote the probability measure for the right jumps chain $(Z_x)$
also by $\P_{\omega}$. This is simply the projection of the measure $\P_{\omega}$ for the extended single site
Markov chains, of which the right jumps chain is a deterministic function. \\

\noindent
\emph{Relation to the Random Walk $(X_n)$} \\

We denote by $T_x$ the first hitting time of site $x$,
\begin{align*}
T_x = \inf \{n \geq 0 : X_n = x\} ~,~ x \in \Z.
\end{align*}
Also, we say that a jump pattern $\{J_n^x\}_{n \in \N, x \in \Z}$ is \emph{non-degenerate} if
\begin{align*}
|\{n : J_{n+1}^x \not= J_n^x \}| = \infty ~,~ \mbox{for each } x \in \Z.
\end{align*}
Clearly, for any initial environment $\omega$, the corresponding jump
pattern $\{J_n^x\}_{n \in \N, x \in \Z}$ is non-degenerate $\P_{\omega}$ a.s.
The following important proposition relating transience/recurrence of the random walk
$(X_n)$ to survival of the Markov chain $(Z_x)$ is shown in \cite{Amir2013}\footnotemark{}.
\footnotetext{The terminology there is slightly different. The jump pattern is referred to as an 
\emph{arrow environment} and denoted by $a$. After the arrow environment is 
chosen (according to some random rule which differs depending on the model) 
the walker follows the directional arrows deterministically on its walk.}

\begin{Prop}
\label{prop:SurvivalZxTransienceXn}
If $X_0 = 1$ and $\{J_n^x\}_{n \in \N, x \in \Z}$ is non-degenerate, then
\begin{align}
\label{eq:SurvivalZxTransienceXn}
T_0 = \infty ~\mbox{ if and only if } Z_x > 0, \mbox{ for all } x > 0.
\end{align}
Moreover, if $T_0 < \infty$ then, for each $x \in \N$, $Z_x$ is equal to the number of right
jumps of the process $(X_n)$ from site $x$ before hitting 0.
\end{Prop}

\subsection{Basic Lemmas}
\label{subsec:BasicLemmas}

For $n \geq 0$ we denote by $\AM_n^+$ the event that the random walk steps right at time $n$
and never returns to its time-$n$ location, and by $\AM_n^{-}$ the event that the random walks
steps left at time $n$ and never returns:
\begin{align}
\label{eq:DefAnpm}
\AM_n^+ = \{X_m > X_n, \forall m > n\} ~\mbox{ and }~
\AM_n^- = \{X_m < X_n, \forall m > n\}.
\end{align}
The following simple facts will be needed in several instances below. A proof
is provided in Appendix \ref{sec:BasicTransienceConditions}.

\begin{Lem}
\label{lem:SimpleTransConditions}
For any initial environment $\omega$:
\begin{itemize}
\item[(i)] $\P_{\omega}(\AM_0^+) > 0$ if and only if $\P_{\omega}(X_n \rightarrow \infty) > 0$, and \\
$\P_{\omega}(\AM_0^-) > 0$ if and only if $\P_{\omega}(X_n \rightarrow -\infty) > 0$.
\item[(ii)] $\P_{\omega}(X_n \rightarrow \infty) = \P_{\omega}(\liminf_{n \to \infty} X_n > -\infty)$, and \\
$\P_{\omega}(X_n \rightarrow -\infty) = \P_{\omega}(\limsup_{n \to \infty} X_n < \infty)$.
\item[(iii)] $\P_{\omega}(X_n \rightarrow \infty) = 1$ if $\P_{\omega}(X_n \rightarrow \infty) > 0$ and $\P_{\omega}(X_n \rightarrow -\infty) = 0$. \\
$\P_{\omega}(X_n \rightarrow -\infty) = 1$ if $\P_{\omega}(X_n \rightarrow -\infty) > 0$ and $\P_{\omega}(X_n \rightarrow \infty) = 0$.
\end{itemize}
\end{Lem}

Combining Proposition \ref{prop:SurvivalZxTransienceXn} and part (i) of Lemma 
\ref{lem:SimpleTransConditions} gives the following useful lemma.

\begin{Lem}
\label{lem:ProbSurvivalZxProbTransienceXn}
For any initial environment $\omega$,
\begin{align}
\label{eq:ProbA0Plus}
\P_{\omega}(\AM_0^+) = \P_{\omega}(X_1 = 1) \cdot \P_{\omega}(Z_x > 0, \forall x > 0).
\end{align}
Consequently, $\P_{\omega}(X_n \rightarrow \infty) > 0 \mbox{ if and only if } \P_{\omega}(Z_x > 0, \forall x > 0) > 0$.
\end{Lem}

\begin{proof}
Fix any initial environment $\omega$, and let $\omega'$ denote the environment at time $1$
induced by jumping right from $X_0 = 0$ starting in $\omega$:
\begin{align*}
\{\omega_0 = \omega, X_0 = 0, X_1 = 1\} \implies \omega_1 = \omega'.
\end{align*}
Since $\omega(x) = \omega'(x)$, for all $x > 0$, the distribution of the random variables $(J_n^x)_{n,x > 0}$,
is the same in the two environments $\omega$ and $\omega'$. Thus,
\begin{align*}
\P_{\omega}(Z_x > 0, \forall x > 0) = \P_{\omega'}(Z_x > 0, \forall x > 0).
\end{align*}
So,
\begin{align*}
\P_{\omega}(\AM_0^+)
& \equiv \P_{\omega,0}(X_n > 0, \forall n > 0) \\
& = \P_{\omega,0}(X_1 = 1) \cdot \P_{\omega,0}(X_n > 0, \forall n > 1|X_1 = 1) \\
& = \P_{\omega,0}(X_1 = 1) \cdot \P_{\omega',1}(X_n > 0, \forall n > 0) \\
& \stackrel{(*)}{=} \P_{\omega,0}(X_1 = 1) \cdot \P_{\omega'}(Z_x > 0, \forall x > 0) \\
& = \P_{\omega,0}(X_1 = 1) \cdot \P_{\omega}(Z_x > 0, \forall x > 0).
\end{align*}
This proves \eqref{eq:ProbA0Plus}, and the ``consequently'' part of the proposition follows immediately from part (i) of
Lemma \ref{lem:SimpleTransConditions}. Step $(*)$ follows from Proposition \ref{prop:SurvivalZxTransienceXn}\footnotemark{}.
\end{proof}

\footnotetext{ In the proof we have used the explicit notation $\P_{\omega,0}$, rather than simply $\P_{\omega}$,
for the random walk variables $X_n$, $n \geq 0$, to emphasize that the initial position $X_0 = 0$ plays
a role in their distribution. By contrast, $\P_{\omega}$, $\P_{\omega'}$ are used for the distribution of the
right jumps Markov chain $(Z_x)_{x \geq 0}$, where the initial position of the random walk plays no role.}

For the proofs of Theorems \ref{thm:BallisticityWhenNonCritical} and \ref{thm:R1Speed} we will
need the following lemma relating hitting times to speed. The same result is shown in \cite[Lemma 2.1.17]{Zeitouni2004},
for the case $C < \infty$ without the a priori assumption that $X_n \rightarrow \infty$. It is easy to see that with this
assumption the claim also holds in the case $C = \infty$.
\begin{Lem}
\label{lem:HittingTimesVersusSpeed}
If $\lim_{n \to \infty} X_n = \infty$ and $\lim_{x \to \infty} T_x/x = C  \in (0,\infty]$, then
\begin{align*}
\lim_{n \to \infty} X_n/n = 1/C.
\end{align*}
\end{Lem}
We note that, although stated in  \cite[Lemma 2.1.17]{Zeitouni2004} in the context of random walks in random environment,
the proof is entirely non-probabilistic and holds for \emph{any} nearest neighbor walk trajectory $(X_0,X_1,...)$ such that
$X_n \rightarrow \infty$ and $\lim_{x \to \infty} T_x/x = C$.

\subsection{Comparison of Environments}
\label{subsec:ComparisonOfEnvironments}

Let $\prec$ be the ordering on the set of single site configurations $\Lambda$ defined by
\begin{align*}
(q,0) \prec ... \prec (q,R-1) \prec (p,L-1) \prec ... \prec (p,0).
\end{align*}
We write $\lambda \preceq \lambdat$ if $\lambda \prec \lambdat$ or $\lambda = \lambdat$,
and $\omega \preceq \omegat$ if $\omega(x) \preceq \omegat(x)$, for all $x \in \Z$.
In this case, we also say that the environment $\omegat$ \emph{dominates} the environment $\omega$.
The following lemma relating the possibility of right transience in different environments will be important
for the analysis of transience and recurrence in the critical case $\alpha = 1/2$.

\begin{Lem}
\label{lem:ComparisonOfEnvironments}
If $q < p$ and $\omega \preceq \omegat$, then $\P_{\omega}(\AM_0^+) \leq \P_{\omegat}(\AM_0^+)$.
In particular, by Lemma \ref{lem:SimpleTransConditions}, if $q < p$, $\omega \preceq \omegat$,
and $\P_{\omegat}(X_n \rightarrow \infty) = 0$, then $\P_{\omega}(X_n \rightarrow \infty) = 0$.
\end{Lem}

For the proof it will be convenient to introduce the following definitions.
\begin{itemize}
\item The threshold function $f: \Lambda \times [0,1] \rightarrow \{1,-1\}$ is defined by
\begin{align*}
f(\lambda, u) & = \indicator \{u \leq p \} - \indicator \{u > p \} ~,~ \mbox{ if }\lambda \in \Lambda_p\\
f(\lambda, u) & = \indicator \{u \leq q \} - \indicator \{u > q \} ~,~ \mbox{ if } \lambda \in \Lambda_q.
\end{align*}
\item The transition function $g: \Lambda \times \{1,-1\} \rightarrow \Lambda$ is defined by
\begin{align*}
g(\lambda,j) = {\lambda'} \iff \{ Y_n^x = \lambda, J_n^x = j \} \mbox{ implies } Y_{n+1}^x = \lambda'.
\end{align*}
That is, $g(\lambda,j)$ is the (deterministic) next configuration at site $x$ if the walk jumps
in direction $j$ from site $x$ when $x$ is in configuration $\lambda$.
\end{itemize}

\begin{proof}[Proof of Lemma \ref{lem:ComparisonOfEnvironments}]
For $x \in \Z$, let $(Y_n^x, J_n^x)_{n \in \N}$ and $(\Yt_n^x, \Jt_n^x)_{n \in \N}$ denote, respectively, the
state sequences of the extended single site Markov chains at $x$ for the environments $\omega$ and $\omegat$.
Also, let $(U_n^x)_{x \in \Z, n \in \N}$ be i.i.d. uniform([0,1]) random variables. For each $x$, we will
use the i.i.d. sequence $(U_n^x)_{n \in \N}$ to couple the state sequences $(Y_n^x, J_n^x)_{n \in \N}$ and
$(\Yt_n^x, \Jt_n^x)_{n \in \N}$ in such a way that $J_n^x \leq \Jt_n^x$, for all $n$. By independence, this coupling
at each individual site $x$ passes to a coupling of the entire joint processes $(Y_n^x, J_n^x)_{x \in \Z, n \in \N}$
and $(\Yt_n^x, \Jt_n^x)_{x \in \Z, n \in \N}$, with the correct law. This final larger coupling will be used to show that
$\P_{\omega}(\AM_0^+) \leq \P_{\omegat}(\AM_0^+)$. \\

\noindent
\emph{Step 1}: The Coupling \\
For a fixed site $x$, we construct the sequences $(Y_n^x, J_n^x)_{n \in \N}$ and $(\Yt_n^x, \Jt_n^x)_{n \in \N}$
inductively from the i.i.d. random variables $(U_n^x)_{n \in \N}$ as follows.
\begin{itemize}
\item $Y_1^x = \omega(x)$ and $\Yt_1^x = \omegat(x)$.
\item For $n \geq 1$,
\begin{align*}
J_n^x = f(Y_n^x,U_n^x),Y_{n+1}^x = g(Y_n^x, J_n^x) ~\mbox{ and }~
\Jt_n^x = f(\Yt_n^x,U_n^x), \Yt_{n+1}^x = g(\Yt_n^x, \Jt_n^x).
\end{align*}
\end{itemize}
Clearly, the sequences $(Y_n^x,J_n^x)_{n \in \N}$ and $(\Yt_n^x,\Jt_n^x)_{n \in \N}$ each have
have the appropriate marginal laws under this coupling. Moreover, by considering the various possible
cases for $Y_n^x, \Yt_n^x \in \Lambda$ and possible ranges for $U_n^x \in [0,1]$ one finds that,
since $q < p$, whatever the value of $U_n^x$ is:
\begin{align*}
Y_n^x \preceq \Yt_n^x \implies J_n^x \leq \Jt_n^x \mbox{ and } Y_{n+1}^x \preceq \Yt_{n+1}^x.
\end{align*}
Since $Y_1^x = \omega(x) \preceq \omegat(x) = \Yt_1^x$ it follows, by induction, that
\begin{align}
\label{eq:DominationJumpSequences}
J_n^x \leq \Jt_n^x, \mbox{ for all } n.
\end{align}

\noindent
\emph{Step 2}: Relation to the probability of $\AM_0^+$ \\
Let $(Z_x)_{x \geq 0}$ and $(\Zt_x)_{x \geq 0}$ denote, respectively, the right jumps Markov chains
constructed from the jump patterns $(J_n^x)_{x \in \Z, n \in \N}$ and $(\Jt_n^x)_{x \in \Z, n \in \N}$
according to (\ref{eq:DefRightJumpsMC}). Also, for $x, k \geq 0$ define $\Theta_{x,k}$ and $\Thetat_{x,k}$ by
\begin{align*}
\Theta_{x,k} = \inf \Big\{n : \sum_{m=1}^n \indicator \{J_m^x = -1\} = k \Big\} ~,~
\Thetat_{x,k} = \inf \Big\{n : \sum_{m=1}^n \indicator \{ \Jt_m^x = -1\} = k \Big\}.
\end{align*}
If $Z_{x-1} \leq \Zt_{x-1}$, then applying the definition (\ref{eq:DefRightJumpsMC}) gives
\begin{align*}
Z_x & = \sum_{m=1}^{\Theta_{x,Z_{x-1}}} \indicator \{J_m^x = 1\}
\stackrel{(a)}{\leq} \sum_{m=1}^{\Thetat_{x,Z_{x-1}}} \indicator \{J_m^x = 1\} \\
& \leq \sum_{m=1}^{\Thetat_{x,\Zt_{x-1}}} \indicator \{J_m^x = 1\}
\stackrel{(b)}{\leq} \sum_{m=1}^{\Thetat_{x,\Zt_{x-1}}} \indicator \{\Jt_m^x = 1\}
= \Zt_x.
\end{align*}
Here, (a) follows from (\ref{eq:DominationJumpSequences}), which implies $\Theta_{x,k} \leq \Thetat_{x,k}$
for any $k$, and (b) follows directly from (\ref{eq:DominationJumpSequences}). Since $Z_0 = \Zt_0 = 1$,
it follows, by induction, that
\begin{align}
\label{eq:DominationRightJumpsChains}
Z_x \leq \Zt_x , \mbox{ for all } x \in \Z.
\end{align}
Now, since (\ref{eq:DominationJumpSequences}) and (\ref{eq:DominationRightJumpsChains}) both hold with probability 1,
under our coupling, it follows from Lemma \ref{lem:ProbSurvivalZxProbTransienceXn} that
\begin{align*}
\P_{\omega}(\AM_0^+)
& = \P_{\omega}(X_1 = 1) \cdot \P_{\omega}(Z_x > 0, \forall x > 0)  \\
& = \P_{\omega}(J_1^0 = 1) \cdot \P_{\omega}(Z_x > 0, \forall x > 0) \\
& \leq \P_{\omegat}(J_1^0 = 1) \cdot \P_{\omegat}(Z_x > 0, \forall x > 0) \\
& = \P_{\omegat}(X_1 = 1) \cdot \P_{\omegat}(Z_x > 0, \forall x > 0) \\
& = \P_{\omegat}(\AM_0^+).
\end{align*}
Here, we have dropped the tildes on all random variables corresponding to the initial
environment $\omegat$, since the probability measure $\P_{\omegat}$ is now explicit.
\end{proof}

\section{The Noncritical Case}
\label{sec:NoncriticalCase}

Here we analyze the behavior of the random walk $(X_n)$ for $\alpha \not= 1/2$, proving Theorems
\ref{thm:RightLeftTransienceCutoff}--\ref{thm:R1Speed}. We begin in section \ref{subsec:SurvivalRightJumpsMarkovChain}
with a key lemma for the survival probability of the right jumps Markov chain $(Z_x)$, from which we derive a number of useful
corollaries. Using these results, Theorem \ref{thm:RightLeftTransienceCutoff} on the cutoff for right/left transience and Theorem
\ref{thm:BallisticityWhenNonCritical} on ballisticity of the random walk are then proved in sections \ref{subsec:CutoffForRightLeftTransience}
and \ref{subsec:Ballisticity}. Theorems \ref{thm:L1Speed} and \ref{thm:R1Speed} on the exact speed of the random walk in certain
special cases are proved afterward in sections \ref{subsec:SpeedWithL1} and \ref{subsec:SpeedWithR1}.

Throughout we use the following notation:
\begin{itemize}
\item $T_x^{(i)}$, $x \in \Z$ and $i \in \N$, is the $i$-th hitting time of site $x$.
\begin{align}
\label{eq:Def_ith_HittingTimes}
T_x^{(1)} = T_x ~~\mbox{ and }~~ T_x^{(i+1)} = \inf\{n > T_x^{(i)}: X_n = x\},
\end{align}
with the convention $T_x^{(j)} = \infty$, for all $j > i$, if $T_x^{(i)} = \infty$.
\item $N_x$ is the total number of visits to site $x$, as in (\ref{eq:DefNx}), and
$N_x^y$ is the number of visits to site $x$ up to time $T_y$.
\begin{align*}
N_x & = |\{n \geq 0: X_n = x\}| ~,~ x \in \Z. \\
N_x^y & = |\{0 \leq n \leq T_y: X_n = x\}| ~,~ x,y \in \Z.
\end{align*}
\item $R_x$ is the total number of right jumps from site $x$, and $L_x$ is the
total number of left jumps from site $x$.
\begin{align*}
R_x & = |\{n \geq 0: X_n = x, X_{n+1} = x + 1\}| ~,~ x \in \Z. \\
L_x & = |\{n \geq 0: X_n = x, X_{n+1} = x - 1\}| ~,~ x \in \Z.
\end{align*}
\item $B_x$ is the farthest distance the random walk ever steps backward from site $x$ after hitting $x$ for the first time.
\begin{align*}
B_x = \sup \{k \geq 0: \exists n \geq T_x \mbox{ with } X_n = x - k \} ~,~ x \in \Z.
\end{align*}
In the case $T_x = \infty$, $B_x \equiv 0$.
\item $\AM_n^+$, given by (\ref{eq:DefAnpm}), is the event that the random walk steps to the right at time $n$
and never returns to its time-$n$ location.
\item $\BM_{\epsilon}$, $0 < \epsilon < 1$, is the event that $B_x \leq \epsilon x$, for all sufficiently large $x$.
\begin{align}
\label{eq:DefBepsilon}
\BM_{\epsilon} = \left\{\exists N \in \N \mbox{ s.t. } B_x \leq \epsilon x, \forall x \geq N  \right\}.
\end{align}
\end{itemize}

\subsection{Survival of Right Jumps Markov Chain $(Z_x)$}
\label{subsec:SurvivalRightJumpsMarkovChain}

\begin{Lem}
\label{Lem:ProbZxGreater0}
If $\alpha = \alpha(p,q,R,L) > 1/2$, then there exists some $\beta = \beta(p,q,R,L) > 0$ such that,
for any initial environment $\omega$,
\begin{align*}
\P_{\omega}(Z_x > 0, \forall x > 0) \geq \beta.
\end{align*}
\end{Lem}

\begin{proof}
Fix $p,q,R,L$ such that $\alpha > 1/2$ and any initial environment $\omega$. Define $0 < \epsilon < 1/4$ by the relation
$\alpha = 1/2 + 2 \epsilon$, and for $\lambdah = (\lambda,j) \in \Lambdah$, let $\phi(\lambdah) = \indicator\{j = 1\}$.

By (\ref{eq:ExpectedValuePhiEqualAlpha}) we have $\alpha = \E_{\pih}(\phi)$, where $\pih$ is the stationary distribution for the
extended single site transition matrix $\Mh$. So, by standard large deviation bounds for finite-state Markov chains, there exist
some $0 < a < 1$ and $n_0 \in \N$ such that for any initial state $\lambdah \in \Lambdah$ the Markov chain $(\Yh_n)$ with transition
matrix $\Mh$ satisfies
\begin{align*}
\P_{\lambdah}\left(\frac{1}{n} \sum_{m=1}^{n} \phi(\Yh_m) \leq 1/2 + \epsilon \right)
= \P_{\lambdah}\left(\frac{1}{n} \sum_{m=1}^{n} \phi(\Yh_m) \leq \E_{\pih}(\phi) - \epsilon \right)
\leq a^n~, n \geq n_0.
\end{align*}
Using this estimate we obtain the following important inequality:
\begin{align}
\label{eq:Zx_given_ZxMinus1}
\P_{\omega} (Z_x& \leq n(1/2 + \epsilon)/(1/2 - \epsilon) ~|~ Z_{x-1} = n) \nonumber \\
& = \P_{\omega}(\Theta_x \leq n/(1/2 - \epsilon) ~|~ Z_{x-1} = n) \nonumber \\
& = \P_{\omega}\left( \exists~ n \leq m \leq n/(1/2-\epsilon) : \sum_{i = 1}^{m} (1 - \phi(\Yh_i^x)) = n \right) \nonumber \\
& = \P_{\omega}\left( \exists~ n \leq m \leq n/(1/2-\epsilon) : \frac{1}{m} \sum_{i = 1}^{m} \phi(\Yh_i^x) = \frac{m - n}{m} \right) \nonumber \\
& \leq \P_{\omega}\left( \exists~ n \leq m \leq n/(1/2-\epsilon) : \frac{1}{m} \sum_{i = 1}^{m} \phi(\Yh_i^x) \leq 1/2 + \epsilon \right) \nonumber \\
& \leq \P_{\omega}\left( \exists m \geq n : \frac{1}{m} \sum_{i = 1}^{m} \phi(\Yh_i^x) \leq 1/2 + \epsilon \right) \nonumber \\
& \leq \sum_{m=n}^{\infty} a^{m} = \frac{a^n}{1-a} ~,~ \mbox{for all $n \geq n_0$ and $x \in \N$.}
\end{align}

Now, define $b > 1$ by $b = \frac{1/2 + \epsilon}{1/2 - \epsilon}$, and take $n_1 \geq n_0$ sufficiently large
that $\frac{a^{n_1}}{1-a} < 1$. Thus, $\frac{a^{\lceil n_1 b^{x-1}\rceil}}{1-a} < 1$, $\forall x \in \N$.
Applying the inequality (\ref{eq:Zx_given_ZxMinus1}) gives, \\
\begin{align*}
\P_{\omega}( & Z_x > 0, \forall x > 0) \\
& \geq \P_{\omega}( Z_x \geq n_1 b^x, \forall x > 0) \\
& = \P_{\omega}(Z_1 \geq n_1 b) \cdot \prod_{x=2}^{\infty}
\P_{\omega}(Z_x \geq n_1 b^x | Z_1 \geq n_1 b, ... , Z_{x-1} \geq n_1 b^{x-1}) \\
& \geq \P_{\omega}(Z_1 \geq n_1 b) \cdot \prod_{x=2}^{\infty}
\P_{\omega}(Z_x \geq n_1 b^x | Z_{x-1} = \lceil n_1 b^{x-1} \rceil) \\
& \geq (\min\{p,q\})^{\ceil{n_1 b}} \cdot \prod_{x=2}^{\infty}
\left(1 - \frac{a^{\lceil n_1 b^{x-1}\rceil}}{1-a} \right) \equiv \beta.
\end{align*}
Note that $\sum_{x=2}^{\infty} \frac{a^{\lceil n_1 b^{x-1}\rceil}}{1-a} < \infty$, so
$\prod_{x=2}^{\infty} \left(1 - \frac{a^{\lceil n_1 b^{x-1}\rceil}}{1-a} \right) > 0$.
\end{proof}

\begin{Cor}
\label{cor:PrAnPlusGreaterEqualBeta}
If $\alpha = \alpha(p,q,R,L) > 1/2$ then there exists some $\beta = \beta(p,q,R,L) > 0$, such that
for any initial environment $\omega$ and random walk path $(x_0, x_1,...,x_n)$,
\begin{align}
\label{eq:PrAnPlusGreaterEqualBeta}
\P_{\omega,x_0}(\AM_n^+|X_0 = x_0,...,X_n = x_n) \geq \beta.
\end{align}
\end{Cor}

\begin{proof}
Since the claimed bound is uniform in the initial environment $\omega$, it suffices to consider
the case $x_0 = n = 0$. By Lemma \ref{Lem:ProbZxGreater0}, there exists some $\beta' > 0$
such that $\P_{\omega}(Z_x > 0, \forall x > 0) \geq \beta'$, for any initial environment $\omega$.
Thus, by Lemma \ref{lem:ProbSurvivalZxProbTransienceXn},
\begin{align*}
\P_{\omega}(\AM_0^+) \geq \min\{p,q\} \cdot \beta' \equiv \beta
\end{align*}
for any initial environment $\omega$.
\end{proof}

\begin{Cor}
\label{cor:NxDominatedByGeometric}
If $\alpha > 1/2$ then, for any initial environment $\omega$ and site $x \geq 0$,
\begin{align*}
\P_{\omega}(N_x \geq k) \leq (1 - \beta)^{k-1} ~,~ \mbox{ for all } k \geq 1
\end{align*}
where $\beta > 0$ is the constant in Corollary \ref{cor:PrAnPlusGreaterEqualBeta}.
\end{Cor}

\begin{proof}
Let $A_x^{(i)}$ be the set of all random walk paths $(x_0, x_1, ..., x_n)$, of any length $n$,
which end in an $i$-th hitting time of site $x$. That is, $\{X_0 = x_0, X_1 = x_1,..., X_n = x_n\} \implies T_x^{(i)} = n$.
For brevity we denote $(X_0,...,X_n)$ as $X_0^n$ and $(x_0,...,x_n)$ as $x_0^n$. By Corollary
\ref{cor:PrAnPlusGreaterEqualBeta}, for any $i \geq 1$, we have
\begin{align*}
& \P_{\omega}(T_x^{(i+1)} < \infty | T_x^{(i)} < \infty) \\
& = \sum_{x_0^n \in A_x^{(i)}} \P_{\omega}(X_0^n = x_0^n | T_x^{(i)} < \infty)
\cdot \P_{\omega} (T_x^{(i+1)} < \infty | T_x^{(i)} < \infty, X_0^n = x_0^n) \\
& = \sum_{x_0^n \in A_x^{(i)}} \P_{\omega}(X_0^n = x_0^n | T_x^{(i)} < \infty)
\cdot \P_{\omega} (T_x^{(i+1)} < \infty | X_0^n = x_0^n) \\
& \leq \sum_{x_0^n \in A_x^{(i)}} \P_{\omega}(X_0^n = x_0^n | T_x^{(i)} < \infty)
\cdot \P_{\omega}((\AM_n^+)^c|X_0^n = x_0^n) \\
& \leq (1- \beta).
 \end{align*}
Hence, for each $k \geq 1$,
\begin{align*}
\P_{\omega}(N_x \geq k)
= \P_{\omega}(T_x^{(1)} < \infty) \cdot \prod_{i = 1}^{k-1} \P_{\omega} (T_x^{(i+1)} < \infty | T_x^{(i)} < \infty)
\leq (1-\beta)^{k-1}.
\end{align*}
 \end{proof}

\begin{Cor}
\label{cor:BxDominatedByGeometric}
If $\alpha > 1/2$ then, for any initial environment $\omega$ and site $x \geq 0$,
\begin{align}
\label{eq:ExponentialBoundOnBx}
\P_{\omega}(B_x \geq k) \leq (1 - \beta)^k ~,~ \mbox{ for all } k \geq 1
\end{align}
where $\beta > 0$ is the constant in Corollary \ref{cor:PrAnPlusGreaterEqualBeta}.
In particular, by the Borel-Cantelli lemma,
\begin{align*}
\P_{\omega}(\BM_{\epsilon}) = 1, \mbox{ for each } 0 < \epsilon < 1.
\end{align*}
\end{Cor}

\begin{proof}
The proof is similar to that of Corollary \ref{cor:NxDominatedByGeometric}.
For $x \in \Z$, let $\tau_x^{(0)}$ be the first hitting time of site $x$, and let $\tau_x^{(i)}$, $i \in \N$,
be the first time greater than $\tau_x^{(i-1)}$ at which the walk steps backward from its
position $x - (i-1)$ at time $\tau_x^{(i-1)}$. That is, $\tau_x^{(0)} = T_x$, and for $i \geq 1$,
\begin{align*}
\tau_x^{(i)}
& = \inf \{n > \tau_x^{(i-1)} : X_n < X_{\tau_x^{(i-1)}} \} \\
& = \inf \{n > T_x : X_n = x - i \}
\end{align*}
with the convention $\tau_x^{(j)} = \infty$, for all $j > i$, if $\tau_x^{(i)} = \infty$. Also, let $A_x^{(i)}$ be the
set of all random walk paths $(x_0, x_1, ..., x_n)$, of any length $n$, which end in an $i$-th ``back step time''
from site $x$. That is, $\{X_0 = x_0, X_1 = x_1,..., X_n = x_n\} \implies \tau_x^{(i)} = n$.
As above, we denote $(X_0,...,X_n)$ as $X_0^n$ and $(x_0,...,x_n)$ as $x_0^n$. By Corollary
\ref{cor:PrAnPlusGreaterEqualBeta}, for any $i \geq 0$, we have
\begin{align*}
& \P_{\omega}(\tau_x^{(i+1)} < \infty | \tau_x^{(i)} < \infty) \\
& = \sum_{x_0^n \in A_x^{(i)}} \P_{\omega}(X_0^n = x_0^n | \tau_x^{(i)} < \infty)
\cdot \P_{\omega} (\tau_x^{(i+1)} < \infty | \tau_x^{(i)} < \infty, X_0^n = x_0^n) \\
& = \sum_{x_0^n \in A_x^{(i)}} \P_{\omega}(X_0^n = x_0^n | \tau_x^{(i)} < \infty)
\cdot \P_{\omega} (\tau_x^{(i+1)} < \infty | X_0^n = x_0^n) \\
& \leq \sum_{x_0^n \in A_x^{(i)}} \P_{\omega}(X_0^n = x_0^n | \tau_x^{(i)} < \infty)
\cdot \P_{\omega}((\AM_n^+)^c|X_0^n = x_0^n) \\
& \leq (1- \beta).
\end{align*}
So, for each $k \geq 1$,
\begin{align*}
\P_{\omega}(B_x \geq k)
= \P_{\omega}(\tau_x^{(0)} < \infty) \cdot \prod_{i = 0}^{k-1} \P_{\omega} (\tau_x^{(i+1)} < \infty | \tau_x^{(i)} < \infty)
\leq (1-\beta)^k.
\end{align*}
\end{proof}

\subsection{Proof of Theorem \ref{thm:RightLeftTransienceCutoff}}
\label{subsec:CutoffForRightLeftTransience}

\begin{proof}[Proof of Theorem \ref{thm:RightLeftTransienceCutoff}] 
If $\alpha > 1/2$ then Corollary \ref{cor:BxDominatedByGeometric} implies that $B_0$ is $\P_{\omega}$ a.s. finite,
for any initial environment $\omega$. Thus, by part (ii) of Lemma \ref{lem:SimpleTransConditions}, for $\alpha > 1/2$
we must have $X_n \rightarrow \infty$, $\P_{\omega}$ a.s., for any initial environment $\omega$. It follows by symmetry
that, for $\alpha < 1/2$ and any initial environment $\omega$, $X_n \rightarrow -\infty$, $\P_{\omega}$ a.s.
\end{proof}

\subsection{Proof of Theorem \ref{thm:BallisticityWhenNonCritical}}
\label{subsec:Ballisticity}

For the proof of Theorem \ref{thm:BallisticityWhenNonCritical} we will assume
that $\alpha > 1/2$, the case $\alpha < 1/2$ follows by symmetry considerations.
The primary ingredients for the proof are Corollaries \ref{cor:NxDominatedByGeometric} and
\ref{cor:BxDominatedByGeometric}, above, and Lemmas \ref{lem:SimpleConsequencesAlphaGreaterHalf}
and \ref{lem:StrongLawForNx}, given below. Lemma \ref{lem:SimpleConsequencesAlphaGreaterHalf} is a
simple consequence of Theorem \ref{thm:RightLeftTransienceCutoff}. Lemma \ref{lem:StrongLawForNx}
shows that, when $\alpha > 1/2$, the sequence $(N_x)$ obeys a strong law of large numbers. The
proof of this fact is somewhat lengthy and is deferred to Appendix \ref{sec:StrongLawForNx}.

\begin{Lem}
\label{lem:SimpleConsequencesAlphaGreaterHalf}
Assume that $\alpha > 1/2$ and $X_0 = 0$.
\begin{itemize}
\item[(i)] For all $x \geq 0$, the random variables $N_x$, $L_x$, and $R_x$ are each
independent of the environment to the left of site $x$ when site $x$ is first reached:
\begin{align*}
N_x, L_x, R_x \perp \{ \omega_{T_x}(y), y < x \}.
\end{align*}
\item[(ii)] $N_x^y$ and $N_y$ are independent, for all $0 \leq x < y$.
\item[(iii)] If, for some $y \geq 0$, $\omega(x)$ is constant for $x \geq y$, then $N_x$ and $N_y$
have the same distribution for all $x \geq y$. Similarly, if $\omega(x) = \omega(y)$ for all $x \geq y$, then
$R_x$ and $R_y$ have the same distribution for all $x \geq y$, and $L_x$
and $L_y$ have the same distribution for all $x \geq y$.
\end{itemize}
\end{Lem}

\begin{proof}
Since $\alpha > 1/2$ and $X_0 = 0$, Theorem \ref{thm:RightLeftTransienceCutoff} shows
that $T_x$ is a.s. finite, for each $x \geq 0$, and that regardless of the environment to the left of site $x$
at time $T_x$, the walks returns to site $x$ with probability 1 each time it steps left from $x$.
This implies (i). Now, (ii) and (iii) follow easily since (i) shows that the distribution of $N_x$, $L_x$, and $R_x$
are each entirely determined by the values of $\omega_{T_x}(y), y \geq x$, which are the same as the original
values $\omega(y), y \geq x$.
\end{proof}

\begin{Lem}
\label{lem:StrongLawForNx}
If $\alpha > 1/2$ then, for any initial environment $\omega$,
\begin{align*}
\lim_{n \to \infty} \frac{1}{n} \sum_{x = 1}^n (N_x - \E_{\omega}(N_x)) = 0, ~\P_{\omega} \mbox{ a.s. }
\end{align*}
\end{Lem}

\begin{proof}[Proof of Theorem \ref{thm:BallisticityWhenNonCritical}, Equation (\ref{eq:LiminfXnnGreaterDelta}), with $\alpha > 1/2$]
Let $\alpha > 1/2$, and fix any initial environment $\omega$. Also, let $\beta > 0$ be the constant defined
in Corollary \ref{cor:PrAnPlusGreaterEqualBeta}. We will show that:
\begin{itemize}
\item[(i)] $\limsup_{x \to \infty} \frac{1}{x} \sum_{y = 1}^x N_y \leq 1/\beta,~ \P_{\omega} \mbox{ a.s. }$
\item[(ii)] $\limsup_{x \to \infty} T_x/x \leq \limsup_{x \to \infty} \frac{1}{x} \sum_{y = 1}^x N_y,~ \P_{\omega} \mbox{ a.s. }$
\item[(iii)] $\liminf_{n \to \infty} X_n/n \geq \left(\limsup_{x \to \infty} T_x/x\right)^{-1},~ \P_{\omega} \mbox{ a.s. }$
\end{itemize}
The result (\ref{eq:LiminfXnnGreaterDelta}) follows directly from these three facts. \\

\noindent
\emph{Proof of (i)}:
This is immediate from Lemma \ref{lem:StrongLawForNx} and Corollary \ref{cor:NxDominatedByGeometric}. \\

\noindent
\emph{Proof of (ii)}:
Since $\alpha > 1/2$, $X_n \rightarrow \infty$ $\P_{\omega}$ a.s. So, $\sum_{x \leq 0} N_x$ is $\P_{\omega}$ a.s. finite.
Thus, $\P_{\omega}$ a.s. we have
\begin{align*}
\limsup_{x \to \infty} \frac{T_x}{x}
= \limsup_{x \to \infty} \frac{1}{x} \sum_{y = -\infty}^x N_y^x
\leq \limsup_{x \to \infty} \frac{1}{x} \sum_{y = -\infty}^x N_y
= \limsup_{x \to \infty} \frac{1}{x} \sum_{y = 1}^x N_y.
\end{align*}

\noindent
\emph{Proof of (iii)}:
For $0 < \epsilon < 1$, let $\BM_{\epsilon}' = \BM_{\epsilon} \cap \{T_x < \infty, \forall x > 0\}$, where $\BM_{\epsilon}$ is defined
by (\ref{eq:DefBepsilon}). On the event $\BM_{\epsilon}'$, for all sufficiently large $x$ and $T_x \leq n < T_{x+1}$, we have
\begin{align*}
\frac{X_n}{n} \geq \frac{x - \epsilon x}{T_{x+1}} = (1-\epsilon) \frac{x+1}{T_{x+1}} - \frac{1-\epsilon}{T_{x+1}}.
\end{align*}
So,
\begin{align*}
\liminf_{n \to \infty} \frac{X_n}{n}
\geq \liminf_{x \to \infty}~ \left( (1- \epsilon) \frac{x+1}{T_{x+1}} - \frac{1-\epsilon}{T_{x+1}} \right)
= (1-\epsilon) \cdot \left(\limsup_{x \to \infty} T_x/x\right)^{-1}.
\end{align*}
The result follows since $\P_{\omega}(\BM_{\epsilon}') = 1$, for each $\epsilon > 0$, due to
Corollary \ref{cor:BxDominatedByGeometric} and the fact that the random walk $(X_n)$
is a.s. right transient with $\alpha > 1/2$.
\end{proof}

\begin{proof}[Proof of Theorem \ref{thm:BallisticityWhenNonCritical}, Equation (\ref{eq:SpeedENxInverse}), with $\alpha > 1/2$]
By assumption $\omega(x)$ is constant for $x \geq m$, so Lemma \ref{lem:SimpleConsequencesAlphaGreaterHalf}
implies $N_x$ and $N_m$ are equal in law, for all $x \geq m$, under $\P_{\omega}$. Thus,
\begin{align}
\label{eq:ENxEqualEN0}
\E_{\omega}(N_x) = \E_{\omega}(N_m) \equiv \gamma, \mbox{ for all } x \geq m.
\end{align}
To show that $\lim_{n \to \infty} X_n/n = 1/\gamma$, $\P_{\omega}$ a.s., note
first that (\ref{eq:ENxEqualEN0}) and Lemma \ref{lem:StrongLawForNx} imply that
\begin{align}
\label{eq:LimNxAvg}
\lim_{x \to \infty} \frac{1}{x} \sum_{y = 1}^x N_y = \gamma,~ \P_{\omega} \mbox{ a.s. }
\end{align}
So, by point (ii) above, we have
\begin{align}
\label{eq:LimsupTxxBound}
\limsup_{x \to \infty} T_x/x \leq \gamma, ~\P_{\omega} \mbox{ a.s. }
\end{align}
On the other hand, on the event $\BM_{\epsilon}' = \BM_{\epsilon} \cap \{T_x < \infty, \forall x > 0\}$, we have
\begin{align}
\label{eq:AddNyNotTx}
\liminf_{x \to \infty} \frac{T_x}{x}
= \liminf_{x \to \infty} \frac{1}{x} \sum_{y={-\infty}}^x N_y^x
\geq \liminf_{x \to \infty} \frac{1}{x} \sum_{y=1}^{\floor{(1-\epsilon)x}} N_y^x
= \liminf_{x \to \infty} \frac{1}{x} \sum_{y=1}^{\floor{(1-\epsilon)x}} N_y.
\end{align}
Since $\P_{\omega}(\BM_{\epsilon}') = 1$, for each $\epsilon > 0$, and
the RHS of (\ref{eq:AddNyNotTx}) is equal to $(1-\epsilon) \gamma$,
$\P_{\omega}$ a.s., by (\ref{eq:LimNxAvg}), this implies
\begin{align}
\label{eq:LiminfTxxBound}
\liminf_{x \to \infty} T_x/x \geq \gamma, ~\P_{\omega} \mbox{ a.s. }
\end{align}
Together, (\ref{eq:LimsupTxxBound}) and (\ref{eq:LiminfTxxBound}) imply
$\lim_{x \to \infty} T_x/x = \gamma$, $\P_{\omega}$ a.s.,
so the result follows from Lemma \ref{lem:HittingTimesVersusSpeed}.
\end{proof}

\subsection{Proof of Theorem \ref{thm:L1Speed}}
\label{subsec:SpeedWithL1}

The proof of Theorem \ref{thm:L1Speed} is based on the speed formula
given in Theorem \ref{thm:BallisticityWhenNonCritical}, and uses the assumptions on $L$
and $\omega$ to obtain a more explicit expression for $\gamma$.

\begin{proof}[Proof of Theorem \ref{thm:L1Speed}]
We will prove the theorem under the assumption $\omega(x) = (q,0)$, for all $x \geq 0$. The case
$\omega(x) = (q,0)$ in a neighborhood of $+\infty$ follows immediately from this.
The main observation is that since $L = 1$ and the random walk starts from $X_0 = 0$
in an environment $\omega$ satisfying $\omega(x) = (q,0)$, for all $x \geq 0$, we have
\begin{align*}
\omega_n(x) = (q,0), \mbox{ for each  $n \geq 0$ and $x > X_n$}.
\end{align*}
That is, the environment to the right of the current position of the
random walk always consists entirely of sites in the $(q,0)$ configuration.
Consequently, when the walk jumps right the environment both at its current
position and to its right consists entirely of sites in the $(q,0)$ configuration:
\begin{align}
\label{eq:q0CurrentAndToRight}
\{ X_{n-1} = x-1 \mbox{ and } X_n = x \}~\implies~ \omega_n(y) = (q,0), \mbox{ for all } y \geq x.
\end{align}
Using this fact we will show that:
\begin{itemize}
\item[(i)] $\gamma \equiv \E_{\omega}(N_0) = \frac{1 + \eta}{1 - \eta}$,  where $\eta \equiv \P_{\omega}(T_{-1} < \infty)$.
\item[(ii)] $\eta$ satisfies $P(\eta) = 0$, where $P$ is as in \eqref{eq:DefPolynomialP}.
\end{itemize}
Also, using direct calculus arguments we will show that:
\begin{itemize}
\item[(iii)] The polynomial $P(t)$ has a unique real root in the interval (0,1).
\end{itemize}
Clearly, $\eta > \P_{\omega}(T_{-1} =  1) = 1-q$, and by Corollary \ref{cor:PrAnPlusGreaterEqualBeta}, we know
$\eta < 1$. Thus, the theorem follows from points (i)-(iii) and Theorem \ref{thm:BallisticityWhenNonCritical}. \\

\noindent
\emph{Proof of (i):}
Since $\alpha > 1/2$ the random walk returns to site $0$ with probability $1$ every time it steps left from $0$,
and by (\ref{eq:q0CurrentAndToRight}), applied in the case $x = 0$, we know that at each time $n$ when the random
walk returns to site $0$ after stepping left on its last visit, we have $\omega_n(x) = \omega_0(x) = (q,0)$, for all $x \geq 0$.
Therefore, since $\P_{\omega'}(T_{-1} < \infty)$ does not depend on the values of $\omega'(x)$, $x < 0$,
it follows that  $L_0$ is a geometric random variable with distribution
\begin{align*}
\P_{\omega}(L_0 = k) = \eta^k (1 - \eta),~ k \geq 0.
\end{align*}
Hence, by Lemma \ref{lem:SimpleConsequencesAlphaGreaterHalf},
\begin{align*}
\E_{\omega}(N_0)
= \E_{\omega}(R_0 + L_0)
\stackrel{(*)}{=} [\E_{\omega}(L_1) + 1] + \E_{\omega}(L_0)
= 2 \E_{\omega}(L_0) + 1
= \frac{1 + \eta}{1 - \eta}.
\end{align*}
Step (*) follows from the fact that $R_0 = L_1 + 1$ a.s., since the random walk is a.s. transient to $+\infty$. \\

\noindent
\emph{Proof of (ii):}
For $i \geq 0$, let $A_i$ be the event that the random walk steps right from site $0$ and eventually
returns $i$ times without ever jumping left from 0, and let $A_i'$ be the event that the random walk steps
right from site $0$ and eventually returns $i$ times without stepping left from $0$, but then does
step left on its next visit:
\begin{align*}
A_i & = \{N_0 \geq i + 1, T_{-1} > T_0^{(i+1)}\}, \\
A_i' & = \{N_0 \geq i+1, T_{-1} = T_0^{(i+1)} + 1\}.
\end{align*}
Clearly, $\P_{\omega}(A_0) = 1$. We claim also that:
\begin{align}
\label{eq:ProbAiPrimeGivenAi}
\P_{\omega}(A_i'|A_i) =
\left\{ \begin{array}{l}
(1-q)~, \mbox{ for } 0 \leq i \leq R-1\\
(1-p)~, \mbox{ for } i \geq R
\end{array} \right.
\end{align}
and
\begin{align}
\label{eq:ProbAi1GivenAi}
\P_{\omega}(A_{i+1}|A_i) =
\left\{ \begin{array}{l}
q \eta ~, \mbox{ for } 0 \leq i \leq R-1\\
p \eta ~, \mbox{ for } i \geq R.
\end{array} \right.
\end{align}
To see (\ref{eq:ProbAiPrimeGivenAi}), note that after jumping right from site $0$ and returning $i$ times in a row,
site $0$ will be in configuration $(q,i)$, for $0 \leq i \leq R-1$, and in configuration $(p,0)$ for $i \geq R$.
Thus, for $0 \leq i \leq R-1$, we have
\begin{align*}
\P_{\omega}(A_i'|A_i)
= \P_{\omega}\left(X_{T_0^{(i+1)}+1} = -1\right. \left| \omega_{T_0^{(i+1)}}(0) = (q,i)\right)
= (1 - q)
\end{align*}
and, for $i \geq R$, we have
\begin{align*}
\P_{\omega}(A_i'|A_i)
= \P_{\omega}\left(X_{T_0^{(i+1)}+1} = -1\right. \left| \omega_{T_0^{(i+1)}}(0) = (p,0)\right)
= (1 - p).
\end{align*}
Now, (\ref{eq:ProbAi1GivenAi}) follows from (\ref{eq:ProbAiPrimeGivenAi}) and the following
calculation which is valid for all $i \geq 0$:
\begin{align*}
\P_{\omega}(A_{i+1}|A_i)
& = \P_{\omega}(X_{T_0^{(i+1)}+1} = 1| A_i) \cdot \P_{\omega}(T_0^{(i+2)} < \infty|A_i, X_{T_0^{(i+1)} + 1} = 1) \\
& = \P_{\omega}((A_i')^c| A_i) \cdot \eta.
\end{align*}
The second equality above follows from (\ref{eq:q0CurrentAndToRight}), which implies that on the event
$\{X_{T_0^{(i+1)} + 1} = 1\}$, all sites $x \geq 1$ are in the $(q,0)$ configuration at time ${T_0^{(i+1)} + 1}$.

Now, from (\ref{eq:ProbAiPrimeGivenAi}) and (\ref{eq:ProbAi1GivenAi}), along with
the fact $\P_{\omega}(A_0) = 1$, we conclude that
\begin{align*}
\P_{\omega}(A_i')
& =  \bigg(\prod_{j = 1}^i \P_{\omega}(A_j|A_{j-1}) \bigg) \cdot \P_{\omega}(A_i'|A_i)
= \left\{ \begin{array}{l}
(q\eta)^i(1-q)~, \mbox{ for } 0 \leq i \leq R-1\\
(q\eta)^R (p\eta)^{i - R} (1-p)~, \mbox{ for } i \geq R.
\end{array} \right.
\end{align*}
So,
\begin{align*}
\eta = \P_{\omega}(T_{-1} < \infty) = \sum_{i = 0}^{\infty} \P_{\omega}(A_i')
= (1-q) \frac{1 - (q \eta)^R}{1 - q \eta} + (1-p) \frac{(q \eta)^R}{1 - p \eta}.
\end{align*}
For $0 < \eta < 1$, this condition is equivalent to $P(\eta) = 0$. \\

\noindent
\emph{Proof of (iii):}
From point (ii) above we know that $\eta\in(0,1)$ is a root of the polynomial $P(t)$. We now show
that there cannot be any other roots in $(0,1)$. Observe that $t=1$ is a root of $P$ and that $P$ factors as
$$
P(t)=(1-t)\Big(1-q+(pq-p-q)t+pqt^2-(p-q)q^Rt^R\Big)\equiv(1-t)Q(t).
$$
Thus, we need to show that the only root of $Q$ in $(0,1)$ is $\eta$.
Observe that $\frac1q>1$ is a root of $Q$.  For $R \geq 3$, we have
$$
Q''(t) = 2pq-(p-q)q^R R(R-1)t^{R-2}.
$$
So, if $R \geq 3$ and $q\ge p$ then $Q$ is convex and can have at most two
real roots.  Also, if $R \in \{1,2\}$ then $Q$ is quadratic and, thus, has at most
two real roots. In either case, this completes the proof.

Now assume that $R \geq 3$ and $q<p$. In this case, $Q''$ has one real root. Thus, $Q'$ can
have no more than two real roots. Let $t^+$ denote the largest root of $Q$ in $(0,1)$. We will
show below that $Q(1)<0$. Using this along with the facts that $Q(t^+)=Q(\frac1q)=0$ and $Q(t)<0$
for sufficiently large $t$, it follows that $Q'$ has two roots in $(t^+,\infty)$. If there were
another root $t^-\in(0,1)$ of $Q$, then $Q'$ would have to have a root in $(t^-,t^+)$,
but this is impossible since $Q'$ cannot have more than two real roots. It remains
to show that $Q(1)<0$.

We have
\begin{equation}\label{Q(1)}
Q(1)=1-2q+q^{R+1}-p(1-2q+q^R).
\end{equation}
Since we are assuming that $\alpha>\frac12$, it follows from Proposition 1 that
$p>\frac{1-2q+q^{R+1}}{1-2q+q^R}\equiv p_0$, if $1-2q+q^{R+1}>0$. On the other hand, if
$1-2q+q^{R+1}\le0$, then $p\in(0,1)$ is unrestricted. In the former case, it follows from
\eqref{Q(1)} that for any $q$, $Q(1)<1-2q+q^{R+1}-p_0(1-2q+q^R)=0$. In the latter case:
\begin{itemize}
\item If $1 - 2q + q^R = 0$, then \eqref{Q(1)} implies $Q(1) = 1 - 2q + q^{R+1} < 0$. 
\item If $1 - 2q + q^R > 0$, then \eqref{Q(1)} implies $Q(1) < 1 - 2q + q^{R+1} \leq 0$, for all $p \in (0,1)$. 
\item If $1 - 2q + q^R < 0$, then \eqref{Q(1)} implies $Q(1) < 1 - 2q + q^{R+1} - (1 - 2q + q^R) < 0$, for all $p \in (0,1)$. 
\end{itemize}
\end{proof}

\subsection{Proof of Theorem \ref{thm:R1Speed}}
\label{subsec:SpeedWithR1}

Unlike the proof of Theorem \ref{thm:L1Speed} for the speed with $L = 1$, the proof of Theorem
\ref{thm:R1Speed} for the speed with $R=1$ does not rely on the implicit characterization of the speed
given by Theorem \ref{thm:BallisticityWhenNonCritical} in terms of $\gamma$. Instead, the proof is based
on a direct method for estimating the hitting times $T_x$ for large $x$.

\begin{proof}[Proof of Theorem \ref{thm:R1Speed}]
For $0 \leq i \leq L-1$, we define $a_i$ to be the expected hitting time of site $1$, starting from site $0$,
in an initial environment with all sites $x < 0$ in the $(p,0)$ configuration and site $0$ in the $(p,i)$
configuration. Also, we define $a_L$ to be the expected hitting time of site $1$, starting from site $0$,
in an initial environment with all sites $x < 0$ in the $(p,0)$ configuration and site $0$ in the $(q,0)$
configuration.
\begin{align*}
a_i = \E_{\omega^{(i)}}(T_1) ~,~  0 \leq i \leq L,
\end{align*}
where the environments $\omega^{(i)}$ satisfy:
\begin{align*}
\omega^{(i)}(x) & = (p,0) ~,~ x < 0 ~\mbox{ and }~ 0 \leq i \leq L. \\
\omega^{(i)}(0) &= (p,i) ~,~0 \leq i \leq L-1. \\
\omega^{(L)}(0) & = (q,0).
\end{align*}

The proof proceeds in two steps. First we set up a linear system of equations for the $a_i$'s,
which can be solved to obtain the desired speed formula in the case that the initial environment $\omega$
satisfies $\omega(x) = (p,0)$, for all $x < 0$. Then, using this result, we show that the same speed formula
holds in the general case. \\

\noindent
\emph{Case (1)}: $\omega(x) = (p,0)$, for all $x < 0$. \\

Since $\alpha > 1/2$, $T_x$ is a.s. finite for each $x > 0$, and we define
$\Delta_x$, $x \geq 0$, by
\begin{align*}
\Delta_x = T_{x+1} - T_x.
\end{align*}
The key observation is that because $R = 1$ and the random walk starts at $X_0 = 0$ in an
environment $\omega$ satisfying $\omega(x) = (p,0)$, for all $x < 0$, we have
\begin{align*}
\omega_n(x) = (p,0), ~\mbox{ for each } n \geq 0 \mbox{ and } x < X_n.
\end{align*}
That is, the environment to the left of the current position of the random walk always consists entirely
of sites in the $(p,0)$ configuration. Applying this fact at the random time $T_x$ it follows that,
for each $x > 0$, $\Delta_x$ is independent of $\Delta_0,...,\Delta_{x-1}$ and has distribution:
\begin{align*}
\P_{\omega}(\Delta_x = k) & = \P_{\omega^{(i)}}(T_1 = k) ~, \mbox{ if } \omega(x) = (p,i) ~,~ 0 \leq i \leq L-1. \\
\P_{\omega}(\Delta_x = k) & = \P_{\omega^{(L)}}(T_1 = k) ~, \mbox{ if } \omega(x) = (q,0).
\end{align*}
Thus, defining
\begin{align*}
A_i^x & = \{ 0 \leq y \leq x-1 : \omega(x) = (p,i)\} ~,~0 \leq i \leq L-1 \\
A_L^x & = \{ 0 \leq y \leq x-1: \omega(x) = (q,0)\}
\end{align*}
and applying the strong law of large numbers for the i.i.d. random variables
$\{\Delta_y: \omega(y) = (p,i)\}$ and $\{\Delta_y: \omega(y) = (q,0)\}$ we have that $\P_{\omega}$ a.s.
\begin{align}
\label{eq:SLLNDecompositionTxx}
\lim_{x \to \infty} T_x/x
= \lim_{x \to \infty} \frac{1}{x} \sum_{i = 0}^L  \sum_{y \in A_i^x} \Delta_y
= \lim_{x \to \infty} \sum_{i = 0}^L \frac{|A_i^x|}{x} \left(  \frac{1}{|A_i^x|} \sum_{y \in A_i^x} \Delta_y \right)
= \sum_{i = 0}^L d_i a_i. \footnotemark{}
\end{align}
\footnotetext{Of course, in order to apply the strong law to conclude that
$\lim_{x \to \infty} \frac{1}{|A_i^x|} \sum_{y \in A_i^x} \Delta_y = a_i$, we need
$|A_i^x| \rightarrow \infty$. However, if $|A_i^x| \not\rightarrow \infty$, for some $i$,
then $d_i = 0$. So, $\lim_{x \to \infty} \frac{1}{x} \sum_{y \in A_i^x} \Delta_y = 0 = d_i a_i$,
and (\ref{eq:SLLNDecompositionTxx}) still holds.}
So, by Lemma \ref{lem:HittingTimesVersusSpeed},
\begin{align}
\label{eq:SpeedGivenByHittingTimes}
\lim_{n \to \infty} \frac{X_n}{n} = \frac{1}{\sum_{i=0}^L d_i a_i} ~,~ \P_{\omega} \mbox{ a.s. }
\end{align}

Now, by conditioning on the first step of the walk it is easy to see that
the following relations between the $a_i$'s hold:
\begin{align}
\label{eq:LinearSystemExpectedHittingTimes}
a_i & = p \cdot 1 ~+~ (1-p) \cdot (1 + a_0 + a_{i+1}) ~,~ 0 \leq i \leq L-1. \nonumber \\
a_L & = q \cdot 1 ~+~ (1-q) \cdot (1 + a_0 + a_L).
\end{align}
One possible solution to the system (\ref{eq:LinearSystemExpectedHittingTimes}) is $a_0 = a_1 = ... = a_L = \infty$.
However, by (\ref{eq:SpeedGivenByHittingTimes}), this implies $X_n/n \rightarrow 0$, $\P_{\omega}$ a.s., which
contradicts Theorem \ref{thm:BallisticityWhenNonCritical}. Also, if $a_j = \infty$, for any $j$, then to satisfy
(\ref{eq:LinearSystemExpectedHittingTimes}) we must have $a_i = \infty$, for all $i$, which, as just shown, cannot
happen. Over the real numbers the system (\ref{eq:LinearSystemExpectedHittingTimes}) has a unique solution
given by (\ref{eq:Def_a0}) and (\ref{eq:Def_ai}). This is shown in Appendix \ref{subsec:ExpectedHittingTimes}. \\

\noindent
\emph{Case (2)}: General Case\\

Fix any initial environment $\omega$ such that the limiting right densities $d_i$ exist,
and let $s = 1/ (\sum_{i=0}^L d_i a_i)$. Also, for an arbitrary environment $\omega'$,
let $\tau$ denote the last hitting time of site $0$ (which is a.s. finite by Theorem
\ref{thm:RightLeftTransienceCutoff}).

We observe that:

\begin{enumerate}

\item For any $\omega'$,
\begin{align*}
\P_{\omega'}(\tau = 0) = \P_{\omega''}(\tau = 0) > 0,
\end{align*}
by Corollary \ref{cor:PrAnPlusGreaterEqualBeta}, where
$\omega''$ is the environment defined by
\begin{align}
\label{eq:DefOmegaDoublePrime}
\omega''(x) = \omega'(x) ~,~ x \geq 0 ~~\mbox{ and }~~\omega''(x) = (p,0) ~,~ x < 0.
\end{align}

\item For any environment $\omega'$ with $\P_{\omega}(\omega_{\tau} = \omega') > 0$, we have
\begin{align*}
\P_{\omega}(X_n/n \rightarrow s | \omega_{\tau} = \omega')
= \P_{\omega'}(X_n/n \rightarrow s| \tau = 0)
= \P_{\omega''}(X_n/n \rightarrow s| \tau = 0),
\end{align*}
where $\omega''$ is defined by (\ref{eq:DefOmegaDoublePrime}).

\item  For any environment $\omega'$ with $\P_{\omega}(\omega_{\tau} = \omega') > 0$,
$\omega'(x) = \omega(x)$, for all but finitely many $x$. So, the limiting right densities $d_i'$  of
states in each configuration for the environment $\omega'$ are the same as the limiting
right densities $d_i$ for the initial environment $\omega$.

\end{enumerate}
It follows from these three observations and the result for Case (1) that, for any
environment $\omega'$ with $\P_{\omega}(\omega_{\tau} = \omega') > 0$,
\begin{align*}
\P_{\omega}(X_n/n \rightarrow s | \omega_{\tau} = \omega')
= \P_{\omega''}(X_n/n \rightarrow s| \tau = 0)
= \P_{\omega''}(X_n/n \rightarrow s)
= 1.
\end{align*}
Hence, $X_n/n \rightarrow s$, $\P_{\omega}$ a.s.

\end{proof}

\section{The Critical Case}
\label{sec:CriticalCase}

Here we analyze the transience/recurrence properties of the random walk $(X_n)$ in the critical case
$\alpha = 1/2$, proving Theorems \ref{recurrpossible}--\ref{LR2crit}. We begin in section
\ref{subsec:TransienceRecurrenceMConN0} with an important lemma for transience/recurrence of
Markov chains on $\N_0$. Then, in section \ref{subsec:StepDistributionZxChain} we establish a framework
relating the right jumps Markov chain $(Z_x)$ to the setup of this lemma. Using this framework, Theorem
\ref{recurrpossible} is proved in section \ref{subsec:ProofThmRecurrPossible}, Theorem
\ref{L1andEnvironment} in section \ref{subsec:ProofThmL1andEnvironment} and Theorem \ref{LR2crit} in section
\ref{subsec:ProofThmLR2crit}. Theorem \ref{RorL1} is proved in section \ref{subsec:ProofThmRorL1}, using other methods.

\subsection{Transience and Recurrence for Markov Chains on $\N_0$}
\label{subsec:TransienceRecurrenceMConN0}

Let $(\ZM_x)_{x \geq 0}$ be a time-homogenous Markov chain 
on the state space $\N_0 = \{0,1,2,...\}$ with step distribution $U(n)$. That is,
\begin{align*}
\P(\ZM_{x+1} = m | \ZM_x = n) = \P( U(n) = m ) ~, ~~ n,m \geq 0.
\end{align*}
We will say that the chain $(\ZM_x)$ is \emph{irreducible and aperiodic with the exception of state $0$} if
$0$ is an absorbing state, but it is possible to redefine the transition probabilities from $0$ to make the chain irreducible and aperiodic. 
Also, we will say that the chain $(\ZM_x)$ is \emph{irreducible and aperiodic with the possible exception of state $0$} if it is either 
irreducible and aperiodic or irreducible and aperiodic with the exception of state $0$. Finally, we will say that the step distribution 
$U(n)$ is \emph{well concentrated} if
\begin{equation}\label{def:mu}
\mu \equiv \lim_{n \to \infty} \E(U(n))/n
\end{equation}
exists and there exist constants $C, c > 0$ and $N \in \N$ such that:
\begin{align}
\label{eq:UnEpsilonSquaredBound}
& \P\left(|U(n) - \mu n| > \epsilon n\right) \leq C e^{-c \epsilon^2 n} , \mbox{ for $0 < \epsilon \leq 1$ and  $n \geq N$.} \\
\label{eq:UnEpsilonBound}
& \P\left(|U(n) - \mu n| > \epsilon n\right) \leq C e^{-c \epsilon n} , \mbox{ for $\epsilon \geq 1$ and  $n \geq N$.}
\end{align}
In this case, we define also the quantities $\rho(n)$, $\nu(n)$, and $\theta(n)$ by
\begin{equation}\label{def:rhonutheta}
\rho(n) = \E(U(n) - \mu n),~ \  \nu(n) = \E((U(n) - \mu n )^2)/n,~ \ \theta(n) = 2 \rho(n)/\nu(n).
\end{equation}

The following lemma is essentially Theorem 1.3 from \cite{Kozma2013}\footnotemark{}.

\begin{Lem}
\label{lem:ThetaLessGreater1}
Let $(\ZM_x)_{x \geq 0}$ be a time-homogenous Markov chain on state space $\N_0$, 
which is irreducible and aperiodic with the possible exception of state $0$ and has well 
concentrated step distribution $U(n)$. Also, denote by $\P_k$ the probability measure for 
the chain $(\ZM_x)$ started from $\ZM_0 = k$. Then the following hold for any initial state $k \geq 1$.
\begin{itemize}
\item[(i)] If $\mu < 1$, then $\P_k(\ZM_x > 0, \forall x \geq 0) = 0$.
\item[(ii)] If $\mu > 1$, then $\P_k(\ZM_x > 0, \forall x \geq 0) > 0$.
\item[(iii)] If $\mu = 1$, $\liminf_{n \to \infty} \nu(n) > 0$, and $\theta(n) < 1 + \frac{1}{\ln(n)} - \frac{a(n)}{n^{1/2}}$
for sufficiently large $n$, for some function $a(n) \rightarrow \infty$, then $\P_k(\ZM_x > 0, \forall x \geq 0) = 0$.
\item[(iv)] If $\mu = 1$, $\liminf_{n \to \infty} \nu(n) > 0$, and $\theta(n) > 1 + \frac{2}{\ln(n)} + \frac{a(n)}{n^{1/2}}$
for sufficiently large $n$, for some function $a(n) \rightarrow \infty$, then $\P_k(\ZM_x > 0, \forall x \geq 0) > 0$.
\end{itemize}
\end{Lem}

\footnotetext{
There are three small differences. First, in Theorem 1.3 of \cite{Kozma2013} the chain $(\ZM_x)$ is required to be 
truly irreducible and aperiodic, without the possible exception of state $0$. Second, instead of \eqref{eq:UnEpsilonSquaredBound}
and \eqref{eq:UnEpsilonBound} the following somewhat stronger concentration condition for $U(n)$ is assumed:
There exist $c > 0$ and $N \in \N$ such that
\begin{align}
\label{eq:UnGeneralEpsilonSquareBound}
\P\left(|U(n) - \mu n| > \epsilon n\right) \leq 2 e^{-c \epsilon^2 n}, \mbox{ for all $\epsilon > 0$ and $n \geq N$}.
\end{align}
Finally, there is no assumption that $\liminf_{n \to \infty} \nu(n) > 0$ for cases (iii) and (iv).

Allowing the possible exception of state $0$ in the irreducible and aperiodic hypothesis clearly has no effect,
since the probability of ever hitting state $0$, starting from a state $k \geq 1$, depends only on the transition probabilities from 
the nonzero states.  Also, the concentration condition (\ref{eq:UnGeneralEpsilonSquareBound}) is used in \cite{Kozma2013} only 
to bound the error terms in certain Taylor series expansions, and these estimates remain valid if (\ref{eq:UnEpsilonSquaredBound})
and (\ref{eq:UnEpsilonBound}) hold instead, so there is no issue with using the weaker concentration condition.
However, the proof of cases (iii) and (iv) given in \cite{Kozma2013} actually works as stated only if
$\liminf_{n \to \infty} \nu(n) > 0$, so we require this condition also in our statement. }

\subsection{Step Distribution of the Right Jumps Markov chain}
\label{subsec:StepDistributionZxChain}

By definition (\ref{eq:DefRightJumpsMC}) for the right jumps Markov chain $(Z_x)_{x \geq 0}$,
\begin{align*}
\P(Z_{x+1} = m |Z_x = n) = \P(U(n,x+1) = m) ~,~n,m,x \geq 0,
\end{align*}
where $U(n,x)$ is the (random) number of right jumps in the sequence $(J_k^x)_{k \in \N}$ before the
time of the $n$-th left jump:
\begin{align}
\label{eq:UofnDistribution}
U(n,x)=\inf \Big\{\ell \geq 0: \sum_{k=1}^{\ell} \indicator\{J_k^x = -1\} =n\Big\}-n.
\end{align}
If the initial environment $\omega(x)$ is constant for all $x \geq 0$, then the distribution of the jump sequence
$(J_k^x)_{k \in \N}$ is the same for all $x \geq 0$, so the distribution of $U(n,x)$ is also the same for all $x \geq 0$.
In this case, the right jumps chain $(Z_x)_{x \geq 0}$ is time-homogeneous (where $x$ is the time variable) with
step distribution $U(n) = U(n,x)$. It is also irreducible and aperiodic with the exception of state $0$. For example,
redefining the transition probabilities from state $0$ as $\P(Z_{x+1} = 1|Z_x = 0) = 1$
would make the chain irreducible and aperiodic. 

For the remainder of this section we assume $\omega(x) = \omega(0)$, for all $x \geq 0$.
For our analysis of the step distribution $U(n)$ we fix an arbitrary site $x \geq 0$ and 
decompose $U(n) = U(n,x)$ as
\begin{align}
\label{eq:UnEqualSumGammaj}
U(n)=\sum_{j=1}^n\Gamma_j,
\end{align}
where $\Gamma_j$ is the number of right jumps in the sequence $(J_k^x)_{k \in \N}$
between the $(j-1)$-th and $j$-th left jumps. That is, $\Gamma_j = k_ j - k_{j-1} - 1$,
where $k_0 = 0$ and, for $j \geq 1$, $k_j = \inf \{k > k_{j-1}: J_{k}^x = -1 \}$.

We think of the $(\Gamma_j)_{j=1}^n$ as the values obtained in $n$ ``sessions,'' and
denote by $\omega^j = \omega^j(x)$ the configuration at site $x$ at the beginning of the $j$-th 
session. Thus, $\omega^1 = \omega(x)$ and, for $j \geq 2$, $\omega^j = Y_{k_{j-1}+1}^x$ 
is the configuration at site $x$ immediately after the $(j-1)$-th left jump in the sequence 
$(J_k^x)_{k \in \N}$. It is straightforward to see that conditioned on $\omega^j$, $\Gamma_j$ 
is independent of $\Gamma_1,...,\Gamma_{j-1}$ and has the following distribution:
\begin{align}
\label{eq:GammajDist}
\Gamma_j & \sim S_i, \mbox{ if } \omega^j = (q,i) , \mbox{ for some } 0 \leq i \leq R-1, \mbox{ and } \nonumber \\
\Gamma_j & \sim S_R, \mbox{ if } \omega^j = (p,i), \mbox{ for some } 0 \leq i \leq L-1,
\end{align}
where $S_0, ... ,S_R$ are random variables with law
\begin{equation}
\label{Sidist}
\P(S_i=k)=\begin{cases} q^k(1-q),\ 0 \leq k \leq R-i-1;\\
q^{R-i}p^{k-(R-i)}(1-p),\ k\ge R-i.\end{cases}
\end{equation}
In particular, $S_R$ is a standard geometric random variable with parameter $1-p$.

Now, the configuration $\omega^{j+1}$ at the beginning of the next session is determined entirely by $\omega^j$
and $\Gamma_j$. More precisely, $\omega^{j+1}$ is the (deterministic) configuration obtained by jumping right
$\Gamma_j$ times from site $x$, starting in configuration $\omega^j$, and then jumping left once.
Thus, assuming that $L\ge2$:
\begin{align}
\label{eq:omegajplus1_determined}
& \mbox{ If } \omega^j = (q,i), 0 \leq i \leq R-1, \mbox{ then }
\omega^{j+1} = \begin{cases} (p,1) ~, \mbox{ if } \Gamma_j \geq R-i; \nonumber\\
(q,0)~, \mbox{ if } \Gamma_j < R-i.\end{cases} \\
& \mbox{ If } \omega^j = (p,i), 0 \leq i \leq L-2, \mbox{ then }
\omega^{j+1} = \begin{cases} (p,1) ~, \mbox{ if } \Gamma_j \geq 1; \nonumber\\
(p,i+1)~, \mbox{ if } \Gamma_j = 0.\end{cases} \\
& \mbox{ If } \omega^j = (p,L-1), \mbox{ then }
\omega^{j+1} = \begin{cases} (p,1) ~, \mbox{ if } \Gamma_j \geq 1;\\
(q,0)~, \mbox{ if } \Gamma_j = 0.\end{cases}
\end{align}
If $L=1$ then the configuration $\omega^j$ at the beginning of each of the right jumps sessions 
after the first is always $(q,0)$, since the configuration at site $x$ immediately after a left jump is $(q,0)$.

From \eqref{eq:GammajDist}-\eqref{eq:omegajplus1_determined} it follows that, for any $L \geq 2$, the sequence of
configurations $(\omega^j)_{j=1}^n$ is a Markov chain with (initial state $\omega(x)$ and)
transition matrix $\Ah$ given by
\begin{equation}\label{Ahatmatrix}
\begin{aligned}
& \Ah_{(q,i),(q,0)} =1-q^{R-i} ~,~ \Ah_{(q,i),(p,1)}=q^{R-i}~\mbox{ for }~ 0 \leq i \leq R-1; \\
& \Ah_{(p,i),(p,1)} =p ~,~ \Ah_{(p,i),(p,i+1)}=1-p ~\mbox{ for }~ 0 \leq i \leq L-2; \\
& \Ah_{(p,L-1),(p,1)} =p ~,~ \Ah_{(p,L-1),(q,0)}=1-p.
\end{aligned}
\end{equation}
In the case $L=1$, $(\omega^j)_{j=1}^n$ is still a Markov chain, but it is degenerate.
The transition matrix $\Ah$ has $\Ah_{\lambda,(q,0)} = 1$, for all $\lambda \in \Lambda$.

In either case, the transition matrix $\Ah$ is indecomposable, and the $L$ states $\Lambda_0 \equiv \{(q,0),(p,1),...,(p,L-1)\}$
constitute a closed, irreducible set of states. We denote by $A$ the corresponding transition matrix
obtained from $\Ah$ by restricting to these $L$ states, and by $\psi$ the unique invariant measure for $A$
(for $L=1$, $A = \psi = 1$). Also, we denote by $e_{(p,i)}$ the unit $L$-vector with a $1$ in the position of state
$(p,i)$, and by $e_{(q,0)}$ the unit $L$-vector with a $1$ in the position of $(q,0)$. Finally, we let $E$
denote the $L$-vector with components, $E_{(q,0)}=\E(S_0)$ and $E_{(p,i)}=\E(S_R)$, for $i\in\{1,\ldots, L-1\}$. \\

\noindent
\emph{Basic Lemmas} \vspace{1.5 mm} \\
The following lemmas characterize some key properties of the step distribution $U(n) = \sum_{j=1}^n \Gamma_j$ 
in the critical case, $\alpha = 1/2$. Proofs are deferred to Appendix \ref{sec:ProofOfUnLemmas},
but in all three cases use the underlying Markov chain $(\omega^j)_{j=1}^n$.

\begin{Lem}
\label{lem:muEqual1}
If $\alpha = 1/2$, then
\begin{align}
\label{eq:muEqual1}
\mu \equiv \lim_{n \to \infty} \frac{\E(U(n))}{n} = \left<\psi, E \right> = 1,
\end{align}
where $\left<\cdot, \cdot \right>$ denotes the standard inner product of two real vectors. 
\end{Lem}

\begin{Lem}
\label{lem:ConcentrationEstimate}
If $\alpha = 1/2$, then the step distribution $U(n)$ is well concentrated.
\end{Lem} 

\begin{Lem}
\label{lem:nuLowerBound} 
If $\alpha = 1/2$, then $\liminf_{n \to \infty} \nu(n) > 0$, where $\nu(n)$ is given by \eqref{def:rhonutheta}.  
\end{Lem}

\noindent \bf Remark.\rm\ 
The assumption $\alpha = 1/2$ is actually not necessary for the conclusions of Lemmas \ref{lem:ConcentrationEstimate} 
and \ref{lem:nuLowerBound} to hold, but it simplifies the writing of the proofs slightly and is the only case where we will apply them. \\

\noindent
\emph{Setup for the Proofs of Theorems \ref{recurrpossible}, \ref{L1andEnvironment}, and \ref{LR2crit}} \vspace{1.5 mm} \\
For the proofs of Theorems \ref{recurrpossible}, \ref{L1andEnvironment}, and \ref{LR2crit} below we will adopt the 
framework given here for analyzing the step distribution $U(n)$ of the right jumps Markov chain, without further mention,
whenever the initial environment $\omega$ is constant over all $x \geq 0$. In this case, since $\alpha = 1/2$ for all three 
of these theorems, we know $\mu = 1$ by Lemma \ref{lem:muEqual1}. Thus, \eqref{def:rhonutheta} becomes 
\begin{equation}\label{def:rhonuthetaalphahalf}
\rho(n) = \E(U(n) - n),~ \  \nu(n) = \E((U(n) - n)^2)/n,~ \ \theta(n) = 2 \rho(n)/\nu(n)
\end{equation}
and it follows from Lemmas \ref{lem:ProbSurvivalZxProbTransienceXn}, \ref{lem:ThetaLessGreater1}, 
\ref{lem:ConcentrationEstimate}, and \ref{lem:nuLowerBound} that
\begin{align}
\label{eq:ThetaTransienceRecurrenceCondition}
& \P_{\omega}(X_n \rightarrow \infty) = 0 , ~\mbox{ if }~ \theta(n) \leq 1 + O(\frac1n) ~\mbox{ and } \nonumber \\
& \P_{\omega}(X_n \rightarrow \infty) > 0 , ~\mbox{ if } \lim_{n \to \infty} \theta(n) > 1.
\end{align}
In all cases one of these two possibilities for $\theta(n)$ will occur. 

\subsection{Proof of Theorem \ref{recurrpossible}}
\label{subsec:ProofThmRecurrPossible}

\begin{proof}[Proof of Theorem \ref{recurrpossible}]
If the initial environment $\omega$ satisfies $\omega(x) = \lambda \in \Lambda_0$ for all $x \geq 0$, 
then the Markov chain representation of section \ref{subsec:StepDistributionZxChain} 
and Lemma \ref{lem:muEqual1} give 
\begin{align*}
\E&(\Gamma_j)
= \sum_{\lambda'} \P(\omega^j = \lambda'|\omega^1 = \lambda) \cdot \E(\Gamma_j|\omega^j = \lambda') \\
& = \left<e_{\lambda} A^{j-1}, E \right> = \left<\psi, E \right> + \left<(e_{\lambda} A^{j-1} - \psi), E \right> = 1 + O(a^j),
\end{align*}
for some $0 < a < 1$, which depends on the matrix $A$. Thus, in this case, for all $n \geq j$ 
we have from \eqref{def:rhonuthetaalphahalf}
\begin{align*}
\rho(n) = \sum_{i = 1}^n \E(\Gamma_i) - n 
= \bigg(\sum_{i = 1}^{j} \E(\Gamma_i) - j \bigg) + O(a^j) = \rho(j) + O(a^j). 
\end{align*} 
Also, by Lemma \ref{lem:nuLowerBound}, we know that there exists some $\epsilon > 0$, which can be chosen uniformly over 
$\lambda \in \Lambda_0$, such that $\liminf_{n \to \infty} \nu(n) \geq \epsilon$ if $\omega(x) = \lambda$, for $x \geq 0$. 
Combining these observations we see that there exists some $n_0 \in \N$ satisfying
\begin{align}
\label{eq:WhenBiggern0}
\nu^{(\lambda)}(n) \geq \epsilon/2 ~\mbox{ and }~ \rho^{(\lambda)}(n) - \rho^{(\lambda)}(n_0) \leq \epsilon/4, 
\mbox{ for all $\lambda \in \Lambda_0$ and $n \geq n_0$}, 
\end{align}
where $\rho^{(\lambda)}(n)$ and $\nu^{(\lambda)}(n)$ are the quantities $\rho(n)$ and $\nu(n)$ 
with initial environment $\omega(x) = \lambda$, $x \geq 0$. 

Now, for any fixed $j$, Lemma \ref{lem:muEqual1} implies
\begin{align*}
& j = \sum_{i=1}^j \left<\psi, E \right> 
= \sum_{i=1}^j \left<\psi A^{j-1}, E \right>
= \sum_{i=1}^j \sum_{\lambda \in \Lambda_0} \psi_{\lambda} \left<e_{\lambda} A^{j-1}, E \right> \\
& =  \sum_{i=1}^j \sum_{\lambda \in \Lambda_0} \psi_{\lambda} \cdot \E(\Gamma_j^{(\lambda)}) 
= \sum_{\lambda \in \Lambda_0} \psi_{\lambda} \cdot \E(U^{(\lambda)}(j)),
\end{align*} 
where $\Gamma_j^{(\lambda)}$ and $U^{(\lambda)}(j)$ are the random variables  
$\Gamma_j$ and $U(j)$, with initial environment $\omega(x) = \lambda$, $x \geq 0$. 
Thus, for any fixed $j$,
\begin{align*}
\sum_{\lambda \in \Lambda_0} \psi_{\lambda} \cdot \big[ \E(U^{(\lambda)}(j)) - j\big]
= \sum_{\lambda \in \Lambda_0} \psi_{\lambda} \cdot \rho^{(\lambda)}(j) = 0,
\end{align*}  
so there exists some $\lambda_j \in \Lambda_0$ such that
\begin{align}
\label{eq:rholambdaj}
\rho^{(\lambda_j)}(j) \leq 0. 
\end{align}

Define $\lambda^* = \lambda_{n_0}$.  Then, by \eqref{eq:WhenBiggern0} and \eqref{eq:rholambdaj},
$\rho^{(\lambda^*)}(n) \leq \rho^{(\lambda^*)}(n_0) + \epsilon/4 = \rho^{(\lambda_{n_0})}(n_0) + \epsilon/4 \leq \epsilon/4$, 
for $n \geq n_0$. So, by \eqref{eq:WhenBiggern0}, 
\begin{align*}
\theta^{(\lambda^*)}(n) 
= \frac{2 \rho^{(\lambda^*)}(n)}{\nu^{(\lambda^*)}(n)} 
\leq \frac{2 \cdot (\epsilon/4)}{\epsilon/2} = 1, \mbox{ for } n \geq n_0. 
\end{align*} 
It follows, from \eqref{eq:ThetaTransienceRecurrenceCondition}, that $\P_{\omega}(X_n \rightarrow \infty) = 0$ for any initial environment 
$\omega$ with $\omega(x) = \lambda^*$, for all $x \geq 0$. Thus, also, $\P_{\omega}(X_n \rightarrow \infty) = 0$ for any initial environment 
$\omega$ which is equal to $\lambda^*$ in a neighborhood of $+\infty$. An analogous argument shows that there exists some $\lambda_{*}$ 
such that $\P_{\omega}(X_n \rightarrow -\infty) = 0$, for any initial environment $\omega$ which is equal to $\lambda_{*}$ in a neighborhood 
of $-\infty$. So, by part (ii) of Lemma \ref{lem:SimpleTransConditions}, the random walk $(X_n)$ is $\P_{\omega}$ a.s. recurrent for any initial 
environment $\omega$ which is equal to $\lambda^*$ in a neighborhood of $+\infty$ and equal to $\lambda_{*}$ in a neighborhood 
of $-\infty$. By Lemma \ref{lem:ComparisonOfEnvironments}, and symmetry considerations, we may take $\lambda^* = (q,0)$ and 
$\lambda_{*} = (p,0)$ in the case of positive feedback, $q < p$. 
\end{proof}

\subsection{Proof of Theorem \ref{L1andEnvironment}}
\label{subsec:ProofThmL1andEnvironment}

For notational convenience in the proof of Theorem \ref{L1andEnvironment} we define
\begin{align*}
\lambda_0 = (q,0),...,\lambda_{R-1} = (q,R-1), \lambda_R = (p,0).
\end{align*}
As discussed above, in the case $L=1$ the transition matrix $\Ah$ is degenerate with $\Ah_{\lambda,(q,0)} = 1$, 
for all $\lambda \in \Lambda$. Thus, in this case, $\omega^j = (q,0)$, for all $j \geq 2$ (independent of the values of the 
$\Gamma_j$'s). Also, with $L = 1$, $\psi$ is simply the length-$1$ vector $1$ and $E$ is simply the length-$1$ 
vector $\E(S_0)$. The following facts are immediate from this.
\begin{enumerate}
\item If $L = 1$ and $\omega(x) = \lambda_i$, $x \geq 0$, then
\begin{align}
\label{eq:L1GammajDist}
\Gamma_1, \Gamma_2,... \mbox{ are independent with }
\Gamma_1 \sim S_i  ~\mbox{ and }~ \Gamma_j \sim S_0,~ j \geq 2.
\end{align}
\item If $L=1$ and $\alpha = 1/2$ then, by Lemma \ref{lem:muEqual1},
\begin{align}
\label{eq:ExS0is1}
\E(S_0) = \left<\psi, E \right> = 1.
\end{align}
\end{enumerate}
Using these facts we now prove Theorem \ref{L1andEnvironment}.

\begin{proof}[Proof of Theorem \ref{L1andEnvironment}]

By assumption $\alpha = 1/2$ and $L=1$, and the initial environment is a constant in a neighborhood of $-\infty$ in
the case of negative feedback, $p < q$. Thus, by Theorem \ref{RorL1}, the probability of the random walk $(X_n)$
being transient to $-\infty$ is equal to $0$\footnotemark{}. So, by Lemma \ref{lem:SimpleTransConditions},
the probability of being transient to $+\infty$ is either $0$ or $1$, and if it is $0$, the process is recurrent.
Moreover, without loss of generality, clearly we may  assume that the initial environment $\omega$ is constant for all
$x \geq 0$ (rather than only in a neighborhood of $+\infty$). Thus, it suffices to show the following to 
establish the transience/recurrence claims in the theorem:
\footnotetext{Theorem \ref{RorL1} is not proved till later in section \ref{subsec:ProofThmRorL1},
but the proof is independent of the proof of this theorem.} 
\begin{align}
\label{eq:BulletConditions}
\bullet & \mbox{ If $\omega(x) = \lambda_0$ for all $x \geq 0$, then $\P_{\omega}(X_n \rightarrow \infty) = 0$. } \nonumber \\
\bullet & \mbox{ If $\omega(x) = \lambda_i$ for all $x \geq 0$, $1 \leq i \leq R$, then
$\P_{\omega}(X_n \rightarrow \infty) = 0$ } \nonumber \\ & \mbox{ if and only if $P_{R,i}(q) \geq 0$. }
\end{align}

For the remainder of the proof we assume that
\begin{align}
\label{eq:ConstantEnvironmentToRight}
\omega(x) = \lambda_i,~ x \geq 0,
\end{align}
for some $0 \leq i \leq R$. To determine if there is positive probability of transience to $+\infty$ we will 
calculate $\theta(n) = 2 \rho(n) / \nu(n)$ and apply \eqref{eq:ThetaTransienceRecurrenceCondition}. 

We begin with $\rho(n)$. Since $L=1$ and $\alpha = 1/2$, $\E(S_0) = 1$, by \eqref{eq:ExS0is1}.
Thus, by  \eqref{def:rhonuthetaalphahalf} and \eqref{eq:L1GammajDist},
\begin{align}
\label{eq:rhosimple}
\rho(n) & = \E(S_i) + (n-1) \E(S_0) - n = \E(S_i) - 1.
\end{align}
A direct computation yields
\begin{equation}\label{ExpSi}
\begin{aligned}
\mathbb{E}(S_i)
& = \sum_{k = 0}^{R-i-1} k \cdot q^k(1-q) + \sum_{k = R - i}^{\infty} k \cdot q^{R-i} p^{k-(R-i)}(1-p) \\
& =\frac1{1-q}\big[-(1-q)q^{R-i}(R-i)+(1-q^{R-i})q\big] ~+\\
&~~~~~~ \frac1{1-p}(\frac qp)^{R-i}\big[(1-p)p^{R-i}(R-i)+p^{R-i+1}\big].
\end{aligned}
\end{equation}
Since $L=1$ and $\alpha = 1/2$, Proposition \ref{prop:PropertiesOfAlpha} implies
\begin{equation}
\label{pqcrit}
p=p_0=\frac{1-2q+q^{R+1}}{1-2q+q^R},\ \
1-p=1-p_0=\frac{q^R-q^{R+1}}{1-2q+q^R}.
\end{equation}
Substituting for $p$ above we obtain, after some lengthy simplifications,
$$
\mathbb{E}(S_i)=\frac{1-2q+q^{i+1}}{q^i(1-q)}.
$$
Thus, by \eqref{eq:rhosimple},
\begin{align}
\label{eq:rho}
\rho(n) =\frac{1-2q-q^i+2q^{i+1}}{q^i(1-q)} ~,~ \mbox{ for all } n. 
\end{align}

We now turn to the calculation of $\nu(n)$. From \eqref{def:rhonuthetaalphahalf} we recall that $\nu(n) = \E[(U(n)-n)^2]/n$. 
Using the independence of the random variables $\Gamma_1,...,\Gamma_n$ we have
\begin{equation}
\label{U(n)-n}
\begin{aligned}
&\mathbb{E}\left[(U(n)-n)^2\right]=\mathbb{E}\Big[\Big(\sum_{j=1}^n\Gamma_j-(n-1+\E(S_i))-(1-\E(S_i))\Big)^2\Big]=\\
&\Var(S_i)+(n-1)\Var(S_0)+(1 - \E(S_i))^2 = n[\E(S_0^2)-1]+ O(1).
\end{aligned}
\end{equation}
A tedious computation gives
\begin{equation}\label{expS0squared}
\begin{aligned}
\mathbb{E}(S_0^2)
& = \sum_{k = 0}^{R-1} k^2 \cdot q^k(1-q) + \sum_{k = R}^{\infty} k^2 \cdot q^R p^{k-R}(1-p) \\
& = \frac{q+q^2-q^R\Big(R^2-(2R^2-2R-1)q+(R-1)^2q^2\Big)}{(1-q)^2} ~+\\
& ~~~~~~~~ \frac{q^R\Big(R^2-(2R^2-2R-1)p+(R-1)^2p^2\Big)}{(1-p)^2}.
\end{aligned}
\end{equation}
Substituting for $p$ from \eqref{pqcrit} and doing a lot of  algebra, one eventually finds that
\begin{equation}\label{ES02}
\mathbb{E}(S_0^2)=\frac1{q^R(1-q)^2}\Big[2-8q+8q^2+(2R+1)q^R+(2-6R)q^{R+1}+(4R-5)q^{R+2}\Big].
\end{equation}
Finally, from \eqref{U(n)-n} and \eqref{ES02}, we have
\begin{equation}\label{eq:nu}
\begin{aligned}
&\nu(n)= \frac{\mathbb{E}[(U(n)-n)^2]}n=\mathbb{E}(S_0^2)-1+O(\frac1n)=\\
&\frac1{q^R(1-q)^2}\Big[2-8q+8q^2+2Rq^R+(4-6R)q^{R+1}+(4R-6)q^{R+2}\Big]+O(\frac1n).
\end{aligned}
\end{equation}

Now, combining \eqref{eq:rho} and \eqref{eq:nu} shows that $\theta(n) = \theta + O(\frac1n)$ where
\begin{equation}\label{thetaagain}
\begin{aligned}
&~\theta =\frac{q^{R-i}(1-q)(1-2q-q^i+2q^{i+1})}{1-4q+4q^2+Rq^R+(2-3R)q^{R+1}+(2R-3)q^{R+2}} \\
&= \frac{q^{R-i}-3q^{R-i+1}+2q^{R-i+2}-q^R+3q^{R+1}-2q^{R+2}}{1-4q+4q^2+Rq^R+(2-3R)q^{R+1}+(2R-3)q^{R+2}}.
\end{aligned}
\end{equation}
For $1 \leq i \leq R$, $\theta \leq 1$ is equivalent to $P_{R,i}(q) \geq 0$, and in the case $i = 0$, 
$\theta$ is $0$. Thus,  \eqref{eq:BulletConditions} follows from \eqref{eq:ThetaTransienceRecurrenceCondition}. 

It remains only to show the claims concerning the polynomial $P_{R,R}(q)$. For these
we will use the factored representation $P_{R,R}(q) = q(1-q)^2 \Pt_{R,R}(q)$, where
$\Pt_{R,R}(q)=-1+\sum_{j=1}^{R-3}jq^{j+1}+(2R-1)q^{R-1}$, as in \eqref{eq:FactoredPolyR}.
Since, $P_{R,R}(q)$ and $\Pt_{R,R}(q)$ have the same sign for all $q \in (0,1)$, it suffices to prove
the claims for the polynomial $\Pt_{R,R}(q)$.

Now, clearly, $\Pt_{R,R}$ is increasing, and $\Pt_{R,R}(0)=-1$.
For $R\ge4$, one can rewrite $\Pt_{R,R}$ as
$\Pt_{R,R}(q)=-1+(\frac q{1-q})^2[1-(R-2)q^{R-3}+(R-3)q^{R-2}]+(2R-1)q^{R-1}$.
Using this, we find that  $\Pt_{R,R}(\frac12)=(\frac12)^{R-1}$, for all $R\ge2$.
Consequently, $\Pt_{R,R}$ has a unique root $q_*(R)\in(0,\frac12)$, with
$\Pt_{R,R}(q) < 0$ for $q<q_*(R)$ and $\Pt_{R,R}(q) > 0$ for $q > q_*(R)$.
Furthermore,
\begin{align} \label{PRdec}
&\Pt_{R+1,R+1}(q)-\Pt_{R,R}(q)=(2R+1)q^R-(R+1)q^{R-1}= \nonumber \\
&q^{R-1}\big[(2R+1)q-(R+1)\big]\le
-\frac12q^{R-1} < 0,\ \text{for}\ q\in[0,\frac12].
\end{align}
So, $q_*(R)$ is increasing in $R$. Also, we have
$\Pt_{\infty,\infty}(q)\equiv\lim_{R\to\infty}\Pt_{R,R}(q)=\frac{2q-1}{(1-q)^2}$.
Since the root of $\Pt_{\infty,\infty}$ is at $q=\frac12$, it follows
that $\lim_{R\to\infty}q_*(R)=\frac12$.
\end{proof}

\subsection{Proof of Theorem \ref{LR2crit}}
\label{subsec:ProofThmLR2crit}

For the proof of Theorem \ref{LR2crit} we will need the following lemma. 

\begin{Lem}
\label{lem:MustBeTransientTopmInfinity} 
If the initial environment $\omega$ is constant in a neighborhood of $+\infty$ ($-\infty$) and
$\P_{\omega}(X_n \rightarrow +\infty) > 0$ ($\P_{\omega}(X_n \rightarrow -\infty) > 0$), then
\begin{align*}
\P_{\omega}(X_n \rightarrow +\infty) + \P_{\omega}(X_n \rightarrow -\infty) = 1. 
\end{align*}
\end{Lem} 

\begin{proof}
We will prove the claim for the case of constant initial environment in a neighborhood of $+\infty$;
the other claim follows from symmetry considerations. Thus, we assume the environment $\omega$ 
satisfies $\P_{\omega}(X_n \rightarrow \infty) > 0$ and $\omega(x) = \lambda$, $x \geq N$, for some 
$\lambda \in \Lambda$ and $N \in \N$. Also, we let $\omega'$ be any environment with $\omega'(x) = \lambda$, 
for all $x \geq 0$. Since we assume $\P_{\omega}(X_n \rightarrow \infty) > 0$ it follows 
from part (i) of Lemma \ref{lem:SimpleTransConditions} that $\P_{\omega'}(X_n \rightarrow \infty), \P_{\omega'}(\AM_0^+) > 0$,
and we define $\delta = \P_{\omega'}(\AM_0^+)$. 

If the random walk $(X_n)$ is run starting in $\omega$, then at every time $n$ that it first hits a site 
$x \geq N$ the environment at time $n$ is equal to $\lambda$ at all sites $y \geq x$. Thus,
\begin{align}
\label{eq:DeltaChanceOfEscape}
& \P_{\omega}\big(\AM_n^+ | X_0 = x_0, ..., X_n = x_n \big)
= \P_{\omega'}(\AM_0^+) = \delta, \nonumber \\
& \mbox{for any path $(x_0,...,x_n)$ satisfying $x_0 = 0, x_n \geq N, x_m < x_n$ for $m < n$}.  
\end{align}
We define stopping times $\tau_i, \tau_i'$ and stopping points $z_i$ as follows:
\begin{itemize}
\item $\tau_0 = T_N$, $z_0 = N$.
\item For $i \geq 1$, 
\begin{align*}
& \tau_i' = \inf \{n > \tau_{i-1}: X_n = z_{i-1}\}, \\
& z_i = \sup\{X_n : \tau_{i-1} \leq n < \tau_i' \}, \\
& \tau_i = \inf \{n > \tau_i' : X_n = z_i + 1\}. 
\end{align*}
\end{itemize}
Of course, not all these stopping times are necessarily finite. We think of the times as an ordered list $\tau_0, \tau_1', \tau_1, \tau_2', \tau_2,...$ 
and if any element in this list is equal to $\infty$ then all elements to the right of it are defined to be $\infty$ as well. By construction, the list is 
an increasing sequence of positive integers up till the point of the first $\infty$, and it follows from \eqref{eq:DeltaChanceOfEscape} that,
for any $i \geq 1$,
\begin{align*}
\P_{\omega}(\tau_i' < \infty|\tau_{i-1} < \infty) \leq 1 - \delta.
\end{align*}
Thus, for any $i \geq 1$, 
\begin{align*}
\P_{\omega}(\tau_i < \infty) = \P_{\omega}(\tau_0 < \infty) \cdot 
\prod_{j=1}^i \Big( \P_{\omega} (\tau_j' < \infty|\tau_{j-1} < \infty) \cdot \P_{\omega}(\tau_j < \infty|\tau_j' < \infty) \Big) \leq (1-\delta)^i.
\end{align*}
So, $\P_{\omega}(\tau_i < \infty, \mbox{ for all } i > 0) = 0$. However, by part (ii) of Lemma \ref{lem:SimpleTransConditions}, we also have 
$\P_{\omega}(X_n \not\rightarrow +\infty \mbox{ and } X_n \not\rightarrow -\infty) \leq \P_{\omega}(\tau_i < \infty, \mbox{ for all } i > 0)$. 
Thus, $\P_{\omega}(X_n \rightarrow +\infty) + \P_{\omega}(X_n \rightarrow -\infty) = 1$. 
\end{proof}

\begin{proof}[Proof of Theorem \ref{LR2crit}]

We will show that if $\omega(x)=(p,0)$ for $x \geq 0$, then $\theta(n)$ from \eqref{def:rhonuthetaalphahalf} satisfies 
$\theta(n) \leq 1+O(\frac1n)$, if $q\ge q_1^*$, and $\lim_{n\to\infty}\theta(n)>1$, if $q<q_1^*$. We will also show that 
if $\omega(x)=(p,1)$ for $x \geq 0$ or $\omega(x)=(q,1)$ for $x \geq 0$, then $\theta(n)$ satisfies $\theta(n)\le 1+O(\frac1n)$ 
for all $q\in(0,1)$. Finally, we will show that if $\omega(x)=(q,0)$ for $x \geq 0$, then $\theta(n)$ satisfies 
$\theta(n)\le 1+O(\frac1n)$, if $q\le q_2^*$, and $\lim_{n\to\infty}\theta(n)>1$, if $q>q_2^*$. It follows, 
by \eqref{eq:ThetaTransienceRecurrenceCondition}, that:
\begin{align}
\label{eq:TransiencPossibilitiesConstantEnvironmentsRL2} 
& \bullet \mbox{ If $\omega(x) = (p,0)$ for $x \geq 0$, then $\P_{\omega}(X_n \rightarrow \infty) > 0$ if and only if $q < q_1^*$.} \nonumber \\
& \bullet \mbox{ If $\omega(x) = (q,0)$ for $x \geq 0$, then $\P_{\omega}(X_n \rightarrow \infty) > 0$ if and only if $q > q_2^*$.} \nonumber \\
& \bullet \mbox{ If $\omega(x) = (p,1)$ for $x \geq 0$ or $\omega(x) = (q,1)$ for $x \geq 0$, then $\P_{\omega}(X_n \rightarrow \infty) = 0$.}
\end{align}
Clearly, \eqref{eq:TransiencPossibilitiesConstantEnvironmentsRL2} is still valid if the hypothesis ``$x \geq 0$'' is changed to 
``$x$ in a neighborhood of $+\infty$''. Thus, this will prove the theorem, in light of Lemmas \ref{lem:SimpleTransConditions} 
and \ref{lem:MustBeTransientTopmInfinity}, and the symmetry that holds because $R=L$ and $p = 1-q$, along with 
Lemma \ref{lem:ComparisonOfEnvironments} in the case $q<p$, where the theorem sometimes allows for nonconstant 
environments in a neighborhood of $+\infty$ or $-\infty$. 

By definition, $\theta(n) = 2\rho(n)/\nu(n)$. The calculation of the two components $\rho(n)$ and $\nu(n)$ will be 
done separately, but we begin first with some general setup that will be used in both cases. Throughout we will 
use implicitly the following basic fact many times, which is immediate from the construction of the joint process
$(\omega^j,\Gamma_j)$:
\begin{align*}
\mbox{ Conditioned on $\omega^i$, $(\omega^j,\Gamma_j)_{j=1}^{i-1}$ and $(\Gamma_j)_{j=i}^{\infty}$ are independent. }
\end{align*}

\noindent
\emph{Setup} \\
Since $L=2$, the transition matrix $A$ defined in section \ref{subsec:StepDistributionZxChain}
corresponds to the recurrent states $(p,1)$ and  $(q,0)$. Ordering the states in the order they appear here, 
and using the fact that $R=2$ and the assumption $p = 1-q$, it follows from \eqref{Ahatmatrix} that
\begin{equation*}
A=
\left(\begin{matrix}
p& 1-p \\ q^2& 1-q^2
\end{matrix}\right)
= \left(\begin{matrix}
1-q& q  \\ q^2& 1-q^2
\end{matrix}\right).
\end{equation*}
This matrix $A$ has eigenvalues
\begin{align}
\label{eq:RL2Eigenvalues}
\eta_1 = 1 ~,~ \eta_2 = 1 - q - q^2
\end{align} 
with corresponding left eigenvectors
\begin{align}
\label{eq:RL2Eigenvectors}
w_1 = (q,1) ~,~ w_2 = (-1,1).
\end{align}
For future reference we observe that $|\eta_2| < 1$, for any $q \in (0,1)$, and that 
the unit vectors $e_1 = (1,0)$ and $e_2 = (0,1)$ decompose as 
\begin{align}
\label{eq:e1e2Decomposition}
e_1 = c_1 w_1 + c_2 w_2 ~,~ e_2 = d_1 w_1 + d_2 w_2 
\end{align}
where 
\begin{align}
\label{eq:c1c2d1d2}
c_1 = \frac{1}{q+1} ~,~ c_2 = - \frac{1}{q+1} ~,~ d_1 = \frac{1}{q+1} ~,~ d_2 = \frac{q}{q+1}.
\end{align} 

With $R=2$ and $p = 1-q$ the distribution of the random variables $S_0, S_1$ 
and $S_2$ given in \eqref{Sidist} becomes:
\begin{align}
\label{SidistRL2}
& \P(S_0 = k) = q^k(1-q),~ k = 0,1 ~\mbox{ and }~  \P(S_0 = k) = q^3 (1-q)^{k-2}, ~k \geq 2. \nonumber \\
& \P(S_1 = 0) = 1-q ~\mbox{ and }~ \P(S_1 = k) = q^2 (1-q)^{k-1}, ~k \geq 1. \nonumber \\
& \P(S_2 = k) = q(1-q)^k,~ k \geq 0. 
\end{align}
Direct calculations yield
\begin{align}
\label{eq:ListOfExpectations}
& \E(S_0) = 2q ~,~ \E(S_1) = 1 ~,~ \E(S_2) = (1-q)/q, \nonumber \\
& \E(S_0^2) = 2(1+q) ~,~\E(S_2^2) = (1-q)(2-q)/q^2, \nonumber \\
& \E(S_0|S_0 \geq 2) = (1+q)/q ~,~ \E(S_2|S_2 \geq 1) = 1/q. 
\end{align} 

By \eqref{eq:GammajDist} we have $\E(\Gamma_i|\omega^i = (p,1)) = \E(S_2)$ and $\E(\Gamma_i|\omega^i = (q,0)) = \E(S_0)$,
and with our chosen state ordering $(p,1),(q,0)$ the expectation vector $E$ from section \ref{subsec:StepDistributionZxChain} 
becomes $E = (\E(S_2), \E(S_0))$. Thus, since $\omega^{i+\ell}$ is distributed as $e_1A^{\ell}$ when $\omega^i = (p,1)$ 
and as $e_2 A^{\ell}$ when $\omega^i = (q,0)$, we have
\begin{align}
\label{eq:ExTil_given_omegaip1}
\E\Big(&\Gamma_{i+\ell}\Big|\omega^i = (p,1)\Big) \nonumber \\
& = \P\Big(\omega^{i+\ell} = (p,1)\Big|\omega^i = (p,1)\Big) \cdot \E(S_2) ~+~ \P\Big(\omega^{i+\ell} = (q,0)\Big|\omega^i = (p,1)\Big) \cdot \E(S_0) \nonumber \\
& = \left<e_1A^{\ell}, E \right> = c_1 \left<w_1,E\right> \eta_1^{\ell} + c_2 \left<w_2,E\right> \eta_2^{\ell}
= 1 + \Big(\frac{1-2q}{q} \Big) \eta_2^{\ell},
\end{align}
and 
\begin{align}
\label{eq:ExTil_given_omegaiq0}
\E\Big(&\Gamma_{i+\ell}\Big|\omega^i = (q,0)\Big) \nonumber \\
& = \P\Big(\omega^{i+\ell} = (p,1)\Big|\omega^i = (q,0)\Big) \cdot \E(S_2) ~+~ \P\Big(\omega^{i+\ell} = (q,0)\Big|\omega^i = (q,0)\Big) \cdot \E(S_0) \nonumber \\
&= \left<e_2A^{\ell}, E \right> = d_1 \left<w_1,E\right> \eta_1^{\ell} + d_2 \left<w_2,E\right> \eta_2^{\ell}
= 1 + (2q-1) \eta_2^{\ell},
\end{align} 
for any $\ell \geq 0$. Similarly, denoting $E' = (\E(S_2^2), \E(S_0^2))$, we have 
\begin{align}
\label{eq:ExTil2_given_omegaip1}
\E\Big(&\Gamma_{i+\ell}^2\Big|\omega^i = (p,1)\Big) \nonumber \\
&= \left<e_1A^{\ell}, E' \right> 
= c_1 \left<w_1, E'  \right> \eta_1^{\ell} + c_2 \left<w_2, E'  \right> \eta_2^{\ell}
= \frac{2-q+3q^2}{q(1+q)} + O(|\eta_2|^{\ell}).
\end{align}

\noindent
\emph{Calculation of $\rho(n)$} \\
We will write $\rho^{(p,0)}(n)$ for the quantity $\rho(n)$ when the initial environment is $\omega(x) = (p,0)$, 
$x \geq 0$, and similarly we denote by $\rho^{(q,0)}(n)$, $\rho^{(p,1)}(n)$, and $\rho^{(q,1)}(n)$ the 
quantity $\rho(n)$ with initial environments $(q,0)$, $(p,1)$, and $(q,1)$ for $x \geq 0$. In all cases, 
we have, from \eqref{def:rhonuthetaalphahalf}, $\rho(n) = \E(U(n)) - n = \sum_{j=1}^n \E(\Gamma_j) - n$, 
where the expectation is (implicitly) the expectation conditioned on $\omega^1 = (p,0),(q,0),(p,1)$, or $(q,1)$. 

Using \eqref{eq:ExTil_given_omegaip1} with $i = 1$, along with \eqref{eq:RL2Eigenvalues}, gives
\begin{align}
\label{eq:rhop1n}
& \rho^{(p,1)}(n) 
= \sum_{j=1}^n \E\Big(\Gamma_j\Big|\omega^1 = (p,1)\Big) - n 
= \sum_{j=1}^n \left[1 + \Big(\frac{1-2q}{q} \Big) \eta_2^{j-1}\right] - n \nonumber \\
& = \Big(\frac{1-2q}{q}\Big) \frac{1}{1 - \eta_2} + O(|\eta_2|^n)
=  \frac{1-2q}{q^2(1+q)} + O(|\eta_2|^n).
\end{align}
Similarly, using 
\eqref{eq:ExTil_given_omegaiq0} with $i = 1$, along with \eqref{eq:RL2Eigenvalues}, gives
\begin{align}
\label{eq:rhoq0n}
& \rho^{(q,0)}(n) 
= \sum_{j=1}^n \E\Big(\Gamma_j\Big|\omega^1 = (q,0)\Big) - n 
= \sum_{j=1}^n \left[1 + (2q-1) \eta_2^{j-1}\right] - n \nonumber \\
& = (2q-1) \frac{1}{1 - \eta_2} + O(|\eta_2|^n)
=  \frac{2q-1}{q(1+q)} + O(|\eta_2|^n).
\end{align}
Now, if $\omega^1 = (p,0)$ then $\Gamma_1$ is distributed according to $S_2$, and  $\omega^2 = (p,1)$ with probability 1, 
by \eqref{eq:omegajplus1_determined}. Thus, by \eqref{eq:ListOfExpectations} and \eqref{eq:rhop1n},
\begin{align}
\label{eq:rhop0n}
& \rho^{(p,0)}(n)
= \Big[\E\Big(\Gamma_1\Big|\omega^1 = (p,0)\Big) - 1\Big] + \Big[\sum_{j=2}^n \E\Big(\Gamma_j\Big|\omega^2 = (p,1)\Big) - (n-1)\Big] \nonumber \\
& = [\E(S_2) - 1]+ \rho^{(p,1)}(n-1) = \frac{(1-2q)(1+q+q^2)}{q^2(1+q)} + O(|\eta_2|^n). 
\end{align}
Finally, when $\omega^1 = (q,1)$ it follows from \eqref{eq:omegajplus1_determined} that $\omega^2 = (p,1)$ if 
$\Gamma_1 \geq 1$ and $\omega^2 = (q,0)$ if $\Gamma_1 = 0$. Since $\Gamma_1$ is distributed according 
to $S_1$ when $\omega^1 = (q,1)$, we have
\begin{align}
\label{eq:rhoq1n}
& \rho^{(q,1)}(n) 
= \E\Big(\Gamma_1\Big|\omega^1 = (q,1)\Big) + \sum_{j=2}^n \E\Big(\Gamma_j\Big|\omega^1 = (q,1)\Big) - n \nonumber \\
& = \E(S_1)
+ \P(S_1 = 0) \sum_{j=2}^n \E\Big(\Gamma_j\Big|\omega^2 = (q,0)\Big) 
+ \P(S_1 \geq 1) \sum_{j=2}^n \E\Big(\Gamma_j\Big|\omega^2 = (p,1)\Big) - n \nonumber \\
& = \E(S_1) - 1 + (1-q) \rho^{(q,0)}(n-1) + q \rho^{(p,1)}(n-1) = \frac{1-2q}{1+q} + O(|\eta_2|^n),
\end{align} 
by \eqref{eq:ListOfExpectations}, \eqref{eq:rhop1n}, and \eqref{eq:rhoq0n}. \\

\noindent 
\emph{Calculation of $\nu(n)$} \\
From \eqref{def:rhonuthetaalphahalf} we have, 
\begin{equation}\label{eq:nuDef}
\nu(n) = \frac{1}{n} \E[(U(n)-n)^2] = \frac{1}{n} \E\Big[\Big(\sum_{i=1}^n\Gamma_i-n\Big)^2\Big].
\end{equation}
We will show below that with $\omega(x) = (p,1)$ for $x \geq 0$, i.e. with $\omega^1 = (p,1)$,
\begin{equation}\label{eq:nun}
\begin{aligned}
\nu(n) = \frac{2(1-q)(1-q+q^2)}{q^2(q+1)} ~+~ O(\frac1n).
\end{aligned}
\end{equation}

Similar calculations show that \eqref{eq:nun} also holds with $\omega^1 = (p,0), (q,0)$ and $(q,1)$.
The latter are omitted for the sake of brevity. However, this asymptotic equivalence of $\nu(n)$ for 
different values of $\omega^1$ is to be expected from \eqref{eq:nuDef}, since the distribution of 
$\omega^j$ converges exponentially fast to the stationary distribution of the matrix $A$,
for any value of the initial state $\omega^1$. Thus, one can couple the joint processes $(\omega^j,\Gamma_j)_{j=1}^\infty$ 
starting from two different values of $\omega^1$ in such a way that the probability that the tails $(\omega^j,\Gamma_j)_{j=n}^\infty$ 
are not the same in the two processes decays exponentially in $n$. 

We now proceed to the calculation of $\nu(n)$ with $\omega^1 = (p,1)$; thus, henceforth, all expectations 
are conditioned on $\omega^1 = (p,1)$ if not otherwise stated. From \eqref{eq:nuDef} we have
\begin{equation}\label{nu(n)3Terms}
\nu(n) = \frac{1}{n} \sum_{i=1}^n \E(\Gamma_i^2) 
+ \frac{2}{n} \sum_{1 \leq i<j \leq n} \E(\Gamma_i\Gamma_j) 
- 2 \sum_{i=1}^n \E(\Gamma_i) + n.
\end{equation}
By \eqref{eq:ExTil2_given_omegaip1}, the first term on the right hand side of \eqref{nu(n)3Terms} is
\begin{align}
\label{eq:FirstTermnu(n)}
\frac{1}{n} \sum_{i=1}^n \E(\Gamma_i^2) 
= \frac{1}{n} \sum_{i=1}^n \left[ \frac{2-q+3q^2}{q(1+q)} + O(|\eta_2|^{i-1}) \right]
= \frac{2-q+3q^2}{q(1+q)} + O(\frac1n). 
\end{align} 
Also, by \eqref{eq:rhop1n}, the third term on the right hand side of \eqref{nu(n)3Terms} is
\begin{align}
\label{eq:ThirdTermnu(n)}
- 2 \sum_{i=1}^n \E(\Gamma_i) = - 2(n + \rho^{(p,1)}(n)) = -2n +  \frac{2(2q-1)}{q^2(1+q)} + O(|\eta_2|^n).
\end{align}
To compute the second term on the right hand side of \eqref{nu(n)3Terms} we observe that, by 
\eqref{eq:GammajDist} and \eqref{eq:omegajplus1_determined}, the following hold.
\begin{enumerate} 
\item[(a)] If $\omega^i = (p,1)$ then: \\ $~~~~\Gamma_i$ is distributed as $S_2$, and $\omega^{i+1} = (p,1)$ if $\Gamma_i \geq 1$. 
\item[(b)] If $\omega^i = (q,0)$ then: \\ $~~~~\Gamma_i$ is distributed as $S_0$, and $\omega^{i+1} = (p,1)$ if
$\Gamma_i \geq 2$ and $(q,0)$ otherwise. 
\end{enumerate}
From (a) and \eqref{SidistRL2}-\eqref{eq:ExTil_given_omegaip1} we have 
\begin{align}
\label{eq:ExGammaijGivenp1}
\E\Big(\Gamma_i \Gamma_j \Big|\omega^i = (p,1)\Big) 
& = \P( S_2 \geq 1) \cdot \E(S_2 | S_2 \geq 1) \cdot \E\Big(\Gamma_j\Big|\omega^{i+1} = (p,1)\Big) \nonumber \\
& = \frac{1-q}{q} + \Big[ \frac{(1-q)(1-2q)}{q^2} \Big] \eta_2^{j - (i+1)},
\end{align}
and from (b) and \eqref{SidistRL2}-\eqref{eq:ExTil_given_omegaiq0} we have
\begin{align}
\label{eq:ExGammaijGivenq0}
\E\Big(\Gamma_i \Gamma_j \Big|\omega^i = (q,0)\Big)
& = \P(S_0 \geq 2) \cdot  \E(S_0| S_0 \geq 2) \cdot \E\Big(\Gamma_j\Big|\omega^{i+1} = (p,1)\Big) \nonumber \\
&~~~ + \P(S_0 = 1) \cdot 1 \cdot \E\Big(\Gamma_j\Big|\omega^{i+1} = (q,0)\Big) \nonumber \\
& = 2q + (1-2q)(1+q^2) \eta_2^{j-(i+1)}.
\end{align}
Since $\omega^i$ is distributed as $e_1A^{i-1}$, under our assumption $\omega^1 = (p,1)$, 
\eqref{eq:ExGammaijGivenp1} and \eqref{eq:ExGammaijGivenq0} give
\begin{align*}
\E(\Gamma_i \Gamma_j)
& = \P\Big(\omega^i = (p,1)\Big) \cdot \E\Big(\Gamma_i \Gamma_j \Big| \omega^i = (p,1)\Big)  
~+~ \P\Big(\omega^i = (q,0)\Big) \cdot \E\Big(\Gamma_i \Gamma_j \Big| \omega^i = (q,0)\Big) \nonumber \\
& = \left<e_1A^{i-1}, \left( \frac{1-q}{q} + \Big[ \frac{(1-q)(1-2q)}{q^2} \Big] \eta_2^{j - (i+1)} ~,~ 2q + (1-2q)(1+q^2) \eta_2^{j-(i+1)} \right) \right>.
\end{align*} 
Using the decomposition \eqref{eq:e1e2Decomposition}-\eqref{eq:c1c2d1d2} this simplifies to 
\begin{align}
\label{eq:ExGammaij}
\E(\Gamma_i \Gamma_j) 
= 1 + \Big[\frac{(1-2q)(1-q+q^2)}{q}\Big] \eta_2^{j-(i+1)} + \Big[\frac{1-2q}{q}\Big] \eta_2^{i-1} + C \eta_2^{j-2}
\end{align}
after a bit of algebra, where $C$ is an unimportant constant which depends only on $q$. So, 
using \eqref{eq:RL2Eigenvalues}, we find
\begin{align}
\label{eq:SecondTermnu(n)}
\frac{2}{n}&  \sum_{1 \leq i<j \leq n} \E(\Gamma_i\Gamma_j) \nonumber \\
& = \frac{2}{n} \left( \frac{n(n-1)}{2} 
+ \Big[\frac{(1-2q)(1-q+q^2)}{q}\Big] \frac{n}{1 - \eta_2} 
+ \Big[\frac{1-2q}{q}\Big] \frac{n}{1 - \eta_2} 
+ O(1) \right) \nonumber \\
& = n - 1 + \frac{2(1-2q)(2-q+q^2)}{q^2(1+q)} ~+~ O(\frac1n). 
\end{align} 
Combining \eqref{nu(n)3Terms}-\eqref{eq:ThirdTermnu(n)} and \eqref{eq:SecondTermnu(n)} 
and simplifying one arrives at \eqref{eq:nun}. \\

\noindent
\emph{Calculation and Analysis of $\theta(n)$}\\
Recall that, in general, $\theta(n) = 2\rho(n)/\nu(n)$. We denote by $\theta^{(p,0)}(n)$ the quantity 
$\theta(n)$ with $\omega^1 = (p,0)$ and define $\theta^{(p,0)} = \lim_{n \to \infty} \theta^{(p,0)}(n)$. Also, 
we define the analogous quantities for $(q,0), (p,1)$, and $(q,1)$. 
From \eqref{eq:rhop1n}-\eqref{eq:rhoq1n} and \eqref{eq:nun} we have
\begin{align}
\label{eq:thetap0}
& \theta^{(p,0)}(n) = \theta^{(p,0)} + O(\frac1n) ~~\mbox{ where }~~ \theta^{(p,0)} = \frac{(1-2q)(1+q+q^2)}{(1-q)(1-q+q^2)}, \\
\label{eq:thetap1}
& \theta^{(p,1)}(n) = \theta^{(p,1)} + O(\frac1n) ~~\mbox{ where }~~ \theta^{(p,1)} = \frac{1-2q}{(1-q)(1-q+q^2)}, \\
\label{eq:thetaq0}
& \theta^{(q,0)}(n) = \theta^{(q,0)} + O(\frac1n) ~~\mbox{ where }~~ \theta^{(q,0)} = \frac{(2q-1)q}{(1-q)(1-q+q^2)}, \\
\label{eq:thetaq1}
& \theta^{(q,1)}(n) = \theta^{(q,1)} + O(\frac1n) ~~\mbox{ where }~~ \theta^{(q,1)} = \frac{q^2(1-2q)}{(1-q)(1-q+q^2)}.
\end{align}
From \eqref{eq:thetap0} it follows that $\theta^{(p,0)} > 1$ is equivalent to $1-3q-q^2 > 0.$
Consequently, $\theta^{(p,0)} > 1$ if and only if $q\in(0,q_1^*)$, where $q_1^*$ is as in the statement 
of the theorem. From \eqref{eq:thetap1} it follows that $\theta^{(p,1)}>1$ is equivalent to
$q>2$; thus, we always have $\theta^{(p,1)}\leq1$. From \eqref{eq:thetaq1} it follows that $\theta^{(q,1)} > 1$ 
is equivalent to $q^3+q^2-2q+1<0$. Since $q^3+q^2-2q+1=q^3+(1-q)^2$,
we always have $\theta^{(q,1)}\leq1$. Finally, from \eqref{eq:thetaq0} it follows that
$\theta^{(q,0)} > 1$ is equivalent to $q^3+q-1>0$.  Thus, $\theta^{(q,0)}>1$ if and only if 
$q>q^*_2$, where $q^*_2$ is as in the statement of the theorem. This establishes all claims made
in the first paragraph of the proof. 
\end{proof}

\subsection{Proof of Theorem \ref{RorL1}}
\label{subsec:ProofThmRorL1}

Thus far the proofs of transience/recurrence for the random walk $(X_n)$ have centered around
an analysis of the right jumps Markov chain $(Z_x)$. For the  proof of Theorem \ref{RorL1},
we will need to construct another auxiliary process called the left jumps Markov chain.

Consider the random walk $(X_n)_{n \geq 0}$ started from $X_0 = 0$ and restricted to $\mathbb{N}\cup\{0,-1\}$
by the following modification of its transition mechanism: when the walker is at a site $x\ge0$, it behaves as before, but
at the site $-1$ it jumps right with probability one. Denote the modified random walk by $(\widetilde{X}_n)_{n \geq 0}$.
Note that the modified random walk can be defined in terms of the extended single site Markov chains,
$(\widehat{Y}_n^x)_{n \in \N} = (Y_n^x, J_n^x)_{n \in \N}$, $x\ge0$, along with an appropriately defined deterministic
single site mechanism at $x=-1$. Fix $N\in\mathbb{Z}^+$ and let $\widetilde{T}_N=\inf\{n\ge0:\widetilde{X}_n=N\}$
denote the first time the modified random walk hits $N$. Note that $\Tt_N$ is almost surely finite. We define a
process $(\Wt^{(N)}_x)_{x=0}^N$ by setting $\Wt_x^{(N)}$ equal to the number of times the modified walk
$(\widetilde{X}_n)$ jumps left from site $x$ before time $\widetilde{T}_N$. That is,
$$
\Wt^{(N)}_x=|\{n\le \widetilde{T}_N-1:\widetilde{X}_n=x, \widetilde{X}_{n+1}=x-1\}|.
$$
We will refer to this process $(\Wt^{(N)}_x)_{x=0}^N$ as the \emph{left jumps $N$-chain}.
It can also be defined directly in terms of the jump sequences $(J_k^x)_{k \in \N}$, $0 \leq x \leq N$:
\begin{equation}\label{W}
\Wt^{(N)}_N\equiv0, \ \Wt^{(N)}_x=\Theta^{(N)}_x- \Wt^{(N)}_{x+1}-1,\ \ x\in\{N-1,N-2,\ldots, 0\},
\end{equation}
where
\begin{equation} \label{Wtheta}
\Theta^{(N)}_x=\inf \Big\{n\ge1: \sum_{k=1}^n \indicator \{J_k^x=1\}= \Wt^{(N)}_{x+1}+1\Big\}.
\end{equation}
That is, $\Wt^{(N)}_x$ is the number of left jumps in the jump sequence $(J_k^x)_{k \in \N}$ before
the $(\Wt^{(N)}_{x+1} + 1)$-th right jump. In particular, $\Wt^{(N)}_x$ is independent of $\Wt^{(N)}_{x+2},...,\Wt^{(N)}_{N}$
conditioned on $\Wt^{(N)}_{x+1}$, so the sequence $(\Wt^{(N)}_N, ..., \Wt^{(N)}_0)$ is Markovian. The distribution
of the jump sequence $(J_k^x)_{k \in \N}$ is the same for all $x \geq 0$, if the initial environment $\omega$ is constant for
all $x \geq 0$. So, in this case, the transition probabilities
\begin{align*}
\P(&\Wt^{(N)}_x = \ell | \Wt^{(N)}_{x+1} = m) \\
& ~~= \P \Big( \inf \Big\{ n \geq 1: \sum_{k = 1}^n \indicator \{J_k^x = 1\} = m + 1 \Big\} - (m+1) = \ell \Big)
\end{align*}
are independent of $N$ and $x \in \{0,...,N-2\}$, and we may define a single time-homogeneous
Markov chain $(W_n)_{n=0}^\infty$ such that $(\Wt^{(N)}_N, \Wt^{(N)}_{N-1}, \ldots, \Wt^{(N)}_0)$
has the same distribution as $(W_0,W_1,\ldots, W_N)$, for all $N$.

We call $(W_n)_{n=0}^\infty$ the \emph{left jumps Markov chain}. The following proposition characterizes
the transience or recurrence of the original random walk $(X_n)$ in terms of the positive recurrence or
non-positive recurrence of the left jumps Markov chain.

\begin{Prop}\label{leftjumps}
If $X_0 = 0$ and the initial environment $\omega(x)$ is constant for $x \geq 0$, then the random walk $(X_n)$
has positive probability of being transient to $+\infty$ if and only if the left jumps Markov chain $(W_n)$ is
positive recurrent.
\end{Prop}

\begin{proof}
Arguments exactly like the proof of part (iii) of Lemma \ref{lem:SimpleTransConditions} show that the
modified random walk $(\Xt_n)$ either has probability 1 of being transient to $+\infty$ or probability 1
of being recurrent, and clearly the former occurs if and only if the original random walk $(X_n)$ has a positive probability
of being transient to $+\infty$. Thus, it suffices to show that the left jumps Markov chain $(W_n)$ is positive recurrent
if and only if the modified random walk $(\Xt_n)$ is transient to $+\infty$.

Now, by construction of the left jumps Markov chain $(W_n)$, we know $W_N$ and $\Wt^{(N)}_0$ have
the same distribution for each $N > 0$, where $\Wt^{(N)}_0$ is the number of jumps of the modified random walk $(\Xt_n)$
from 0 to $-1$ before it first reaches $N$. Thus, the distribution of $W_N$ is stochastically
increasing, and it converges to a limiting finite distribution if and only if the modified random walk $(\Xt_n)$
is transient. On the other hand, since $(W_n)_{n \geq 0}$ is a (time-homogeneous) irreducible, aperiodic,
Markov chain, the distribution of $W_N$ converges to a finite limiting distribution if and only if this chain
is positive recurrent.
\end{proof}

We now use Proposition \ref{leftjumps} to prove Theorem \ref{RorL1}.

\begin{proof}[Proof of Theorem \ref{RorL1}]
By symmetry it suffices to treat the case $R=1$. In the statement of the theorem, it is
assumed that the initial environment is constant in a neighborhood of $+\infty$ in the negative feedback case.
For the proof, we will make this assumption even in the positive feedback case. This causes no problem
because in the positive feedback case if we can prove that the probability of transience to $+\infty$ is 0
for any constant environment then, by Lemma \ref{lem:ComparisonOfEnvironments}, it is also true for any
non-constant initial environment. Without loss of generality, we may assume also that the initial environment
is constant for all $x\ge0$.

In this case, by Proposition \ref{leftjumps}, it suffices to show the left jumps Markov chain $(W_n)$
is not positive recurrent. By construction of the left jumps chain we have
\begin{align*}
\P(W_n=\cdot|W_{n-1}=m) = \P(\Wt_x^{(N)}=\cdot|\Wt_{x+1}^{(N)}=m),
\end{align*}
where the right hand side is independent of $N$ and $x \in \{0,...,N-2\}$ (due to the assumption
on the initial environment). Now, if we condition on $\Wt_{x+1}^{(N)} = m$, it follows from
\eqref{W} and \eqref{Wtheta} that $\Wt_x^{(N)}$ is equal to the number of left jumps in the jump
sequence $(J_k^x)_{k \in \N}$ before the time of the $(m+1)$-th right jump.

Similarly to the analysis of the right jumps chain, we decompose $\Wt_x^{(N)}$ as
\begin{align*}
\Wt_x^{(N)} = \sum_{j = 1}^{m+1} V_j,
\end{align*}
where $V_j$ is the number of left jumps in the sequence $(J_k^x)_{k \in \N}$ between the $(j-1)$-th
and $j$-th right jumps. Since $R = 1$, the configuration at site $x$ is always $(p,0)$ immediately
after a right jump from site $x$. So, the starting configuration for each of
the ``left jump sessions'' after the first one is $(p,0)$, independent of the number of left jumps
in all previous sessions. It follows that the random variables $V_1,...,V_{m+1}$ are
independent and $V_2,...,V_{m+1}$ are i.i.d. with common distribution $V$, which is the
distribution of the number of left jumps from site $x$ before the first right jump,
starting in the $(p,0)$ configuration:
\begin{equation*}
\mathbb{P}(V=k)=\begin{cases} (1-p)^k p, \ 0 \leq k \leq L-1;\\
(1-p)^L(1-q)^{k-L}q,\ k\ge L.\end{cases}
\end{equation*}
(This is analogous to the situation $L=1$ for the right jumps Markov chain, where
$U(m) = \sum_{j = 1}^m \Gamma_j$ with $\Gamma_1,...,\Gamma_m$ independent
and $\Gamma_2, ... , \Gamma_m$ i.i.d.)

We now show that since $\alpha=\frac12$, $\mathbb{E}(V)=1$.
After a somewhat messy calculation and some algebraic simplification, one finds that
$$
\mathbb{E}(V)=\frac{(1-p)q+(1-p)^L(p-q)}{pq}.
$$
From this it follows that $\mathbb{E}(V)=1$ if and only if $q=\frac{p(1-p)^L}{2p-1+(1-p)^L}$.
Since $R = 1$ and $\alpha = 1/2$, we know that $q= q_0 = \frac{p(1-p)^L}{2p-1+(1-p)^L}$ by
Remark 2 after Proposition \ref{prop:PropertiesOfAlpha}. So, we conclude that $\mathbb{E}(V)=1$.

We have now shown that
\begin{align*}
\P(W_n=\cdot|W_{n-1}=m) = \P \Big(\sum_{j=1}^{m+1} V_j = \cdot \Big),
\end{align*}
where $V_1,...,V_{m+1}$ are independent and $V_2,...,V_{m+1}$ are i.i.d. with mean $1$.
So, the Markov chain $(W_n)$ has the transition probabilities of a critical branching process
with immigration. The immigration term $V_1$ depends on the initial environment,
but is always nonnegative and not identically zero with finite mean. Also, clearly $\E(V^2) < \infty$,
so the branching terms have finite variance. It thus follows from \cite{Seneta1970} that $\frac{W_n}n$
converges in law to a certain nonzero limiting distribution, which implies the
Markov chain $(W_n)_{n \geq 0}$ cannot be positive recurrent.
\end{proof}

\section{Analysis of $\alpha$}
\label{sec:AnalysisOfAlpha}

In this section we prove Proposition \ref{prop:PropertiesOfAlpha}, which
characterizes some properties of the important quantity
\begin{align}
\label{eq:AlphaFormula}
\alpha = p \cdot \pi_p + q \cdot \pi_q
\end{align}
that determines the direction of transience for our random walk (away from borderline critical case).
We recall from \eqref{eq:pippiq} that
\begin{align*}
\pi_p & =  \frac{(1-q)q^R(1 - (1-p)^L)}{(1-q)q^R(1 - (1-p)^L) + p(1-p)^L(1-q^R)} ~, \\
\pi_q & =  \frac{p(1-p)^L(1-q^R)}{(1-q)q^R(1 - (1-p)^L) + p(1-p)^L(1-q^R)}.
\end{align*}
The various pieces of the proposition will be proved separately, but we begin first with two useful observations.
\begin{itemize}
\item[(I)] For any fixed $q, R, L$ the quantity
\begin{align}
\label{eq:Ratio_pippiq}
\frac{\pi_p}{\pi_q} = \frac{(1-q)q^R}{1-q^R} \cdot \frac{1 - (1-p)^L}{p(1-p)^L}
\end{align}
satisfies $\lim_{p \rightarrow 1} \left( \frac{\pi_p}{\pi_q} \right) = \infty.$
Since $\pi_p + \pi_q = 1$, this implies $\lim_{p \to 1} \pi_p = 1$.

\item[(II)] For any fixed $q, R, L$ the quantity $\frac{\pi_p}{\pi_q}$ satisfies
\begin{align*}
\frac{d}{dp} \left( \frac{\pi_p}{\pi_q} \right) = \frac{(1-q)q^R}{1-q^R} \cdot \frac{p(L+1) + (1-p)^{L+1} - 1}{p^2 (1-p)^{L+1}} > 0 ~,~ \forall p \in (0,1).
\end{align*}
Since $\pi_p + \pi_q = 1$, this implies $\frac{d}{dp}(\pi_p) > 0$, $\forall p \in (0,1)$. So,
\begin{align}
\label{eq:DerivativeAlpha}
&\frac{d}{dp} (\alpha)
= \frac{d}{dp} (p \cdot \pi_p + q \cdot \pi_q) \nonumber =  \frac{d}{dp} (p \cdot \pi_p + q \cdot (1-\pi_p)) \nonumber \\
& = \pi_p + (p-q) \cdot \frac{d}{dp} (\pi_p) > 0, \mbox{ for all } p \geq q.
\end{align}
\end{itemize}

\noindent
\emph{Proof of (vi):}
By (I), $\lim_{p \to 1} \alpha = \lim_{p \to 1} (p \cdot \pi _p + q \cdot \pi q) = 1 \cdot 1 ~+~ q \cdot 0 = 1$. \\

\noindent
\emph{Proof of (i):}
This is immediate from (\ref{eq:AlphaFormula}) since $\pi_p + \pi_q = 1$ and $\pi_p, \pi_q > 0$, for any $p,q$. \\

\noindent
\emph{Proof of (ii):}
If $q < 1/2$, then $\alpha < 1/2$ for all $p \leq 1/2$, by (\ref{eq:AlphaFormula}).
But, by (II) and (vi), we also know that $\alpha(p)$ is monotonically increasing
on the interval $[1/2,1) \subset [q,1)$, with $\lim_{p \to 1} \alpha(p) = 1$.
Thus, the claim follows by continuity of $\alpha(p)$. \\

\noindent
\emph{Proof of (iii):}
Plugging $p = 1-q$ into  (\ref{eq:DefAlpha}) and simplifying one finds that
\begin{align*}
\alpha(1-q) < 1/2 & \iff q^R(1/2 - q) - q^L(1/2-q) < 0, \mbox{ and }\\
\alpha(1-q) > 1/2 & \iff q^R(1/2 - q) - q^L(1/2-q) > 0.
\end{align*}
Thus, for $q < 1/2$ and $R > L$, $\alpha(1-q) < 1/2$, which implies $p_0 > 1-q$. While, for
$q < 1/2$ and $R < L$, $\alpha(1-q) > 1/2$, which implies $p_0 < 1-q$.
This proves \eqref{eq:qPlusp0}.

Now, by (II) and symmetry considerations, for any fixed $R,L,p$ we know that
$d/dq(\alpha) > 0$ for $q \leq p$. Thus, for any $0 < q < q' < 1/2$, we have
\begin{align*}
\alpha(p_0(q,R,L),q',R,L) > \alpha(p_0(q,R,L),q,R,L) = 1/2,
\end{align*}
which implies $p_0(q',R,L) < p_0(q,R,L)$. So, $p_0$ is a decreasing
function of $q$, for $q \in (0,1/2)$. \\

\noindent
\emph{Proof of (iv):}
Plugging $L=1$ into (\ref{eq:DefAlpha}) and simplifying one finds that
\begin{align*}
\alpha = 1/2 \iff p(1 - 2q + q^R) = 1 - 2q + q^{R+1}
\end{align*}
and, similarly,
\begin{align*}
\alpha < 1/2 \iff p(1 - 2q + q^R) < 1 - 2q + q^{R+1}, \\
\alpha > 1/2 \iff p(1 - 2q + q^R) > 1 - 2q + q^{R+1}.
\end{align*}
(iv) follows by considering separately the two cases
$1 - 2q + q^{R+1} > 0$ and $1 - 2q + q^{R+1} \leq 0$. \\

\noindent
\emph{Proof of (v):}
If $L = R$, then plugging in $p = 1-q$ into (\ref{eq:DefAlpha}) gives $\alpha = 1/2$.
So, by (ii), if $q < 1/2$ then $p_0 = 1 - q$ is the unique critical point. On the the other hand,
for any $q > 1/2$, if $L = R$ is sufficiently large then there exists another critical point $p_0' > 1 - q$.
This follows from (vi), continuity of $\alpha$, and the following claim. \\

\noindent
\emph{Claim}: For any fixed $q > 1/2$, if $L = R$ is sufficiently large then $\frac{d}{dp}(\alpha) |_{p = 1-q} < 0$.
Thus, there exists some $\epsilon > 0$ such that $\alpha(1-q + \epsilon) < 1/2$. \\

\noindent
\emph{Proof}:
Computing $\frac{d}{dp}(\alpha)$ directly from (\ref{eq:DefAlpha}) and then substituting $L = R$ and $p = 1 - q$,
one finds, after some lengthy simplifications, that the condition $\frac{d}{dp}(\alpha) |_{p = 1-q} < 0$ is
equivalent to the condition
\begin{align*}
R(1-q)(1-2q) + q(1-q^R) < 0.
\end{align*}
For fixed $q > 1/2$, this condition is satisfied for all sufficiently large $R$.

\hfill $\square$

\appendix

\section{Solution of Linear Systems}
\label{sec:SolutionLinearSystems}

\subsection{Stationary Distribution of Single Site Markov Chains}
\label{subsec:StationaryDistributionSingleSiteMC}

Here we solve the linear system $\{ \pi = \pi M , \sum_{\lambda} \pi_\lambda = 1\}$ for the stationary distribution $\pi$
of the single site Markov chain transition matrix $M$. In expanded form this system becomes
\begin{align}
\label{eq:pi_piEqual}
& \pi_{(p,i)} = (1-p) \cdot \pi_{(p,i-1)} ~,~ 1 \leq i \leq L-1\\
\label{eq:pi_p0Equal}
& \pi_{(p,0)} = p \cdot \pi_p + q \cdot \pi_{(q,R-1)} \\
\label{eq:pi_qiEqual}
& \pi_{(q,i)} = q \cdot \pi_{(q, i -1)} ~,~ 1 \leq i \leq R-1 \\
\label{eq:pi_q0Equal}
& \pi_{(q,0)} = (1-q) \cdot \pi_q + (1-p) \cdot \pi_{(p,L-1)} \\
\label{eq:pipPluspiq_Equal1}
& \pi_p +  \pi_q = 1,
\end{align}
where $\pi_p = \sum_{i=0}^{L-1} \pi_{(p,i)}$ and $\pi_q = \sum_{i=0}^{R-1} \pi_{(q,i)}$.
Applying (\ref{eq:pi_piEqual}) and (\ref{eq:pi_qiEqual}) repeatedly gives
\begin{align}
\label{eq:pi_piEq}
\pi_{(p,i)} & = (1-p)^i \cdot \pi_{(p,0)} ~,~0 \leq i \leq L-1; \\
\label{eq:pi_qiEq}
\pi_{(q,i)} & = q^i \cdot \pi_{(q,0)} ~,~0 \leq i \leq R-1.
\end{align}
Hence,
\begin{align}
\label{eq:pi_pEq}
\pi_p &= \sum_{i=0}^{L-1} (1-p)^i \cdot \pi_{(p,0)} = \frac{1 - (1-p)^L}{p} \cdot \pi_{(p,0)}~, \\
\label{eq:pi_qEq}
\pi_q & = \sum_{i=0}^{R-1} q^i \cdot \pi_{(q,0)} = \frac{1 - q^R}{1 - q} \cdot \pi_{(q,0)}.
\end{align}
Plugging (\ref{eq:pi_qiEq}) and (\ref{eq:pi_pEq}) into (\ref{eq:pi_p0Equal}) gives
\begin{align*}
\pi_{(p,0)} = p \cdot \left( \frac{1 - (1-p)^L}{p} \cdot \pi_{(p,0)} \right) ~+~ q \cdot \left( q^{R-1} \cdot \pi_{(q,0)} \right),
\end{align*}
which implies
\begin{align}
\label{eq:pi_p0_Expression1}
\pi_{(p,0)} = \pi_{(q,0)} \cdot \frac{q^R}{(1-p)^L}.
\end{align}
But, by (\ref{eq:pipPluspiq_Equal1}), (\ref{eq:pi_pEq}), and (\ref{eq:pi_qEq}), we also have
\begin{align*}
\frac{1 - (1-p)^L}{p} \cdot \pi_{(p,0)} ~+~  \frac{1 - q^R}{1 - q} \cdot \pi_{(q,0)} = 1
\end{align*}
or, equivalently,
\begin{align}
\label{eq:pi_p0_Expression2}
\pi_{(p,0)} = \left(1 - \pi_{(q,0)} \frac{1 - q^R}{1-q} \right) \cdot \frac{p}{1 - (1-p)^L}.
\end{align}
Equating the right hand sides of (\ref{eq:pi_p0_Expression1}) and (\ref{eq:pi_p0_Expression2})
and solving for $\pi_{(q,0)}$ gives
\begin{align*}
\pi_{(q,0)} = \frac{p(1-q)(1-p)^L}{(1-q)q^R(1 - (1-p)^L) + p(1-p)^L(1-q^R)}.
\end{align*}
Substituting this value of $\pi_{(q,0)}$ into \eqref{eq:pi_p0_Expression1} gives an explicit expression
for $\pi_{(p,0)}$, and the values of $\pi_{(q,i)}, 1 \leq i \leq R-1$, and $\pi_{(p,i)}, 1 \leq i \leq L-1$, are then
easily found by substituting the expressions for $\pi_{(p,0)}$ and $\pi_{(q,0)}$ in (\ref{eq:pi_piEq}) and (\ref{eq:pi_qiEq}),
giving \eqref{eq:pipi_piqi}.

\subsection{Expected Hitting Times with $R=1$}
\label{subsec:ExpectedHittingTimes}

Here we solve the linear system (\ref{eq:LinearSystemExpectedHittingTimes}) for the expected
hitting times $a_i$, $0 \leq i \leq L$. As shown in the proof of Theorem \ref{thm:R1Speed},
using soft methods, these expected hitting times must all be finite.

For simplicity of notation we define $b_i = a_{L-i}$, $0 \leq i \leq L$. Rearranging slightly
the system (\ref{eq:LinearSystemExpectedHittingTimes}) then becomes
\begin{align*}
b_{i+1} & = 1 + (1-p)(a_0 + b_i) ~,~0 \leq i \leq L-1 \\
b_0 & = \frac{1}{q} + \left(\frac{1-q}{q}\right) a_0.
\end{align*}
Thus, for each $0 \leq i \leq L$, we have
\begin{align*}
b_i = u_i + v_i \cdot a_0
\end{align*}
where the sequences $(u_i)_{i=0}^L$ and $(v_i)_{i=0}^L$ are defined recursively by
\begin{align*}
u_0 & = 1/q ~~\mbox{ and }~~ u_{i+1} = 1 + (1-p)u_i,~  0 \leq i \leq L-1,\\
v_0 & = (1-q)/q ~~\mbox{ and }~~ v_{i+1} = (1-p)(1 + v_i),~  0 \leq i \leq L-1.
\end{align*}
By induction on $i$, we find that, for each $1 \leq i \leq L$,
\begin{align*}
u_i & = \frac{(1-p)^i}{q} + \sum_{j = 0}^{i-1} (1-p)^j =  \frac{1 + (p/q - 1)(1-p)^i}{p}~, \\
v_i & = \frac{(1-p)^i}{q} + \sum_{j = 1}^{i-1} (1-p)^j = \frac{1 - p + (p/q - 1)(1-p)^i}{p}.
\end{align*}
Substituting, first for the $b_i$'s and then for the $a_i$'s with
$a_i = b_{L-i}$, one obtains (\ref{eq:Def_a0}) and (\ref{eq:Def_ai}).

\section{Proof of Lemma \ref{lem:SimpleTransConditions}}
\label{sec:BasicTransienceConditions}

Here we prove Lemma \ref{lem:SimpleTransConditions} from section \ref{subsec:BasicLemmas}.
The three parts are proved separately. In each case, we prove only the first of the two statements,
since the second follows by symmetry. The following notation will be used for the proofs.
\begin{itemize}
\item $T_x^{(i)}$ is the $i$-th hitting time of site $x$:
\begin{align*}
T_x^{(1)} = T_x ~~\mbox{ and }~~ T_x^{(i+1)} = \inf\{n > T_x^{(i)}: X_n = x\},
\end{align*}
with the convention $T_x^{(j)} = \infty$, for all $j > i$, if $T_x^{(i)} = \infty$.
\item $m_i = \sup \{ X_n : n \leq T_0^{(i)} \}$ is the maximum position of the random walk
up to the $i$-th hitting time of site $0$.
\item For an initial environment $\omega$ and path $\zeta = (x_0,...,x_n)$,
$\omega^{(\zeta)}$ is the environment induced at time $n$ by
following the path $\zeta$ starting in $\omega$:
\begin{align*}
\{\omega_0 = \omega, X_0 = x_0,...,X_n = x_n\} \implies \omega_n = \omega^{(\zeta)}.
\end{align*}
\end{itemize}

\noindent
\emph{Proof of (ii):}
Clearly, $\P_{\omega}(X_n \rightarrow \infty) \leq \P_{\omega}(\liminf_{n \to \infty} X_n > -\infty)$.
To show the reverse inequality also holds observe that, for any $k \in \Z$,
$\P_{\omega}(\liminf_{n \to \infty} X_n = k) = 0$. Thus,
\begin{align*}
\P_{\omega}\left(\liminf_{n \to \infty} X_n > -\infty, X_n \not\rightarrow \infty \right)
= \P_{\omega}\left(-\infty< \liminf_{n \to \infty} X_n <\infty \right) = 0.
\end{align*}

\noindent
\emph{Proof of (i):}
By (ii), $\P_{\omega}(X_n \rightarrow \infty) \geq \P_{\omega}(\AM_0^+)$. Thus,
$\P_{\omega}(X_n \rightarrow \infty) > 0$, if $\P_{\omega}(\AM_0^+) > 0$.
On the other hand, if $\P_{\omega}(X_n \rightarrow \infty) > 0$ then there exists
some finite path $\zeta = (x_0,...,x_n)$, such that $x_ 0 = 0$, $x_n = 2$, and
\begin{align*}
\P_{\omega}(X_m > 1, \forall m \geq n|X_0 = x_0, ..., X_n = x_n) > 0.
\end{align*}
We construct from $\zeta = (x_0,...,x_n)$ the reduced path $\zetat = (\xt_0,...,\xt_{\nt})$ by setting $\xt_0 = x_0 = 0$,
and then removing from the tail $(x_1,...,x_n)$ all steps before the first hitting time of site $1$ and all
steps in any leftward excursions from site $1$. For example,
\begin{align*}
\mbox{ if } \zeta & = (0,\textbf{-1},\textbf{0},1,2,1,\textbf{0},\textbf{1}, 2, 1, \textbf{0} ,\textbf{-1},\textbf{-2},\textbf{-1},\textbf{-2},\textbf{-1},\textbf{0},\textbf{1},2,3,2), \\
\mbox{ then } \zetat & = (0,1,2,1,2,1,2,3,2)
\end{align*}
(where we denote the removed steps in bold for visual clarity). By construction,
$\omega^{(\zetat)}(x) = \omega^{(\zeta)}(x)$, for all $x \geq 2$.
So, $\P_{\omega}(X_m > 1, \forall m \geq \nt|(X_0, ..., X_{\nt}) = \zetat)
= \P_{\omega}(X_m > 1, \forall m \geq n|(X_0,...,X_n) = \zeta) > 0$.
Thus,
\begin{align*}
\P_{\omega}(\AM_0^+) \geq \P_{\omega}( (X_0,...,X_{\nt}) = \zetat) \cdot \P_{\omega}(X_m > 1, \forall m \geq \nt|(X_0,...,X_{\nt}) = \zetat) > 0.
\end{align*}

\noindent
\emph{Proof of (iii):}
Since we assume $\P_{\omega}(X_n \rightarrow - \infty) = 0$, it follows from (ii) that (a) $T_x$ is $\P_{\omega}$ a.s.
finite for each $x \geq 0$, and (b) every time the random walk steps left from a site $x$ it will eventually return with
probability $1$. Now (b) implies that the probability that the walk is transient to $+\infty$, after first hitting a site
$x \geq 0$, is independent of the trajectory taken to get to $x$. That is,
$\P_{\omega}(X_n\rightarrow \infty|(X_0,...,X_n) = \zeta) = \P_{\omega}(X_n\rightarrow \infty|T_x < \infty)$,
for any $x \geq 0$ and path $\zeta = (x_0, ... ,x_n)$ such that $x_0 = 0, x_n = x$, and $x_m < x$ for $m < n$.
Combining this last observation with (a) shows that
\begin{align*}
& \P_{\omega}(X_n\rightarrow \infty|T_0^{(i)} < \infty, m_i = x)
= \P_{\omega}(X_n \rightarrow \infty|T_0^{(i)} < \infty, m_i = x, T_{x+1} < \infty) \nonumber \\
&= \P_{\omega}(X_n \rightarrow \infty|T_{x+1} < \infty)
= \P_{\omega}(X_n \rightarrow \infty), \mbox{ for all $x \geq 0$ and $i \geq 1$. }
\end{align*}
So, $\P_{\omega}(X_n\rightarrow \infty|T_0^{(i)} < \infty) = \P_{\omega}(X_n\rightarrow \infty)$, for all $i \geq 1$.
Thus, by (ii),
\begin{align*}
\P_{\omega}(X_n \not\rightarrow \infty)
= \P_{\omega}(X_n \not\rightarrow \infty|T_0^{(i)} < \infty)
= \prod_{j=i}^{\infty} \P_{\omega}(T_0^{(j+1)} < \infty|T_0^{(j)} < \infty), \forall i \geq 1.
\end{align*}
Since the LHS is independent of $i$, the product on the RHS is constant for $i \geq 1$. Thus, there are two possibilities: either the
product is $0$ (for all $i \geq 1$) or $\P_{\omega}(T_0^{(j+1)} < \infty|T_0^{(j)} < \infty) = 1$, for all $j \geq 1$. In the latter case,
$\P_{\omega}(X_n \not\rightarrow \infty) = 1$, which contradicts the assumption that $\P_{\omega}(X_n \rightarrow \infty) > 0$.
In the former case, $\P_{\omega}(X_n \rightarrow \infty) = 1$, as required. 

\hfill $\square$

\section{Proof of Lemma \ref{lem:StrongLawForNx}}
\label{sec:StrongLawForNx}

The following strong law for sums of dependent random variables is a special
case of  \cite[Theorem 1]{Etemadi1983}  with $w_i = 1$ and $W_i = i$.

\begin{The}
\label{thm:StrongLawDepedentRVs}
Let $(\xi_i)_{i \in \N}$ be a sequence of nonnegative random variables satisfying:
\begin{enumerate}
\item $\sup_i \E(\xi_i) < \infty$.
\item $\E(\xi_i^2) < \infty$, for each $i$.
\item $\sum_{j = 1}^{\infty} \sum_{i=1}^j \frac{1}{j^2} \cdot Cov^+(\xi_i,\xi_j) < \infty$.
\end{enumerate}
Then
\begin{align*}
\frac{1}{n} \sum_{i=1}^n (\xi_i - \E(\xi_i)) \stackrel{a.s.}{\longrightarrow} 0, \mbox{ as } n \rightarrow \infty.
\end{align*}
\end{The}

Using this theorem we will prove Lemma \ref{lem:StrongLawForNx}. Throughout our proof
the initial environment $\omega$ is fixed, and all random variables are distributed according
to the measure $\P_{\omega}$, which we will abbreviate simply as $\P$. Also, $\beta > 0$ is
the constant given in Corollary \ref{cor:PrAnPlusGreaterEqualBeta}.

\begin{proof}[Proof of Lemma \ref{lem:StrongLawForNx}]

By Corollary \ref{cor:NxDominatedByGeometric},
\begin{align}
\label{eq:ENxVarNxBounds}
\E(N_x) \leq \frac{1}{\beta} ~~\mbox{ and }~~ \E(N_x^2) \leq \frac{2 - \beta}{\beta^2}~,~ \mbox{ for each $x \in \N$}.
\end{align}
Thus, by Theorem \ref{thm:StrongLawDepedentRVs}, it suffices to show that
\begin{align*}
\sum_{y = 1}^{\infty} \sum_{x = 1}^{y} ~ \frac{1}{y^2} Cov^{+} (N_x,N_y) < \infty.
\end{align*}

Since $N_x$ and $N_y$ are nonnegative integer valued random variables, $Cov(N_x,N_y)$
can be represented as the following absolutely convergent double sum:
\begin{align}
\label{eq:CovNxNySumExpression}
Cov(N_x,N_y) = \sum_{j = 1}^{\infty} \sum_{k = 1}^{\infty} \Big( \P(N_x \geq k, N_y \geq j) - \P(N_x \geq k) \P(N_y \geq j) \Big).
\end{align}
To bound this sum we will need the following two estimates for the
differences $D_{k,j} \equiv \P(N_x \geq k, N_y \geq j) - \P(N_x \geq k) \P(N_y \geq j)$:
\begin{align}
\label{eq:jkEstimate}
& \mbox{ For any $1 \leq x < y$ and $k,j \in \N$}, D_{k,j} \leq (1-\beta)^{\max\{j,k\}-1}. \\
\label{eq:xyEstimate}
& \mbox{ For any $1 \leq x < y$ and $k,j \in \N$}, D_{k,j} \leq (1 - \beta)^{y-x}.
\end{align}
\eqref{eq:jkEstimate} follows from Corollary \ref{cor:NxDominatedByGeometric}:
\begin{align*}
D_{k,j} & \equiv \P( N_x \geq k , N_y  \geq j) - \P(N_x \geq k) \P(N_y \geq j) \\
& \leq \P(N_x \geq k, N_y \geq j)
\leq \min\{ \P(N_x \geq k), \P(N_y \geq j) \}
\leq (1-\beta)^{\max\{j,k\}-1}.
\end{align*}
To see \eqref{eq:xyEstimate} recall that $N_x^y$ and $N_y$ are independent
for all $1 \leq x < y$, by Lemma \ref{lem:SimpleConsequencesAlphaGreaterHalf}.
Thus, for any $1 \leq x < y$, we have
\begin{align*}
\P(N_x \geq k, N_y \geq j)
& = \P(N_x^y \geq k, N_y \geq j) + \P(N_x^y < k, N_x \geq k, N_y \geq j) \\
& =  \P(N_x^y \geq k) \P(N_y \geq j) + \P(N_x^y < k, N_x \geq k, N_y \geq j) \\
& \leq \P(N_x \geq k) \P(N_y \geq j) + \P(B_y \geq y - x) \\
& \leq \P(N_x \geq k) \P(N_y \geq j) + (1 - \beta)^{y - x}
\end{align*}
by Corollary \ref{cor:BxDominatedByGeometric}.

Now, for given $1 \leq x < y$, let $n = y - x$ and let $N = \floor{(1-\beta)^{-n/4}}$. Breaking the (absolutely convergent)
double sum in (\ref{eq:CovNxNySumExpression}) into pieces and applying Fubini's Theorem gives
\begin{align*}
Cov(N_x,N_y)
& = \sum_{j=1}^N \sum_{k = 1}^N D_{k,j}
~+~ \sum_{j=1}^N \sum_{k = N+1}^{\infty} D_{k,j}
~+~ \sum_{k=1}^N \sum_{j = N+1}^{\infty} D_{k,j} \\
&~~+  \sum_{k=N+1}^{\infty} \sum_{j = k}^{\infty} D_{k,j}
~+~ \sum_{j=N+1}^{\infty} \sum_{k = j+1}^{\infty} D_{k,j}.
\end{align*}
By \eqref{eq:xyEstimate}, the first term on the RHS of this equation is bounded above by \\
$N^2 (1-\beta)^n$. Similarly, by \eqref{eq:jkEstimate}:
\begin{itemize}
\item The second term is bounded by $N \cdot \sum_{k=N+1}^{\infty} (1-\beta)^{k-1} = N (1-\beta)^N/\beta$.
\item The third term is bounded by $N \cdot \sum_{j=N+1}^{\infty} (1-\beta)^{j-1} = N (1-\beta)^N/\beta$.
\item The fourth term is bounded by $\sum_{k=N+1}^{\infty} \sum_{j = k}^{\infty} (1-\beta)^{j-1} = (1 - \beta)^N/\beta^2$.
\item The fifth term is bounded by $\sum_{j=N+1}^{\infty} \sum_{k = j+1}^{\infty} (1-\beta)^{k-1} = (1 - \beta)^{N+1}/\beta^2$.
\end{itemize}
The upper bound on the first term is at most $(1-\beta)^{n/2}$, and the same is also true for the upper bounds
on each of the other 4 terms for all sufficiently $n$, since $N$ grows exponentially in $n$. Thus, there exists some
$n_0 \in \N$ such that
\begin{align*}
Cov(N_x,N_y) \leq 5 (1-\beta)^{n/2}, \mbox{ whenever } y - x =  n \geq n_0.
\end{align*}
But, for any $1 \leq x \leq y$ such that $y - x = n < n_0$ we also have
\begin{align*}
Cov(N_x,N_y) \leq \E(N_x^2)^{1/2} \cdot \E(N_y^2)^{1/2} \leq \frac{2 - \beta}{\beta^2}
\leq \left( \frac{2 - \beta}{\beta^2 (1-\beta)^{n_0 - 1}} \right) (1-\beta)^n
\end{align*}
by \eqref{eq:ENxVarNxBounds}. Thus, for all $1 \leq x \leq y$,
\begin{align*}
Cov(N_x,N_y) \leq C (1 - \beta)^{n/2}, \mbox{ where } C \equiv \max \left\{5, \frac{2 - \beta}{\beta^2 (1-\beta)^{n_0 - 1}}\right\} \mbox{ and } n = y -x.
\end{align*}
So,
\begin{align*}
\sum_{y = 1}^{\infty} \sum_{x = 1}^{y} ~ \frac{1}{y^2} Cov^{+} (N_x,N_y)
\leq \sum_{y = 1}^{\infty} \sum_{x = 1}^{y} ~ \frac{1}{y^2} \cdot C(1-\beta)^{(y-x)/2}
< \infty.
\end{align*}
\end{proof}

\section{Proofs of Lemmas \ref{lem:muEqual1}, \ref{lem:ConcentrationEstimate}, and \ref{lem:nuLowerBound}}
\label{sec:ProofOfUnLemmas}

\begin{proof}[Proof of Lemma \ref{lem:muEqual1}] 
Since $U(n) = \sum_{j = 1}^n \Gamma_j$, if follows from the Markov chain representation
of section \ref{subsec:StepDistributionZxChain} and the ergodic theorem for finite-state 
Markov chains along with \eqref{eq:GammajDist} that
\begin{align*}
\lim_{n \to \infty} \frac{\E(U(n))}{n} = \lim_{j \to \infty} \E(\Gamma_j) = \left<\psi,E \right>
~\mbox{ and }~
\lim_{n \to \infty} \frac{1}{n} \sum_{j = 1}^n \Gamma_j = \left<\psi,E \right>,~ \mbox{ a.s. }
\end{align*}
By definition, $\Gamma_j$ is the number of right jumps (i.e. $1$'s) in the jump sequence $(J_k^x)_{k \in \N}$
between  the $(j-1)$-th and $j$-th left jumps. So, this implies
\begin{align*}
\lim_{m \to \infty} \frac{1}{m} \sum_{k=1}^m \indicator\{J_k^x = 1\}
= \lim_{n \to \infty} \left( \frac{ \sum_{j=1}^n \Gamma_j }{n +  \sum_{j=1}^n \Gamma_j} \right)
= \frac{ \left<\psi,E \right> }{1 + \left<\psi,E \right> } ~,\mbox{ a.s. }
\end{align*}
On the other hand, as noted at the end of section \ref{subsubsec:SingleSiteMarkovChain},
\begin{align*}
\lim_{m \to \infty} \frac{1}{m} \sum_{k = 1}^m \indicator\{J_k^x = 1\} = \alpha ~, \mbox{ a.s. }
\end{align*}
Since $\alpha = 1/2$, it follows that $\left<\psi,E \right> = 1$.
\end{proof}

\begin{proof}[Proof of Lemma \ref{lem:nuLowerBound}]
We consider separately the cases $L = 1$ and $L \geq 2$. 
In both cases, since $\alpha = 1/2$ we have $\mu = 1$, 
by Lemma  \ref{lem:muEqual1}. Thus, 
$\nu(n) = \E [ (U(n)-n)^2 ]/n$. \\

\noindent
\emph{Case 1}: $L = 1$.\\
In this case $\omega^j = (q,0)$ for all $j \geq 2$, regardless of the values of the $\Gamma_j$'s. 
Thus, $\Gamma_1, ..., \Gamma_n$ are independent and $\Gamma_2,...,\Gamma_n$ are i.i.d. distributed as $S_0$. So, 
\begin{align*}
\liminf_{n \to \infty} \nu(n) 
= \liminf_{n \to \infty} \frac{\E [ (U(n)-n)^2 ]}{n} 
\geq \liminf_{n \to \infty} \frac{\Var(U(n))}{n}
= \Var(S_0) > 0.
\end{align*} 

\noindent
\emph{Case 2}: $L \geq 2$.\\
By construction $\omega^{j+1}$ is a deterministic function of $\omega^j$ and $\Gamma_j$. 
For $\lambda, \lambda' \in \Lambda$, we define $K_{\lambda,\lambda'} = 
\{k \geq 0: \omega^{j+1} = \lambda', \mbox{ if } \omega^j = \lambda \mbox{ and } \Gamma_j = k \}$.
We say a sequence of configurations $\vec{\lambda} = (\lambda_1,...,\lambda_{n+1}) \in \Lambda^{n+1}$ 
is \emph{allowable} if $|K_{\lambda_i, \lambda_{i+1}}| > 0$ for all $1 \leq i \leq n$, and denote 
by $G_{n+1}$ the set of all allowable length-$(n+1)$ configuration sequences. For each allowable configuration 
sequence $\vec{\lambda} \in G_{n+1}$ we define $(\Gamma_{j,{\vec{\lambda}}})_{j=1}^n$ to be independent random 
variables with distribution 
\begin{align*}
\P(\Gamma_{j,\vec{\lambda}} = k) 
& = \P(\Gamma_j = k|\omega^j = \lambda_j, \omega^{j+1} = \lambda_{j+1}) \\
& =  \P(\Gamma_j = k|\omega^j = \lambda_j,  \Gamma_j \in K_{\lambda_j, \lambda_{j+1}}).
\end{align*}
Also, we define $U_{\vec{\lambda}}(n) = \sum_{j=1}^n \Gamma_{j,\vec{\lambda}}$. 

By construction of the joint process $(\omega^j,\Gamma_j)$, it follows that $U(n)$ conditioned on 
$(\omega^1,...,\omega^{n+1}) = \vec{\lambda}$ is distributed as $U_{\vec{\lambda}}(n)$.
Thus, denoting $\vec{\omega} = (\omega^1,...,\omega^{n+1})$, we have 
\begin{align} 
\label{eq:ConditionalSumNun}
\E[(U(n)-n)^2] 
& = \sum_{\vec{\lambda} \in G_{n+1}} \P(\vec{\omega} = \vec{\lambda}) \cdot \E[(U(n)-n)^2 | \vec{\omega} = \vec{\lambda}] \nonumber \\
& = \sum_{\vec{\lambda} \in G_{n+1}} \P(\vec{\omega} = \vec{\lambda}) \cdot \E[(U_{\vec{\lambda}}(n)-n)^2] \nonumber \\
& \geq \sum_{\vec{\lambda} \in G_{n+1}} \P(\vec{\omega} = \vec{\lambda}) \cdot \Var( U_{\vec{\lambda}}(n) ) \nonumber \\
& =  \sum_{\vec{\lambda} \in G_{n+1}} \P(\vec{\omega} = \vec{\lambda}) \sum_{j=1}^n \Var( \Gamma_{j,\vec{\lambda}} ).
\end{align}
The lemma follows easily from this since the pair $((p,1),(p,1))$ is a recurrent state for the Markov chain 
over configuration pairs $(\omega^j,\omega^{j+1})_{j \in \N}$ and the distribution of $\Gamma_j$ conditioned on 
$\omega^j = \omega^{j+1} = (p,1)$ is non-degenerate. Indeed, denoting the variance in the distribution of
$\Gamma_j$ conditioned on $\omega^j = \omega^{j+1} = (p,1)$ as $V_{(p,1),(p,1)}$ and the stationary probability 
of the pair $((p,1),(p,1))$ as $\psi_{(p,1),(p,1)}$, \eqref{eq:ConditionalSumNun} implies
\begin{align*}
\liminf_{n \to \infty} \nu(n) = \liminf_{n \to \infty} \frac{\E[(U(n)-n)^2]}{n} \geq V_{(p,1),(p,1)} \cdot \psi_{(p,1),(p,1)} > 0. 
\end{align*}
\end{proof}

We now proceed to the proof of Lemma \ref{lem:ConcentrationEstimate}. This is based on the following 
basic facts concerning large deviations of i.i.d. random variables and finite-state Markov chains: \\

\noindent
\emph{Fact 1}: If $\xi$ is a random variable with exponential tails and $\xi_1, \xi_2,...$ are i.i.d. random variables
distributed as $\xi$, then there exist constants $b_1, b_2 > 0$ such that the empirical means 
$\xib_n \equiv \frac{1}{n} \sum_{i=1}^n \xi_i$ satisfy:
\begin{align}
\label{eq:iidSmallEpsBound}
\P( |\xib_n - \E(\xi)| > \epsilon ) & \leq b_1 \exp(-b_2 \epsilon^2 n) ~,~ \mbox{for all } 0 < \epsilon \leq 1 \mbox{ and } n \in \N; \\
\label{eq:iidLargeEpsBound}
\P( |\xib_n - \E(\xi)| > \epsilon ) & \leq b_1 \exp(-b_2 \epsilon n) ~,~ \mbox{for all } \epsilon > 1 \mbox{ and } n \in \N. 
\end{align}

\noindent
\emph{Fact 2}: If $(\xi_n)_{n \in \N}$ is an irreducible Markov chain on a finite state space $S$ with 
stationary distribution $\phi$, then there exist constants $b_1, b_2 > 0$ such that the empirical state frequencies 
$\phi_n(s) \equiv \frac{1}{n} \sum_{i = 1}^n \indicator\{\xi_i = s\}$ satisfy
\begin{align*}
\P_{s'}( |\phi_n(s) - \phi(s)| > \epsilon ) & \leq b_1 \exp(-b_2 \epsilon^2 n) ~,~ \mbox{for all } s, s' \in S,~ \epsilon > 0, \mbox{ and } n \in \N.
\end{align*}
Here $\P_{s'}(\cdot) \equiv \P(\cdot | \xi_1 = s')$ is the probability measure for the Markov chain $(\xi_n)$ started from state $s'$. \\
 
Fact 1 can be proved using the standard Chernoff-Hoeffding method for establishing large deviation bounds
of independent random variables. Fact 2 follows from Fact 1, since for a finite-state, irreducible Markov chain 
the return times to a given state are i.i.d. with exponential tails. 
 
\begin{proof}[Proof of Lemma \ref{lem:ConcentrationEstimate}] 
Throughout the proof we assume $\omega(x) = \lambda_0$, $x \geq 0$, for some
$\lambda_0 \in \Lambda_0 = \{(p,1),...,(p,L-1),(q,0)\}$. The result for general $\lambda \in \Lambda$
follows directly from this since, for any initial state $\lambda \in \Lambda$, the Markov chain 
$(\omega^j)_{j \in \N}$ collapses to the recurrent state set $\Lambda_0$ with probability 1 
after a single transition and the random variable $\Gamma_1$ has an exponential tail. 

The bounds for small $\epsilon$ and large $\epsilon$ are established separately. Specifically, we will show 
that there exist constants $c_1, c_2, \epsilon_0 > 0$ and other constants $c_1', c_2', \epsilon_0' > 0$ such that 
the empirical means $\Gammab_n \equiv \frac{1}{n} \sum_{j=1}^n \Gamma_j$ satisfy:
\begin{align}
\label{eq:GammabSmallEpsilon}
& \P(|\Gammab_n - 1| > \epsilon) \leq c_1 \exp(- c_2 \epsilon^2 n) ~,~ \mbox{for all } 0 < \epsilon \leq \epsilon_0 \mbox{ and } n \in \N, \\
\label{eq:GammabLargeEpsilon}
& \P(|\Gammab_n - 1| > \epsilon) \leq c_1' \exp(- c_2' \epsilon n) ~,~ \mbox{for all } \epsilon \geq \epsilon_0' \mbox{ and } n \in \N.
\end{align}
Together \eqref{eq:GammabSmallEpsilon} and \eqref{eq:GammabLargeEpsilon} show that 
\eqref{eq:UnEpsilonSquaredBound} and \eqref{eq:UnEpsilonBound} hold, with $\mu = 1$ and $N=1$, for some 
constants $C,c > 0$ depending on $c_1,c_2,c_1',c_2',\epsilon_0, \epsilon_0'$.  

For the proofs in both cases below we use the following notation for states $\lambda \in \Lambda_0$.

\begin{itemize}
\item $\psi(\lambda) \equiv \psi_{\lambda}$ is the stationary probability of state $\lambda$, as defined in Section 
\ref{subsec:StepDistributionZxChain}, and $\psi_n(\lambda) \equiv \frac{1}{n} \sum_{j=1}^n \indicator\{\omega^j = \lambda\}$
is the empirical frequency of state $\lambda$.
\item $\Gamma_j(\lambda) \equiv \Gamma_{\tau_j(\lambda)}$, where $\tau_j(\lambda)$ is the $j$-{th} visit time to state 
$\lambda$ for the Markov chain $(\omega^i)_{i \in \N}$:
$\tau_{j+1}(\lambda) = \inf \{i > \tau_j(\lambda) : \omega^i = \lambda\} ~\mbox{ with }~ \tau_0(\lambda) \equiv 0.$
\item $E(\lambda) \equiv \E(\Gamma_j(\lambda)) = \E(\Gamma_j|\omega^j = \lambda)$. 
\end{itemize}

\noindent
\emph{Proof of \eqref{eq:GammabSmallEpsilon}}:
For each $\lambda \in \Lambda_0$, $(\Gamma_j(\lambda))_{j \in \N}$ is a sequence of i.i.d. random variables
with mean $E(\lambda)$ and exponential tails. Thus, by Fact 1, there exist constants $b_1, b_2 > 0$ such that
for each $\lambda \in \Lambda_0$, 
\begin{align}
\label{eq:EachLambdaGammaBound}
\P(|\Gammab_n(\lambda) - E(\lambda)| > \epsilon) \leq b_1 \exp(-b_2 \epsilon^2 n) ~,~ \mbox{for all } 0 < \epsilon \leq 1, n \in \N.
\end{align}
Also, by Fact 2, there exists constants $b_3, b_4 > 0$ such that for each $\lambda \in \Lambda_0$, 
\begin{align}
\label{eq:EachLambdaPsiBound}
\P(|\psi_n(\lambda) - \psi(\lambda)| > \epsilon) \leq b_3 \exp(- b_4 \epsilon^2 n) ~,~ \mbox{for all } \epsilon > 0, n \in \N.
\end{align}
Finally, using nonnegativity of the sequence $(\Gamma_j(\lambda))_{j \in \N}$ one may show that, 
for any $0 < \epsilon \leq 1/3$ and $n \in \N$, the following holds:
\begin{align}
\label{eq:ConcentrationBetweenEndpoints}
& \mbox{ If } |\Gammab_{j_{\min}}(\lambda) - E(\lambda)| \leq \epsilon
\mbox{ and } |\Gammab_{j_{\max}}(\lambda) - E(\lambda)| \leq \epsilon, \nonumber \\
& \mbox{ then } |\Gammab_j(\lambda) - E(\lambda)| \leq \epsilon b_5,  
\mbox{ for all } n \psi(\lambda) (1-\epsilon) \leq j \leq n \psi(\lambda) (1+\epsilon),
\end{align}
where
\begin{align*} 
& j_{\min} = j_{\min}(n,\lambda, \epsilon) \equiv \ceil{n \psi(\lambda) (1 - \epsilon)}, \\
& j_{\max} = j_{\max}(n,\lambda, \epsilon) \equiv \max\{j_{\min}, \floor{n \psi(\lambda) (1 + \epsilon)} \}, \\
& b_5 \equiv \max_{\lambda \in \Lambda_0} \{3E(\lambda) + 2\}.  
\end{align*}

Now, define $G_{n,\epsilon}$ to be the ``good event'' that for each $\lambda \in \Lambda_0$ 
the following two conditions are satisfied:
\begin{enumerate}
\item $|\psi_n(\lambda) - \psi(\lambda)| \leq \epsilon \psi(\lambda)$.
\item $|\Gammab_j(\lambda) - E(\lambda)| \leq \epsilon b_5$, for all $n \psi(\lambda) (1-\epsilon) \leq j \leq n \psi(\lambda) (1+\epsilon)$. 
\end{enumerate}
By \eqref{eq:EachLambdaGammaBound} and \eqref{eq:ConcentrationBetweenEndpoints} together with the union bound, we have
\begin{align*}
\P\big( |\Gammab_j&(\lambda) - E(\lambda)| > \epsilon b_5, \mbox{ for some } n \psi(\lambda) (1-\epsilon) \leq j \leq n \psi(\lambda) (1+\epsilon) \big) \\
& \leq 2 b_1 \exp[-b_2 \epsilon^2 (n\psi(\lambda)(1-\epsilon))]
\leq 2 b_1 \exp[- (2/3) b_2 \psi(\lambda) \epsilon^2 n]
\end{align*}
for each $\lambda \in \Lambda_0$, $n \in \N$, and $\epsilon \leq 1/3$. Thus, by \eqref{eq:EachLambdaPsiBound} and the union bound,
\begin{align}
\label{eq:ProbGnepsComplimentBound}
\P(G_{n,\epsilon}^c) 
& \leq 2Lb_1 \exp(- (2/3) b_2 \psi_{\min} \epsilon^2 n) +
Lb_3 \exp(- b_4 \psi_{\min}^2 \epsilon^2 n) \nonumber \\
& \leq b_6 \exp(- b_7 \epsilon^2 n), 
\end{align}
for all $n \in \N$ and $\epsilon \leq 1/3$, where 
\begin{align*}
\psi_{\min} = \min_{\lambda \in \Lambda_0} \psi(\lambda),~
b_6 = 2 L b_1 + L b_3,~ \mbox{and } 
b_7 = \min\{ (2/3) b_2 \psi_{\min}, b_4 \psi_{\min}^2 \}.
\end{align*}

Since $\alpha = 1/2$, Lemma \ref{lem:muEqual1} implies 
$\sum_{\lambda \in \Lambda_0} \psi(\lambda) E(\lambda) = \left<\psi, E \right> = 1$. 
Thus, on the event $G_{n,\epsilon}$, $\epsilon \leq 1/3$, we have
\begin{align}
\label{eq:ConcentrationOnGneps}
|\Gammab_n - 1| 
& = \bigg| \sum_{\lambda \in \Lambda_0} \bigg( \sum_{j=1}^{n \psi_n(\lambda)} \frac{\Gamma_j(\lambda)}{n} - E(\lambda) \psi(\lambda) \bigg) \bigg| \nonumber \\
& \leq  \sum_{\lambda \in \Lambda_0} \bigg( \psi_n(\lambda) \bigg|  \sum_{j=1}^{n \psi_n(\lambda)} \frac{\Gamma_j(\lambda)}{n \psi_n(\lambda)} - E(\lambda) \bigg|
~+~ E(\lambda) \big| \psi_n(\lambda) - \psi(\lambda) \big| \bigg) \nonumber \\
& \leq \sum_{\lambda \in \Lambda_0} \Big( \psi_n(\lambda) \cdot \epsilon b_5 ~+~ E(\lambda) \cdot \epsilon \psi(\lambda) \Big) \nonumber \\
& \leq b_8 \epsilon, \mbox{ where } b_8 \equiv b_5 + \max_{\lambda \in \Lambda_0} E(\lambda).
\end{align}
Together \eqref{eq:ProbGnepsComplimentBound} and \eqref{eq:ConcentrationOnGneps} show that, for any 
$0 < \epsilon \leq 1/3$ and $n \in \N$,
\begin{align*}
\P(|\Gammab_n - 1| > b_8 \epsilon) \leq b_6 \exp(- b_7 \epsilon^2 n), 
\end{align*}
which is equivalent to \eqref{eq:GammabSmallEpsilon}, for $0 < \epsilon \leq \epsilon_0 \equiv b_8/3$, 
with $c_1 = b_6$ and $c_2 = b_7/b_8^2$. \\

\noindent
\emph{Proof of \eqref{eq:GammabLargeEpsilon}}:
Let $r = \max \{p, q\}$ and let $\xi$ be a geometric random variable with parameter $1-r$ started from $0$, i.e. $\P(\xi = k) = r^k (1-r)$, $k \geq 0$. 
Then, $S_0$ and $S_R$ are both stochastically dominated by $\xi$, so $\sum_{j = 1 }^n \Gamma_j$ is stochastically dominated by 
$\sum_{j = 1}^n \xi_j$, for each $n \in \N$, where $\xi_1, \xi_2,...$ are i.i.d. distributed as $\xi$. Further, by Fact 1, there exist constants 
$b_1,b_2 > 0$ such that for all $\epsilon \geq 1$ and $n \in \N$,
\begin{align*}
\P\Big(\xib_n - \frac{r}{1-r} > \epsilon \Big) = \P(\xib_n - \E(\xi) > \epsilon) \leq b_1 \exp( - b_2 \epsilon n).
\end{align*}

Now, since $\alpha = 1/2$, either $p$ or $q$ must be at least $1/2$, so $r/(1-r) \geq 1$. Thus, for 
$\epsilon \geq \epsilon_0' \equiv 2r/(1-r)$ we have 
\begin{align*}
\P(\Gammab_n - 1 > \epsilon) \leq \P(\xib_n > \epsilon) \leq \P\Big(\xib_n - \frac{r}{1-r} > \frac{\epsilon}{2} \Big) \leq b_1 \exp( - (b_2/2) \epsilon n).
\end{align*}
On the other hand, for all $\epsilon \geq \epsilon_0'$ we also have
\begin{align*}
\P(\Gammab_n - 1 < -\epsilon) = 0,
\end{align*}
since $\Gammab_n$ is nonnegative and $\epsilon_0' > 1$. Hence, \eqref{eq:GammabLargeEpsilon} holds with $c_1' = b_1$ and $c_2' = b_2/2$. 
\end{proof}

\bibliography{ref}

\end{document}